\documentclass[11pt]{article}
\usepackage{latexsym}
\usepackage{amsfonts}
\usepackage{amsmath}  
\usepackage{amscd}
\usepackage{theorem}

\usepackage[all]{xy}

\def\RR{\mathbb R}

\def\dim{\myop {dim}}

\def\ind{\myop {ind}}
\def\det{\myop {det}}

\def\ind{\myop {ind}}

\def\Z{\mathbb{Z}}

\def\R{\mathbb{R}}
\def\C{\mathbb{C}}

\def\CP{\mathbb{C} \mathbb{P}}

\def\<{\left\langle}
\def\>{\right\rangle}
\def\({\left(}
\def\){\right)}

\def\cA{\mathcal{A}}
\def\cB{\mathcal{B}}
\def\cC{\mathcal{C}}

\def\cF{\mathcal{F}}
\def\cG{\mathcal{G}}
\def\cH{\mathcal{H}}

\def\cL{\mathcal{L}}

\def\cM{\mathcal{M}}
\def\cN{\mathcal{N}}
\def\cO{\mathcal{O}}

\def\cT{\mathcal{T}}

\def\fg{\mathfrak{g}}

\def\ind{\operatorname{ind}}
\def\Ind{\operatorname{Ind}}

\def\rank{\operatorname{rank}}

\def\det{\operatorname{det}}
\def\dim{\operatorname{dim}}

\def\CP{\mathbb C \mathbb P}

\def\bea{\begin{eqnarray*}}
\def\eea{\end{eqnarray*}}

\theorembodyfont{\itshape}
\newtheorem{main}{Theorem}

\newtheorem{thm}{Theorem}
\newtheorem{prop}[thm]{Proposition}
\newtheorem{cor}[thm]{Corollary}
\newtheorem{lem}[thm]{Lemma}

\theorembodyfont{\rmfamily}
\newtheorem{defn}[thm]{Definition}
\newtheorem{rem}[thm]{Remark}
\newtheorem{conj}[thm]{Conjecture}

\newtheorem{problem}[thm]{Problem}

\newenvironment{proof}{\medskip \noindent
{\bf Proof.}}{\hfill \rule{.5em}{1em}
\\}

\begin{document}
\sloppy
\title{Stable Cohomotopy Seiberg-Witten Invariants of Connected Sums of Four-Manifolds with Positive First Betti Number}

\author{Masashi Ishida and Hirofumi Sasahira}

\date{}

\maketitle

\begin{abstract} 
We shall prove a new non-vanishing theorem for the stable cohomotopy Seiberg-Witten invariant \cite{b-f, b-1} of connected sums of 4-manifolds with positive first Betti number. The non-vanishing theorem enables us to find many new examples of 4-manifolds with non-trivial stable cohomotopy Seiberg-Witten invariants and it also gives a partial, but strong affirmative answer to a conjecture concerning non-vanishing of the invariant. Various new applications of the non-vanishing theorem are also given. For example, we shall introduce variants $\overline{\lambda}_k$ of Perelman's $\overline{\lambda}$ invariants for real numbers $k$ and compute the values for a large class of 4-manifolds including connected sums of certain K{\"{a}}hler surfaces. The non-vanishing theorem is also used to construct the first examples of 4-manifolds with non-zero simplicial volume and satisfying the strict Gromov-Hitchin-Thorpe inequality, but admitting infinitely many distinct smooth structures for which no compatible Einstein metric exists. Moreover, we are able to prove a new result on the existence of exotic smooth structures. 
\end{abstract}


\tableofcontents

\section{Introduction and the main results}\label{intro}

\subsection{Bauer's non-vanishing theorem}

A new point of view in Seiberg-Witten theory \cite{w} was introduced in \cite{furuta}. Furuta \cite{furuta} described the Seiberg-Witten monopole equations as an $S^1$-equivariant map between two Hilbert bundles on the $b_{1}(X)$-dimensional Picard torus $Pic^0(X)$ of a 4-manifold $X$. The Seiberg-Witten moduli space can be seen as the quotient of the zero locus of the map by the action of $S^1$. The map is called the Seiberg-Witten map (or the monopole map), which is denoted in the present article by
\begin{eqnarray*}
\mu : \cA \longrightarrow \cC. 
\end{eqnarray*}
By introducing a new technique which is so called finite dimensional approximation of the Seiberg-Witten map $\mu$, the $10/8$-theorem which gives a strong constraint on the intersection form of spin $4$-manifolds was proved in \cite{furuta}. By developing this idea, Bauer and Furuta \cite{b-f} showed that the finite dimensional approximation of the Seiberg-Witten map $\mu$ gives rise to an $S^1$-equivariant stable cohomotopy class and that the stable cohomotopy class is a differential topological invariant of $X$. The invariant takes its value in a certain complicated equivariant stable cohomotopy group ${\pi}^{{b}^+}_{S^1, \cB}(Pic^0(X), ind D)$, where ${b}^+:={b}^+(X)$ denotes the dimension of the maximal positive definite linear subspace in the second cohomology of $X$ and $ind D$ is the virtual index bundle for the Dirac operators parametrized by $Pic^0(X)$. In this article, let us call it {\it stable cohomotopy Seiberg-Witten invariant} and denote it by $BF_{X}$:
\begin{eqnarray*}\label{b-f-inv}
BF_{X}(\Gamma_{X}) \in {\pi}^{{b}^+}_{S^1, \cB}(Pic^0(X), ind D), 
\end{eqnarray*}
where $\Gamma_{X}$ is a spin$^c$ structure on $X$. As one of remarkable results, Bauer \cite{b-1} proved the following non-vanishing theorem of stable cohomotopy Seiberg-Witten invariant:

\begin{thm}[\cite{b-1}]\label{bau-1}
Let $X:=\#^{n}_{i=1} {X}_{i}$ be a connected sum of $n \geq 2$ of almost complex 4-manifolds $X_{i}$ with $b_{1}(X_{i})=0$. The stable cohomotopy Seiberg-Witten invariants of $X$ do not vanish if the following conditions are satisfied: 
\begin{itemize}
\item Each summand $X_{i}$ satisfies ${b}^{+}(X_{i}) \equiv 3 \ (\bmod \ 4)$.
\item Each summand $X_{i}$ satisfies $SW_{X_{i}}(\Gamma_{X_{i}}) \equiv 1 \ (\bmod \ 2)$, where $\Gamma_{X_{i}}$ is a spin${}^c$ structure compatible with the almost complex structure. 
\item If $n \geq 4$, then $n=4$ and ${b}^{+}(X) \equiv 4 \ (\bmod \ 8)$.
\end{itemize} 
\end{thm}
Since it is known that the integer valued Seiberg-Witten invariant $SW_{X}$ introduced by Witten \cite{w} vanishes for connected sums of any 4-manifolds with ${b}^{+} \geq 1$, this theorem tells us that $BF_{X}$ is strictly stronger than $SW_{X}$. Essential in computing the value of the invariant for connected sums was Bauer's gluing theorem, which states that the invariant of a connected sum is just equal to the smash product of the invariants for each of the two pieces. \par 
Theorem \ref{bau-1} naturally leads us to ask if a similar non-vanishing theorem still holds in the case where $b_{1}(X_{i}) \not=0$. It turns out that this problem is extremely difficult. One of the reason for the difficulty is that the structure of the $S^1$-equivariant stable cohomotopy group $\pi^{b^+}_{S^1, \cB}(Pic^0(X),\ind D)$ is quite complicated. In the case where $b_{1}(X_{i})=0$, the equivariant stable cohomotopy group is isomorphic to the non-equivariant stable cohomotopy group of the complex projective space (\cite{b-1}). In this direction, there are only two works, which are due to Furuta-Kametani-Matsue-Minami \cite{furuta-k-m-2} and Xu \cite{Xu}. It was shown in \cite{furuta-k-m-2} that, under a certain condition, there are closed spin 4-manifolds with $b_{1}=4$, $b_{2}=50$, $\tau=-32$ and non-trivial stable cohomotopy Seiberg-Witten invariants, where $\tau$ is the signature of the manifolds. On the other hand, Xu \cite{Xu} could prove non-triviality of the stable cohomotopy Seiberg-Witten invariant for the connected sum of type $X\# T^4$, where $X$ is an almost complex 4-manifold with $b_{1}(X)=0$ and ${b}^{+}(X) \equiv 3 \ (\bmod \ 4)$ or 4-dimensional torus $T^4$. Based on Theorem \ref{bau-1} and this result of Xu, it is so natural to make the following conjecture as was already proposed in \cite{Xu}:
\begin{conj}[Conjecture 24 in \cite{Xu}]\label{conj-1}
For $i=1,2,3$, let $X_i$ be 
\begin{itemize}
\item a 4-torus $T^4$, or 
\item a closed oriented almost complex 4-manifold with ${b}_{1}({X}_{i})=0$, ${b}^{+}({X}_{i}) \equiv 3 \ (\bmod \ 4)$ and $SW_{X_{i}}(\Gamma_{X_{i}}) \equiv 1 \ (\bmod \ 2)$, where $\Gamma_{X_{i}}$ is a spin${}^c$ structure compatible with the almost complex structure. 
\end{itemize}
Then the connected sum $(\#^{\ell}_{i=1}{X}_{i}) \# {T}^{4}$ has a non-trivial stable cohomotopy Seiberg-Witten invariant, where $\ell=2,3$. 
\end{conj} 

To the best of our knowledge, this conjecture still remains open. However, in the present article, we shall prove that Conjecture \ref{conj-1} in the case where $\ell =2$ is true. See Corollary \ref{conje-cor} below.

\subsection{The main results}

The main purpose of the present article is divided into twofold. The first one is to establish new non-vanishing theorems of stable cohomotopy Seiberg-Witten invariant of connected sums of 4-manifolds ${X}_{i}$ whose first Betti number $b_{1}(X_{i})$ dose not necessarily vanish. See Theorem \ref{main-A}, Theorem \ref{main-B} and Corollary \ref{conje-cor} stated below. The second one is to give several applications of Theorem \ref{main-A} to both topology and differential geometry in dimension four. See Theorems \ref{thm decomp}, \ref{thm exotic-B}, \ref{main-CCC} and \ref{main-CC} below.  \par
To state the first main theorems, let us introduce the following definition: 
\begin{defn}\label{def-1}
Let $X$ be a closed oriented smooth 4-manifold with ${b}^{+}(X)>1$. Let $\Gamma_{X}$ be a spin${}^{c}$ structure on $X$. Let ${c}_{1}({\cal L}_{{\Gamma}_{X}})$ be the first Chern class of the complex line bundle ${\cal L}_{{\Gamma}_{X}}$ associated with $\Gamma_{X}$. Finally, let ${\frak e}_{1}, {\frak e}_{2}, \cdots, {\frak e}_{s}$ be a set of generators of ${H}^{1}(X, {\mathbb Z})$, where $s={b}_{1}(X)$. Then, define 
\begin{eqnarray*}\label{spin-01}
\frak{S}^{ij}(\Gamma_X):=\frac{1}{2}\< c_1(\cL_{\Gamma_X}) \cup \frak{e}_i \cup \frak{e}_j, [X] \>, 
\end{eqnarray*}
where $[X]$ is the fundamental class of $X_i$ and $\<\cdot, \cdot \>$ is the pairing between cohomology and homology. 
\end{defn}

Notice that the image of $c_1( \cL_{\Gamma_X})$ under the natural map from $H^2(X, \Z)$ to $H^2(X, \Z_2)$ is equal to $w_2(X)$. This implies 
\[
\< c_1( \cL_{ \Gamma_{X} }) \cup \frak{e}_i \cup \frak{e}_j, [X] \> 
\equiv 
\< \frak{e}_i \cup \frak{e}_j \cup \frak{e}_i \cup \frak{e}_j, [X] \> 
\equiv
0 \bmod 2.
\]
Hence the $\frak{S}^{ij}(\Gamma_X)$ are integers. \par
The first main theorem of this article can be stated as follows: 
\begin{main}\label{main-A}
For $m=1,2,3$, let $X_m$ be a closed oriented almost complex 4-manifold with ${b}^{+}(X_m)>1$ and satisfying
\begin{eqnarray}\label{spin-0}
{b}^{+}(X_{m})-{b}_{1}(X_{m}) \equiv 3 \ (\bmod \ 4).
\end{eqnarray}
Let $\Gamma_{X_{m}}$ be a spin${}^{c}$ structure on $X_m$ which is induced by the almost complex structure and assume that $SW_{X_{m}}(\Gamma_{X_{m}}) \equiv 1 \ (\bmod \ 2)$. Under Definition \ref{def-1}, moreover assume that the following condition holds for each $m$:
\begin{eqnarray}\label{spin-014}
\frak{S}^{ij}(\Gamma_{X_{m}}) \equiv 0 \ (\bmod \ 2) & \text{for all $i, j$}.
\end{eqnarray}
Then the connected sum $\#^{n}_{m=1}{X}_{m}$ has a non-trivial stable cohomotopy Seiberg-Witten invariant, where $n=2,3$. 
\end{main} 

Suppose now that ${b}_{1}(X_{m})=0$ holds. Then the condition (\ref{spin-0}) is nothing but ${b}^{+}(X_{m})\equiv 3 \ (\bmod \ 4)$, and moreover the condition (\ref{spin-014}) holds trivially by the very definition. Hence, we are able to recover the non-vanishing theorem of Bauer in the case where $\#^{n}_{m=1}{X}_{m}$ and $n=2,3$. In this sense, Theorem \ref{main-A} can be seen as a natural generalization of Theorem \ref{bau-1} except the case $n=4$. Let us here emphasize that the strategy of the proof of Theorem \ref{main-A} is different from that of Bauer's proof of Theorem \ref{bau-1}. In \cite{sasa}, the second author introduced a new differential topological invariant extracted from the finite dimensional approximation of the Seiberg-Witten map. In this article, we shall call the new invariant  {\it spin cobordism Seiberg-Witten invariant} and denote it by $SW^{spin}_{X}$. The spin cobordism Seiberg-Witten invariant takes its value in the spin cobordism group: 
\begin{eqnarray*}\label{b-f-inv}
SW^{spin}_{X}(\Gamma_{X}) \in \Omega_{d}^{spin}, 
\end{eqnarray*}
where $d$ is the dimension of the Seiberg-Witten monopole moduli space associated with the spin$^{c}$ structure $\Gamma_{X}$. This invariant plays a crucial role in the course of the proof of Theorem \ref{main-A}. Indeed, we shall prove that the connected sums $\#^{n}_{m=1}{X}_{m}$ in Theorem \ref{main-A} have non-trivial spin cobordism Seiberg-Witten invariants. By showing that the non-triviality of $SW^{spin}_{X}$ implies the non-triviality of $BF_{X}$, we shall prove Theorem \ref{main-A}. This is totally different from the previous works of Bauer \cite{b-1} and Xu \cite{Xu}. \par 
In particular, Theorem \ref{main-A} implies 
\begin{main}\label{main-B}
For $m=1,2,3$, let $X_m$ be 
\begin{itemize}
\item a closed oriented almost complex 4-manifold with ${b}_{1}({X}_{m})=0$, ${b}^{+}({X}_{m}) \equiv 3 \ (\bmod \ 4)$ and $SW_{X_{m}}(\Gamma_{X_{m}}) \equiv 1 \ (\bmod \ 2)$, where $\Gamma_{X_{m}}$ is a spin${}^c$ structure compatible with the almost complex structure, or 
\item a closed oriented almost complex 4-manifold with ${b}^{+}(X_m)>1$, $c_{1}(X_{m}) \equiv 0 \ (\bmod \ 4)$ (i.e. the image of $c_1(X_m)$ under $H^2(X_m, \Z) \rightarrow H^2(X_m, \Z_4)$ is trivial)  and $SW_{X_{m}}(\Gamma_{X_{m}}) \equiv 1 \ (\bmod \ 2)$, where $\Gamma_{X_{m}}$ is a spin${}^c$ structure compatible with the almost complex structure. 
\end{itemize}
Then a connected sum $\#^{n}_{m=1}{X}_{m}$, where $n=2,3$, has a non-trivial stable cohomotopy Seiberg-Witten invariant.
\end{main} 

The construction of Gompf \cite{gom} provides us examples of the first case in this theorem. In fact, there are infinitely many examples satisfying such properties. On the other hand, we shall prove that products $\Sigma_{g} \times \Sigma_{h}$ of oriented closed surfaces of odd genus $g, h \geq 1$ and primary Kodaira surfaces are non-trivial examples of the second case in Theorem \ref{main-B}. A primary Kodaira surface is a compact complex surface $X$ with $b^+(X)=2$ and $b_{1}(X)=3$, which admits a holomorphic locally trivial fibration over an elliptic curve with an elliptic curve as typical fiber (cf. \cite{BPV, kodaira, thurs, fgg, geige}). This is a non-K{\"{a}}hler, symplectic, spin complex surface. Theorem \ref{main-B} provides us new examples of 4-manifolds with non-trivial stable cohomotopy Seiberg-Witten invariants. For example, by considering a 3-fold connected sum of primary Kodaira surface with itself, we are able to obtain the first example of non-symplectic, spin 4-manifold with odd first Betti number and non-trivial stable cohomotopy Seiberg-Witten invariants. Moreover, by considering a product $\Sigma_{1} \times \Sigma_{1}$ of oriented closed surfaces of genus one, we particularly obtain the following corollary of Theorem \ref{main-B}:
\begin{cor}\label{conje-cor} 
Conjecture \ref{conj-1} in the case where $\ell=2$ is true.
\end{cor}

We are able to give various new applications of these non-vanishing theorems to both differential topology and differential geometry of connected sums of 4-manifolds as follows. Taubes's non-vanishing theorem \cite{t-1} for Seiberg-Witten invariants of symplectic 4-manifolds implies that closed symplectic 4-manifolds $X$ has no decompositions of the form $Y_1 \# Y_2$ with $b^+(Y_1)>0$, $b^+(Y_2) > 0$. It is natural to ask whether connected sums $\#_{m=1}^n X_m$ of symplectic 4-manifolds $X_m$ have no decompositions of the form $\#_{m = 1}^N Y_m$ with $b^+(Y_m) > 0$, $N > n$. However, there are counter examples of this question. In fact, it is known that connected sums $X \# \CP^2$ of any simply connected, elliptic surfaces $X$ and $\CP^2$ is diffeomorphic to $p  \CP^2  \#  q \overline{\CP}^2$ for some $p, q \geq 0$. (See \cite{moi}.) This fact gives us infinitely many counter examples of the question. However, Theorem \ref{main-B} implies that if $X_m$ satisfies the topological condition $c_1(X_m) \equiv 0 ( \bmod 4)$ and if $n \leq 3$ then $\#^n_{m=1} X_m$ has no such decompositions as follows: 
\begin{main} \label{thm decomp}
Let $X_m$ be closed symplectic 4-manifolds with $c_1(X_m) \equiv 0 \ (\bmod \ 4)$ for $m=1, 2, 3$, and $X$ be a connected sum $\#_{m=1}^n X_m$, where $n=2, 3$. Then $X$ can not be written as a connected sum $\#_{m=1}^{N} Y_{m}$ with $b^+(Y_m) >0 $ and with $N>n$.
\end{main}

The next application concerns exotic smooth structures on some connected sums of 4-manifolds. We are able to deduce from Taubes's theorem \cite{t-1} that all simply connected, non-spin symplectic 4-manifolds $X$ have exotic smooth structures. We shall prove that if $b^+(X) \equiv 3 \bmod 4$ then connected sums $X \# X'$ also have exotic smooth structures for some $X'$ by using Theorem \ref{main-A}:
\begin{main} \label{thm exotic-B}
Let $X$ be any closed, simply connected, non-spin, symplectic 4-manifold with $b^+ \equiv 3 \ (\bmod \ 4)$. For $m=1, 2$, let $X_m$ be almost complex 4-manifolds satisfying the conditions in Theorem \ref{main-A}, Then for $n=1, 2$ connected sums $X \# \Big(\#_{m=1}^n X_m \Big) $ always admit at least one exotic smooth structure.  
\end{main}

On the other hand, as was already studied by LeBrun with the first author \cite{ishi-leb-1, ishi-leb-2}, the non-vanishing theorem like Theorem \ref{main-A} has various powerful differential geometric applications. For example, all the major results in \cite{ishi-leb-2} on the connected sum of three 4-manifolds with ${b}_{1}=0$ can be generalized to the case of ${b}_{1}>0$ by virtue of Theorem \ref{main-A}. By combining Theorem \ref{main-A} with the technique developed in \cite{ishi-leb-2}, we are able to prove some new interesting results, i.e., Theorems \ref{main-CCC} and \ref{main-CC} below. \par
In his celebrated work \cite{p-1, p-2, p-3} on Ricci flow, Perelman introduced a functional which is so called ${\mathcal F}$-functional and also introduced an invariant of any closed manifold with arbitrary dimension, which is called $\bar{\lambda}$ invariant. The $\bar{\lambda}$ invariant is arisen naturally from the ${\mathcal F}$-functional. For 3-manifolds which dose not admit positive scalar curvature metrics, Perelman computed its value. See Section 8 in \cite{p-2}. Recently, Li \cite{li} introduces a family of functionals of the type of ${\mathcal F}$-functional with monotonicity property under Ricci flow (see also \cite{o-s-w, cao-X, cao-X-1, li-1}). Inspired by these works of Perelman and Li, we shall introduce, for any real number $k \in {\mathbb R}$, an invariant $\bar{\lambda}_{k}$ of any closed manifold with arbitrary dimension. We shall call it $\bar{\lambda}_{k}$ invariant. In particular, $\bar{\lambda}_{k}$ invariant includes Perelman's $\bar{\lambda}$ invariant as a special case. Indeed, $\bar{\lambda}_{1}=\bar{\lambda}$ holds. In dimension 4, we are able to prove the following result by using Theorem \ref{main-A}:  
\begin{main}\label{main-CCC}
Let $X_m$ be as in Theorem \ref{main-A} and assume moreover that $X_m$ is a minimal K{\"{a}}hler surface. Let $N$ be a closed oriented smooth 4-manifold with $b^{+}(N)=0$ with a Riemannian metric of non-negative scalar curvature. Then, for $n=2,3$ and any real number $k \geq \frac{2}{3}$, $\bar{\lambda}_{k}$ invariant of a connected sum $M:=(\#^{n}_{m=1}{X}_{m}) \# N$ is given by 
\begin{eqnarray*}
\bar{\lambda}_{k}(M)={-4k{\pi}}\sqrt{2\sum^n_{m=1}c^2_{1}(X_{m})}.
\end{eqnarray*}
\end{main}
Here notice that minimality of $X_{m}$ forces that $c^2_{1}(X_{m})=2\chi(X_{m}) + 3 \tau(X_{m}) \geq 0$, where $\chi(X_{m})$ and $\tau(X_{m})$ denote respectively Euler characteristic and signature of $X_m$.  \par 
On the other hand, to state one more interesting application of Theorem \ref{main-A}, we need to recall the definition of Einstein metrics. Recall that any Riemannian metric $X$ is called Einstein if its Ricci curvature, considered as a function on the unit tangent bundle, is constant. Not every closed manifolds can admit Einstein metrics. In fact, it is known that any closed Einstein 4-manifold $X$ must satisfy 
\begin{eqnarray}\label{ht-int}
2\chi(X) \geq 3|\tau(X)|. 
\end{eqnarray}
This inequality is called the Hitchin-Thorpe inequality \cite{hit, thor}. In particular, Hitchin \cite{hit} investigates the boundary case of the inequality and describes what happens in the case. Indeed, Hitchin proved that any closed oriented Einstein 4-manifold satisfying $2\chi(X) = 3|\tau(X)|$ is finitely covered by either $K3$ surface or the 4-torus. Hence, almost all of 4-manifolds satisfying $2\chi(X) = 3|\tau(X)|$ cannot admit any Einstein metric. We shall call $2\chi(X) > 3|\tau(X)|$ the strict Hitchin-Thorpe inequality. By using the original Seiberg-Witten invariant $SW_{X}$, LeBrun \cite{leb-44} constructed the first example of simply connected closed 4-manifold without Einstein metric, but nonetheless satisfies strict Hitchin-Thorpe inequality. On the other hand, in non-simply connected case, Gromov \cite{gromov} constructed the first example of non-simply connected 4-manifold satisfying the same properties. To give such examples, Gromov \cite{gromov} proved a new obstruction to the existence of Einstein metric by using the simplicial volume. More precisely, Gromov proved that any closed Einstein 4-manifold $X$ must satisfy
\begin{eqnarray}\label{gr-int}
\chi(X) \geq \frac{1}{2592{\pi}^2}||X||,  
\end{eqnarray}
where $||X||$ is the simplicial volume of $X$. Here, we notice that any simply connected manifold has vanishing simplicial volume. It is also known that there is an inequality which interpolates between (\ref{ht-int}) and (\ref{gr-int}). In this article, we shall call it the Gromov-Hitchin-Thorpe inequality \cite{kot-g}: 
\begin{eqnarray}\label{k-GHT}
2\chi(X) - 3|\tau(X)| \geq \frac{1}{81{\pi}^2}||X||. 
\end{eqnarray}
The strict case is called the strict Gromov-Hitchin-Thorpe inequality. To the best of our knowledge, the following problem still remains open:
\begin{problem}\label{pro-ein}
Does there exist a closed topological spin or non-spin 4-manifold $X$ satisfying the following three properties?
\begin{itemize}
\item $X$ has non-trivial simplicial volume, i.e., $||X|| \not=0$.
\item $X$ satisfies the strict Gromov-Hitchin-Thorpe inequality. 
\item $X$ admits infinitely many distinct smooth structures for which no compatible Einstein metric exists.
\end{itemize}
\end{problem}
In the case where $X$ is simply connected, i.e., $||X|| =0$, this problem was already solved affirmatively. In particular, Problem \ref{pro-ein} can be seen as a natural generalization of Question 3 considered in \cite{ishi-leb-1}. See also Introduction of \cite{ishi-leb-1}. As a nice application of Theorem \ref{main-A}, we are able to give an affirmative answer to Problem \ref{pro-ein} as follows:
\begin{main}\label{main-CC}
There exist infinitely many closed topological 4-manifolds and each of these 4-manifolds satisfies three properties in Problem \ref{pro-ein}. 
\end{main}
We shall prove this theorem by constructing both spin and non-spin examples. To the best of our knowledge, this result provides us the first example with such properties (see also Remark \ref{final-rem} in subsection \ref{sub-44} below). At least, the present authors do not know how to prove Theorem \ref{main-CC} without Theorem \ref{main-A}. To prove Theorem \ref{main-CC}, we shall also prove a new obstruction to the existence of Einstein metrics by using Theorem \ref{main-A}. See Theorem \ref{einstein} stated in subsection \ref{sub-44} below. \par
The organization  of this article is as follows. In Section 2, we shall discuss differential topological invariants arising from finite dimension approximations of the Seiberg-Witten map $\mu$. In particular, in subsections \ref{sasa-spin-1} and \ref{sasa-spin}, we shall recall the detail of the construction of $SW^{spin}_{X}$ for the reader who is unfamiliar to this subject. We shall also remark that the non-triviality of $SW^{spin}_{X}$ implies the non-triviality of $BF_{X}$. See Proposition \ref{spin-BF} below. In subsection \ref{diag}, we shall discuss a relationship among $BF_{X}$, $SW^{spin}_{X}$ and $SW_{X}$. Interestingly, by introducing a refinement $\widehat{SW}^{spin}_X$ of $SW^{spin}_{X}$, we shall prove that there is a natural commutative diagram among  $BF_{X}$, $\widehat{SW}^{spin}_X$ and $SW_{X}$. See Theorem \ref{prop diagram} below. Though this fact is not used in this article, it is clear that this has its own interest (see also Section \ref{final} below). \par
In Section \ref{Sec-3}, we shall give proofs of both Theorem \ref{main-A} and Theorem \ref{main-B}. In Theorem \ref{thm-spin-non-v} stated in subsection \ref{sub-31} below, we shall prove that the connected sums in Theorem \ref{main-A} have non-trivial spin cobordism Seiberg-Witten invariants. This result is a real key to prove Theorem \ref{main-A} and, in fact, this immediately implies Theorem \ref{main-A}. See subsection \ref{sub-32} below. \par
In Section \ref{sec-4}, we shall discuss several geometric applications of Theorem \ref{main-A}. 
In subsection \ref{subsec-exotic}, we shall prove both Theorem \ref{thm decomp} and Theorem \ref{thm exotic-B}. In subsection \ref{subsec-adj}, we shall remark that Theorem \ref{main-A} implies an adjunction inequality which gives a bound of the genus of embedded surfaces in the connected sum (cf. \cite{K-M, furuta-k-m-1}). In subsection \ref{4.22}, first of all, we shall recall two important differential geometric inequalities \cite{leb-11, leb-17} arising from Seiberg-Witten monopole equations. By combining these differential geometric inequalities with Theorem \ref{main-A}, we shall prove a key result to prove Theorems \ref{main-CCC} and \ref{main-CC}. See Theorem \ref{mono-key-bounds} in subsection \ref{4.22}. By using the result, we shall prove Theorem \ref{main-CCC} in subsection \ref{sub-43}, and also prove Theorem \ref{main-CC} in subsection \ref{sub-44}. \par
Finally, we shall close this article with some remarks in Section \ref{final}.

\section{Finite dimensional approximation of the Seiberg-Witten map and differential topological invariants}\label{Sa}

In this section, we shall recall mainly the detail of the construction of $SW^{spin}_{X}$ and discuss a relationship among $BF_{X}$, $SW^{spin}_{X}$ and $SW_{X}$. See also \cite{b-f, b-1, b, mor, nico, sasa} for more detail. 

\subsection{The Seiberg-Witten map}\label{sub-21}

Let $X$ be a closed oriented 4-manifold with a Riemannian metric $g$. Take a spin$^c$ structure $\Gamma_X$ on $X$. Then we have the spinor bundles $S_{\Gamma_X}^\pm$ and its determinant line bundle $\cL_{\Gamma_X}$. Let $\Gamma(S_{\Gamma}^+)$ be the space of sections of $S_{\Gamma_X}^+$ and $\cA(\cL_{\Gamma_X})$ be the space of $U(1)$-connections on $\cL_{\Gamma_X}$. The Seiberg-Witten monopole equations \cite{w} perturbed by a self dual $2$-form $\eta \in \Omega^+(X)$ are the following equations for pairs $(\phi,A) \in \Gamma(S_{\Gamma_X}^+) \times \cA(\cL_{\Gamma_X})$: 
\[
D_A \phi = 0, \ F_A^+ = q(\phi) + \eta.
\]
Here $D_A:\Gamma(S_{\Gamma_X}^+) \rightarrow \Gamma(S_{\Gamma_X}^-)$ is the 
twisted Dirac operator associated with the connection $A$ and Levi-Civita connection of the Riemannian metric $g$. $F_A^+$ denotes the self-dual part of the curvature $F_{A}$, and $q(\phi)$ is the trace-free part of the endomorphism $\phi \otimes \phi^*$ of $S_{\Gamma_X}^+$. Notice that this endomorphism is identified with an imaginary-valued self-dual 2-form via the Clifford multiplication. \par
Let $\cG:=C^{\infty}(X,S^1)$ be the group of smooth maps from $X$ to $S^1$, then $\cG$ acts on $\Gamma(S_{\Gamma_X}^+) \times \cA(\Gamma_X)$ in the 
following way: 
\[
\gamma (\phi,A):=(\gamma \phi, A - 2\gamma^{-1}d\gamma), 
\]
where $\gamma \in \cG$ and $(A,\phi) \in \Gamma(S_{\Gamma_X}^+) \times \cA(\cL_{\Gamma_X})$. This action preserves the space ${\cal S}_{\Gamma_X}(g,\eta)$ of solutions of the perturbed Seiberg-Witten equations. The quotient space 
\begin{eqnarray}\label{moduli}
\cM^{SW}= \cM_{\Gamma_X}^{SW}(g, \eta) := {\cal S}_{\Gamma_X}(g,\eta) / \cG
\end{eqnarray}
is called the Seiberg-Witten moduli space. More precisely, we need to take the completions of $\Gamma(S_{\Gamma_{X}}^{\pm})$, $\cA$ and $\cG$ in some Sobolev norms. However we do not mention Sobolev norms since it is not important in the present article. See \cite{mor, nico} for more detail. 

The Seiberg-Witten moduli space (\ref{moduli}) can be considered as the zero locus of a certain map between two Hilbert bundles over a torus. This map is called the Seiberg-Witten map.  The Seiberg-Witten map has more information than the moduli space. In fact, we are able to deduce more powerful results by studying the Seiberg-Witten map. This point of view is introduced by Furuta \cite{furuta}. Below, we shall introduce the Seiberg-Witten map following \cite{b-f}. \par 
Fix a connection $A_0 \in \cA(\cL_{\Gamma_X})$ and a point $x_0 \in X$.
Let $\cG_0$ be the subgroup of $\cG$ which consists of maps whose the value 
at $x_0$ is $1$. Then we put 
\begin{eqnarray}\label{cT}
{\cT}:=(A_0+\ker d)/\cG_0,
\end{eqnarray}
where $d:\Omega_X^1 \rightarrow \Omega_X^2$ is the derivative and the action 
of $\gamma \in \cG_0$ is given by 
\[
\gamma(A)=A-2\gamma^{-1}d\gamma
\]
for $A \in A_0+\ker d$. One can see that the space $\cT$ is isomorphic to the $b_1(X)$-dimensional torus $Pic^0(X)=H^1(X,\R)/H^1(X,\Z)$. Let us define the following two Hilbert bundles $\cA$ and $\cC$ over $\cT$ as follows: 
\[
\begin{split}
\cA &=
\Big( A_0+\ker d \Big) \times_{\cG_0}
\Big( \Gamma(S_{\Gamma_X}^+) \oplus \Omega_X^1 \Big), \\
\cC &=
\Big( A_0+\ker d \Big) \times_{\cG_0}
\Big( \Gamma (S_{\Gamma_X}^-) \oplus \Omega_X^+ \oplus \cH_g^1(X) \oplus 
\Omega_X^0/\R \Big).
\end{split}
\]
Here $\cH_g^1(X)$ is the space of harmonic $1$-forms on $X$ and $\R$ represents the space of constant functions on $X$. The actions of $\cG_0$ on $\Gamma(S_{\Gamma_X}^+)$ and $\Gamma(S_{\Gamma_X}^-)$ are the scalar products and the actions on $\Omega_X^1$, $\Omega_X^+$, $\cH_g^1(X)$ and $\Omega_X^0/\R$ are trivial. Of course, there are natural actions  of $S^1$ on $\cA$ and $\cC$ induced by the scalar products on $S_{\Gamma_X}^{\pm}$. We are now in a position to introduce
\begin{defn}[\cite{b-f, furuta}]\label{sw-map}
The Seiberg-Witten map associated with the Riemannian metric $g$ and spin$^c$ structure $\Gamma_X$ on $X$ is an $S^1$-equivariant map 
\begin{eqnarray*}
\mu :\cA \longrightarrow \cC 
\end{eqnarray*}
which is defined by 
\[
\mu(A,\phi,a):=(A, D_{A+a}\phi, F_{A+a}^+ - q(\phi)+\eta, p(a), d^*a).
\]
Here $A \in A_0+\ker d$, $\phi \in \Gamma(S^+_{\Gamma_X})$, $a \in 
\Omega_X^1$ and $p$ is the $L^2$-orthogonal projection onto $\cH_g^1(X)$.
\end{defn}
One can check that the quotient $\mu^{-1}(0)/S^1$ is nothing but the Seiberg-Witten moduli space (\ref{moduli}). \par

\subsection{Finite dimensional approximation and $BF_{X}$}\label{sub-21-1}

Let $\mu : \cA \longrightarrow \cC $ be the Seiberg-Witten map in the sense of Definition \ref{sw-map}. We denote the linear part by $l$ of $\mu$. Namely, we have
\begin{equation} \label{linear part}
l(A,\phi,a) = (A, D_A \phi, d^+ a, p(a),d^*a).
\end{equation}
By a Kuiper's theorem \cite{kui}, we have a global trivialization of $\cC$. We fix a global trivialization $\cC \cong \cT \times \cB$, where $\cB$ is a Hilbert space. For any finite dimensional subspace $W \subset \cB$, let us consider the orthogonal projection $pr_{W} : \cT \times \cB \rightarrow W$ and set $\cF(W):=l^{-1}(W)$. Then we define a map from $\cF(W)$ to $W$ by
\begin{equation*}
f_W:=pr_W \circ \mu|_{\cF(W)}.
\end{equation*}
Let $W^+$ and $\cF(W)^+$ be respectively the one point compactifications of $W$ and $\cF(W)$. Under these notations, one of crucial results due to Bauer and Furuta \cite{b-f} can be stated as follows: 
\begin{thm} [\cite{b-f}]\label{finte-SW-1}
There exist finite dimensional subspaces $W$ in $\cB$ satisfying the following properties:
\begin{enumerate}
\item
The finite dimensional subspace $W \subset \cB$ and the image $Im(l)$ of the linear part $l$ span $\cC \cong \cT \times \cB$, i.e., $W + Im(l)  = \cC$. 

\item
For each finite dimensional subspace $W'\subset \cB$ satisfying $W \subset W'$, the map $f_{W'}:\cF({W'}) \rightarrow {W'}$ induces an $S^1$-equivariant map $f_{W'}^+:\cF({W'})^+ \rightarrow {(W')}^+$ between one point compactifications. And the following diagram
\[
\begin{CD}
\cF(W')^+  @>{f_{W'}^+}>>              (W')^+ \\
@|                             @| \\
\cF(W) \oplus \cF(U)^+ @>>{(f_W \oplus pr_U \circ l|_{\cF(U)})^+}> (W \oplus 
U)^+
\end{CD}
\]
is $S^1$-equivariant homotopy commutative as pointed maps, where $U$ is the 
orthogonal complement of $W$ in $W'$.
\end{enumerate}
\end{thm}

Based on this theorem, let us introduce the following definition:
\begin{defn}\label{finite-sw}
The map $f_W:\cF(W) \rightarrow W$ is called a finite dimensional approximation of the Seiberg-Witten map $\mu$ if the subspace $W \subset \cB$ 
 is a finite dimensional subspace satisfying two conditions in  Theorem \ref{finte-SW-1}.
\end{defn}

Now, let $\mu :\cA \longrightarrow \cC$ be the Seiberg-Witten map associated with the Riemannian metric $g$ and spin$^c$ structure $\Gamma_X$ on a 4-manifold $X$ with $b^+:={b}^+(X)>1$. Let $f_W:\cF(W) \rightarrow W$ be a finite dimensional approximation of $\mu$ and consider the $S^1$-equivariant map $f_{W}^+:\cF({W})^+ \rightarrow {(W)}^+$ between one point compactifications. Then, Bauer and Furuta \cite{b-f} showed that the $S^1$-equivariant stable cohomotopy class $[f_{W}^+]$ of $f_{W}^+$ defines an element of a certain $S^1$-equivariant stable cohomotopy group $\pi^{b^+}_{S^1, \cB}(Pic^0(X),\ind D)$ and the class $[f_{W}^+]$ is independent of the choices of the Riemannian metric $g$ and self-dual $2$-form $\eta$. See \cite{b-f, b} for the precise definition of $\pi^{b^+}_{S^1, \cB}(Pic^0(X),\ind D)$. In the present article, we do not need to recall the precise definition of it. As a result, we reach the following definition: 
\begin{defn}[\cite{b-f}]
Let $X$ be a closed oriented smooth 4-manifold with ${b}^+(X)>1$ and $\Gamma_X$ be any spin$^c$ structure on $X$. Then the value of the  stable cohomotopy Seiberg-Witten invariant $BF_{X}$ for $\Gamma_X$ is defined to be
\begin{equation*}
BF_{X}(\Gamma_X):=[f_{W}^+] \in \pi^{b^+}_{S^1, \cB}(Pic^0(X),\ind D). 
\end{equation*}
\end{defn}

As was already mentioned in Introduction, one of the main purposes of this article is to prove a new non-vanishing theorem of $BF_{X}$. 

\subsection{Finite dimensional approximation and spin structures}\label{sasa-spin-1}

Let $M$ be any closed manifold. Then, the most fundamental question will be to ask if $M$ is orientable or not. If $M$ is orientable and oriented, a next fundamental question will be to ask if the oriented manifold $M$ admits a spin structure. Now, it is known that, for generic $\eta$, the Seiberg-Witten moduli space $\cM^{SW}$ defined by (\ref{moduli}) becomes an orientable, finite dimensional manifold. There is a natural way to orient $\cM^{SW}$. After choosing an orientation on $\cM^{SW}$, one can consider the fundamental homology class $[\cM^{SW}]$ of $\cM^{SW}$. The original Seiberg-Witten invariant $SW_{X}$ is defined by the using $[\cM^{SW}]$. Hence, a next fundamental question for $\cM^{SW}$ will be to ask when $\cM^{SW}$ admits a spin structure and if one can construct effective differential topological invariants by using spin cobordism class of $\cM^{SW}$ equipped with spin structures. These natural questions were explored by the second author \cite{sasa}. Interestingly, it turns out that this is the case in a sense. In this and next subsections, we shall review some of the main points of the theory because this has a fundamental importance to prove Theorem \ref{main-A}. We also hope that this is helpful for the reader who is unfamiliar to this subject. See also \cite{sasa}. The rough story is as follows. Let $F \rightarrow B$ is a vector bundle over an oriented manifold $B$ and $s: B \rightarrow F$ be a section of the bundle which is transverse to the zero section. Assume moreover, $s^{-1}(0)$ is compact. If both $F$ and $B$ admit spin structures then one can see that $s^{-1}(0)$ also admits a spin structure. Then one can consider a spin cobordism class of the compact manifold $s^{-1}(0)$ equipped with spin structure. We shall apply this scheme to a section of the bundle $E \rightarrow  \bar{V}$ (see (\ref{finite-se-bun}) below) induced by the finite approximation $f$ of the Seiberg-Witten map $\mu$. The cental problem of this subsection is to see when both $E$ and $\bar{V}$ admit spin structures. \par
Now, let $X$ be a closed, oriented $4$-manifold with $b^+(X)>1$. Take a Riemannian metric $g$ and a spin$^c$ structure $\Gamma_X$ on $X$. Let $\mu :\cA \longrightarrow \cC$ be the Seiberg-Witten map associated with $g$ and  $\Gamma_X$. By Theorem \ref{finte-SW-1} above, we are able to get a finite dimensional approximation of the Seiberg-Witten map $\mu$: 
\begin{equation*}
f:=f_W : V:=\cF(W) \longrightarrow W
\end{equation*}
We are also able to decompose naturally $V$ and $W$ as 
\begin{equation*}
V=V_{\C} \oplus V_{\R},  \ W=W_{\C} \oplus W_{\R}.
\end{equation*}
Here $V_{\C}$ and $V_{\R}$ respectively denote complex and real vector bundle over $\cT$ (see (\ref{cT})). And $W_{\C}$ and $W_{\R}$ are complex and real vector spaces. The $S^1$ actions on $V$ and $W$ are the scalar products on $V_{\C}$ and $W_{\C}$. \par
On the other hand, it is known that, for generic $\eta$, the perturbed Seiberg-Witten equations have no reducible solution, i.e., $\phi \not \equiv 0$. Hence, when $W$ is sufficiently large, $f^{-1}(0)$ lies in $V_{irr}:=(V_{\C} \backslash \{ 0 \}) \times_{\cT} V_{\R}$.  \par
Since $f$ is equivariant with respect to the $S^1$ actions, $f$ induces a section $s$ of the vector bundle 
\begin{equation}\label{finite-se-bun}
E:=V_{irr} \times_{S^1} W \longrightarrow \bar{V}:=V_{irr}/S^1. 
\end{equation}
By perturbing the section if necessarily, we may assume that the section $s$ is transverse to the zero section. We denote the zero locus of $s$ by $\cM$ i.e.,
\begin{equation}\label{pre-moduli}
\cM:={s}^{-1}(0) \subset \bar{V}.
\end{equation}
Then $\cM$ is a compact submanifold of $\bar{V}$.
Notice that $\cM$ is slightly different from the Seiberg-Witten moduli space (\ref{moduli}). \par
First of all, we shall prove 
\begin{lem} \label{ori}
A choice of an orientation on $\cH_g^1(X) \oplus \cH_g^+(X)$ induces an orientation on $\cM$.
\end{lem}
\begin{proof}
It is sufficient to orient the determinant line bundle $\det_{\R} T\cM$ of the tangent bundle $T\cM$ of $\cM$.  We fix an orientation $\cO$ on $\cH_g^1(X) \oplus \cH_g^+(X)$. Let $\cN$ be the normal bundle of $\cM$ in $\bar{V}$. Then we have 
\begin{equation} \label{decomposition TV}
T \bar{V}|_{\cM}=T\cM \oplus \cN.
\end{equation}
Since $s$ is transverse to the zero section, the derivative of $s$ induces the isomorphism
\begin{equation} \label{iso N E}
\cN \cong E|_{\cM}.
\end{equation}
From (\ref{decomposition TV}) and (\ref{iso N E}), we have
\begin{equation}\label{iso-tan}
T\cM \oplus E|_{\cM} \cong T\bar{V}|_{\cM}.
\end{equation}
Therefore, we have an isomorphism
\begin{equation} \label{det M}
\det \nolimits_{\R} T\cM \cong \det \nolimits_{\R} T \bar{V}|_{\cM} \otimes (\det \nolimits_{\R} E |_{\cM})^*.
\end{equation}
Since the linear part $l$ of the Seiberg-Witten map is given by (\ref{linear part}),  we have 
\[
W_{\R}=\cH_g^+(X) \oplus W_{\R}'
\]
for some vector space $W_{\R}'$. Moreover $l$ induces an isomorphism between $W_{\R}'$ and each fiber of $V_{\R}$. (This gives a trivialization of $V_{\R}$.)  The vector bundle $E$ over $\bar{V}$ has a decomposition $E=E_{\C} \oplus E_{\R}$, where 
\[
E_{\C}=V_{irr} \times_{S^1} W_{\C},
\quad
E_{\R}=\bar{V} \times W_{\R}.
\]
The complex part of $E_{\C}$ has the orientation induced by the complex structure. Hence there is a trivialization $\det_{\R} E_{\C} \cong \underline{\R}$. We obtain
\begin{equation} \label{det E}
\det \nolimits_{\R} E \cong \bar{\pi}^* \det \nolimits_{\R} \underline{\cH}_g^+(X) \otimes \bar{\pi}^* \det \nolimits_{\R} V_{\R},
\end{equation}
where $\underline{\cH}_g^+(X)$ is the trivial vector bundle over $\cT$ with fiber $\cH_g^+(X)$ and $\bar{\pi}:\bar{V} \rightarrow \cT$ is the projection.
On the other hand, there is a natural isomorphism (see Lemma 3.4 in \cite{sasa}):
\begin{equation} \label{TV}
T\bar{V} \oplus \underline{\R} \cong
\bar{\pi}^* T(\cT) \oplus (\bar{\pi}^* V_{\C} \otimes_{\C} H) \oplus \bar{\pi}^* V_{\R}
\end{equation}
where $T(\cT)$ is the tangent bundle of $\cT$ and $H:=V_{irr} \times_{S^1} \C$. Since $T(\cT)$ has a natural trivialization \[
T(\cT) \cong \underline{\cH}_g^1(X):={\cT} \times {\cH}_g^1(X),
\]
we have
\begin{equation} \label{det V}
\det \nolimits_{\R} T \bar{V} \cong \bar{\pi}^* \det \nolimits_{\R} \underline{\cH}_g^1(X) \otimes \bar{\pi}^* \det \nolimits_{\R} V_{\R}.
\end{equation}
From (\ref{det M}), (\ref{det E}) and (\ref{det V}), we have
\[
\begin{split}
\det \nolimits_{\R} T\cM 
&\cong \left.
\bar{\pi}^* \big( \det \nolimits_{\R} \underline{\cH}_g^1(X) \otimes \det \nolimits_{\R} V_{\R} \otimes
( \det \nolimits_{\R} \underline{\cH}_g^+(X) )^* \otimes ( \det \nolimits_{\R} V_{\R} )^* \big) \right|_{\cM}  \\
&\cong \left.
\bar{\pi}^* \big( \det \nolimits_{\R} \underline{\cH}_g^1(X) \otimes (\det \nolimits_{\R} \underline{\cH}_g^+(X))^* \big) 
\right|_{\cM}, 
\end{split}
\]
where we used the canonical trivialization $\det \nolimits_{\R} V_{\R} \otimes ( \det \nolimits_{\R} V_{\R})^* \cong \underline{\R}$. The Riemannian metric on $X$ induces an isomorphism between $\cH_g^+(X)$ and $\cH_g^+(X)^*$. Hence the orientation $\cO$ on $\cH_g^1(X) \oplus \cH_g^+(X)$ orients $\det \nolimits_{\R} \cH_g^1(X) \otimes (\det \nolimits_{\R} \cH_g^+(X))^*$ and hence  $\det \nolimits_{\R} T\cM$. 
\end{proof}

Next, we shall prove the following result: 

\begin{prop} \label{M spin}
Let $\{ \frak{e}_j \}_{j=1}^s$ be a set of generators of $H^1(X, \Z)$, where $s:=b_1(X)$. Let $f:V \rightarrow W$ be a finite dimensional approximation of the Seiberg-Witten map $\mu$ associated with $g$ and $\Gamma_X$. Assume that $m:=\dim_{\C} W_{\C}$ is even. Let ${\cal I}_{\Gamma_{X}} \in \Z$ be the numerical index of the Dirac operator associated with $g$ and $\Gamma_X$. Moreover, put $\frak{S}^{ij}(\Gamma_X):=\frac{1}{2}\< c_1(\cL_{\Gamma_X}) \cup \frak{e}_i \cup \frak{e}_j, [X] \> \in \Z$ as Definition \ref{def-1} above. Suppose that the following conditions are satisfied: 
\[
(*) \left\{
\begin{array}{ll}
{\cal I}_{\Gamma_{X}} \equiv 0 \ (\bmod \ 2), & \\
\frak{S}^{ij}(\Gamma_X) \equiv 0 \ (\bmod \ 2) & \text{for all $i, j$}.
\end{array}
\right.
\]
Then $\cM$ is spin.
\end{prop}

Let us give a proof of this proposition. First of all, by the isomorphism (\ref{iso-tan}), we have $T\cM \oplus E|_{\cM} \cong T\bar{V}|_{\cM}$. Hence, a sufficient condition for $\cM$ to be spin is that both $T\bar{V}$ and $E$ are spin. Below, we consider when the second Stiefel-Whitney classes of both $T\bar{V}$ and $E$ are trivial. \par
Let $\Ind D \in K(\cT)$ be the index bundle of the family $\{ D_{A} \}_{[A] \in \cT}$ of Dirac operators. Since the complex part of the linear part $l$ of the Seiberg-Witten map $\mu$ is the Dirac operator, we have $\Ind D=[V_{\C}] - [\underline{\C}^m ] \in K(\cT)$, where $m:=\dim_{\C} W_{\C}$. The real part $E_{\R}$ of $E:=V_{irr} \times_{S^1} W$ is trivial and the complex part $E_{\C}$ of $E$ is given by $E_{\C}=V_{irr} \times_{S^1} W_{\C}$. Hence $E_{\R}$ is spin and the first Chern class of $E_{\C}$ is 
\begin{equation}\label{ch-iso}
c_1(E_{\C})=m c_1(H) \in H^2(\bar{V},\Z), 
\end{equation}
where, again, $H:=V_{irr} \times_{S^1} \C$. \par
In general, the mod $2$ reduction of the first Chern class of a complex bundle is equal to the second Stiefel-Whitney class. Therefore, we are able to conclude that, by (\ref{ch-iso}), $E_{\C}$ is spin if $m:=\dim_{\C} W_{\C}$ is even. From now on, we assume that $m$ is even. By the above argument, the $E:=V_{irr} \times_{S^1} W$ is spin in this situation. \par
On the other hand, the isomorphism (\ref{TV}) above  means that $T\bar{V}$ is spin if the complex bundle $\bar{\pi}^* V_{\C} \otimes H$ is spin since both $T(\cT)$ and $V_{\R}$ are trivial.  The first Chern class of $\bar{\pi}^* V_{\C} \otimes H$ is given by
\begin{equation} \label{c1 VH}
\begin{split}
c_1(\bar{\pi}^* V_{\C} \otimes H)
&= \bar{\pi}^* c_1(V_{\C}) + (m+a) c_1(H) \\
&= \bar{\pi}^* c_1(\Ind D) + (m+a) c_1(H) \\
&\in H^2(\bar{V}, \Z).
\end{split}
\end{equation}
Here $a :={\cal I}_{\Gamma_{X}} \in \Z$ is the numerical index of the Dirac operator $D_A$.  The first Chern class of $\Ind D$ is calculated by Li--Liu \cite{l-l} and Ohta--Ono \cite{o-o} independently. Let us recall the result. 

Let $\{ \frak{e}_j \}_{j=1}^{s}$ be a set of generator of $H^1(X, \Z)$, where $s:=b_1(X)$. Then, we have an identification between $\cT$ (see (\ref{cT})) and the $n$-dimensional torus $T^n$. Fix a point $x_0 \in X$. We define a map $\psi$ from $X$ to $T^n \cong \cT$ by 
\[
x \longmapsto
\left(
\int_{x_0}^x \frak{e}_1,\dots, \int_{x_0}^x \frak{e}_s
\right).
\]
The symbol $\int_{x_0}^x$ means an integration over some curve on $X$ connecting $x_0$ to $x$ . From the Stokes theorem, this map is independent of the choice of curve. It is easy to see that the induced homomorphism $\psi^*:H^1(\cT, \Z) \rightarrow H^1(X, \Z)$ is isomorphic. Put $\frak{f}_j:=(\psi^*)^{-1}(\frak{e}_j) \in H^1(\cT, \Z)$. Then a result in \cite{l-l, o-o} tells us that the following holds:
\begin{prop} [\cite{l-l, o-o}] \label{c1 ind D}
The first Chern class of $\Ind D$ is given by
\[
c_1(\Ind D)=\frac{1}{2}\sum_{i<j}\< c_1(\cL_{\Gamma_X}) \cup \frak{e}_i \cup \frak{e}_j, [X] \> \frak{f}_i \cup \frak{f}_j
\in H^2(\cT, \Z).
\]
\end{prop}

From (\ref{c1 VH}) and Proposition \ref{c1 ind D}, it follows that $V_{\C} \otimes H$ is spin if 
\begin{gather*}
a \equiv 0 \pmod 2, \\
\frak{S}^{ij}(\Gamma_X):=\frac{1}{2}\< c_1(\cL_{\Gamma_X}) \cup \frak{e}_i \cup \frak{e}_j, [X] \>  \equiv 0 \pmod 2
\quad ({}^{\forall} i, j )
\end{gather*}
because we assume that $m:=\dim_{\C} W_{\C}$ is even. Hence, we have shown Proposition \ref{M spin} as promised. \par 
On the other hand, note that, if $\frak{S}^{ij}(\Gamma_X)$ are even for all $i, j$, then Proposition \ref{c1 ind D} tells us that the first Chern class of $\Ind D$ is also even. We next explain that, under the condition $(*)$ in Proposition \ref{M spin}, a choice of square root of the complex determinant line bundle of the family of Dirac operators $\{ D_A \}_{A \in \cT}$ induces a spin structure on $\cM$. More precisely, we have 
\begin{prop} \label{M spin 2}
Assume that the condition $(*)$ in Proposition \ref{M spin} holds. An orientation $\cO$ on $\cH_g^1(X) \oplus \cH_g^+(X)$ and a square root $L$ of $\det_{\C} (\Ind D)$ induce a spin structure on $\cM$.
\end{prop}

\begin{proof} 
See also the proof of Proposition 3.10 in \cite{sasa}. Assume now that $(*)$ in Proposition \ref{M spin} is satisfied. Then, it follows from the above discussion that both $T\bar{V}$ and $E:=V_{irr} \times_{S^1} W$ have spin structures. More precisely, the spin structures on $T \bar{V}$ and $E$ induce a spin structure on $\cM$ compatible with the isomorphism $T\cM \oplus E|_{\cM} \cong T\bar{V}|_{\cM}$. Hence, we only have to show that an orientation on $\cH_g^1(X) \oplus \cH_g^+(X)$ and a square root of $\det_{\C} \Ind D$ induce spin structures on $T\bar{V}$ and $E$. \par
First, we equip $E$ with a spin structure in the following way. Fix an orientation on $W_{\R}$. Then we obtain a spin structure on the real part $E_{\R}=\bar{V} \times W_{\R}$ of $E$ compatible with the orientation. It is well known that a choice of square root of the complex determinant line bundle of a complex vector bundle induces a spin structure when the first Chern class is even (see Lemma 3.9 in \cite{sasa}). Since $\det \nolimits_{\C} E_{\C}=H^{\otimes m}$ and $m$ is even, $H^{\otimes m/2}$ is a square root of $\det_{\C} E_{\C}$. Hence, $E_{\C}$ has a natural spin structure. Therefore, we obtain a spin structure on $E=E_{\C} \oplus E_{\R}$.  \par
Next, we equip $T\bar{V}$ with a spin structure. From (\ref{TV}), it is sufficient to equip $V_{\R}$ ,$\underline{\cH}_g^1(X)$, $V_{\C} \otimes H$ with spin structures.  Take an orientation on $\cH_g^+(X)$. Then we have a spin structure on $V_{\R}$ compatible with orientations on $W_{\R}$, $\cH_g^+(X)$ and the linear part $l$ of the Seiberg-Witten map $\mu$. The orientation on $\cH_g^+(X)$ also induces a spin structure on $\underline{\cH}_g^1(X)$ compatible with $\cO$.
Let $L$ be a square root of $\det_{\C} \Ind D$. Then $L \otimes H^{\otimes (m+a)/2}$ is a square root of 
\[
\det \nolimits_{\C} (V_{\C} \otimes H)=\det \nolimits_{\C} (\Ind D) \otimes H^{\otimes (m+a)}.
\]
Hence we have a spin structure on $T\bar{V}$.
Therefore we obtain a spin structure on $\cM$. It is not hard to see that the spin structure is independent of the choice of orientations on $W_{\R}$ and $\cH_g^+(X)$.  Hence the claim follows. 
\end{proof}

\subsection{Spin cobordism Seiberg-Witten invariant}\label{sasa-spin}

Suppose that $(*)$ in Proposition \ref{M spin} holds. As we shall see below, then  we are able to show that the spin cobordism class of $\cM$ is independent of the choice of both Riemannian metric $g$ and finite dimensional approximation of the Seiberg-Witten map $\mu$. In fact, we have 
\begin{thm}\label{spin SW inv}
Let $X$ be a closed, oriented $4$-manifold with $b^+(X)>1$ and $\Gamma_X$ be a spin$^c$ structure on $X$. Fix an orientation $\cO$ on $\cH_g^1(X) \oplus \cH_g^+(X)$. Assume that the condition $(*)$ in Proposition \ref{M spin} holds and take a square root $L$ of $\det_{\C} (\Ind D)$. Then the spin cobordism class of $\cM$ is independent of the choice of both Riemannian metric on $X$ and finite dimensional approximation of the Seiberg-Witten map. Hence, the spin cobordism class gives rise to a differential topological invariant of $X$.
\end{thm}
Let us give outline of the proof. Take two finite dimensional approximations $f_0:V_0 \rightarrow W_0$ and $f_1:V_1 \rightarrow W_1$. We write $\cM_i$ for $f_i^{-1}(0)/S^1$. Considering a larger finite dimensional approximation $f:V \rightarrow W$ with $V_i \subset V$, $W_i \subset W$, we are able to reduce the proof to the case when $V_0 \subset V_1$, $W_0 \subset W_1$. Let $U$ be the orthogonal complement of $W_0$ in $W_1$. By Theorem \ref{finte-SW-1}, we have a homotopy $\tilde{f}^+$ between $(f_0 \oplus pr_{U} \circ l|_{\cF(U)})^+$ and $f_1^+$. By perturbing the map $\tilde{f}^+$ if necessarily, we may conclude that $\widetilde{\cM}=(\tilde{f}^+)^{-1}(0)/S^1$ becomes a smooth manifold with boundary $\cM_0 \coprod \cM_1$. The square root $L$ also induces a spin structure on $\widetilde{\cM}$ whose restrictions to $\cM_i$ are the spin structures on $\cM_i$. So $\widetilde{\cM}$ is a spin cobordism between $\cM_0$ and $\cM_1$.
Hence the spin cobordism class is independent of the finite dimensional approximation of the Seiberg-Witten map. The proof of independence on Riemannian metrics is also similar. See \cite{sasa} for more details. \par
We are now in a position to introduce the precise definition of the spin cobordism Seiberg-Witten invariant:
\begin{defn}\label{sasahira-inv}
Let $X$ be closed, oriented $4$-manifold with $b^+(X)>1$ and $\Gamma_X$ be a spin$^c$ structure on $X$. Fix an orientation $\cO$ on $\cH_g^1(X) \oplus \cH_g^+(X)$. Assume that the condition $(*)$ in Proposition \ref{M spin} holds and choose a square root $L$ of $\det_{\C} (\Ind D)$. Then we denote the spin cobordism class of $\cM$ by
\[
SW^{spin}(\Gamma_X, L) \in \Omega^{spin}_d.
\]
Here $d$ is the dimension of $\cM$. We shall call this spin cobordism Seiberg-Witten invariant of $X$. 
\end{defn}

\begin{rem}\label{dimension-f}
It is known that the dimension of Seiberg-Witten moduli space $\cM^{SW}$ (see (\ref{moduli})) associated with the spin$^c$ structure $\Gamma_X$ is given by 
\begin{equation}\label{SW-dim}
\dim \cM^{SW} = \frac{1}{4} \Big( {c^2_1(\cL_{\Gamma_X}) - 2\chi(X) - 3\tau(X)} \Big).
\end{equation}
On the other hand, $\cM$ (see (\ref{pre-moduli})) is slightly different from $\cM^{SW}$. However, the  dimension of $\cM$ is also given by the above formula. In fact, since $\cM$ is the zero locus of the section $s$ of the bundle $E$, we have
\begin{equation} \label{eq dim M 1}
\dim \cM = \dim \bar{V} - \rank_{\R} E.
\end{equation}
The rank of $E_{\R}$ is equal to $b^+ (X) + n$ and the rank of $E_{\C}$ over $\R$ is equal to $2m$.
On the other hand, the dimension of $\bar{V}$ is given by
\[
b_1(X) + 2(m+a) + n - 1
= b_1(X)+2m + \frac{1}{4} \Big({c^2_1(\cL_{\Gamma_X}) - \tau(X)} \Big) + n - 1.
\]
From (\ref{eq dim M 1}), we have
\begin{equation*} \label{eq dim M 2}
\begin{split}
\dim \cM
&= b_1(X) + \frac{1}{4} \Big({c^2_1(\cL_{\Gamma_X}) - \tau(X)} \Big) - 1 - b^+(X) \\
&=\frac{1}{4} \Big({c^2_1(\cL_{\Gamma_X}) - 2\chi(X) - 3\tau(X)} \Big).
\end{split}
\end{equation*}
hence we have
\begin{equation}\label{SW-Mdim}
\dim \cM^{SW} = \dim \cM = \frac{1}{4} \Big({c^2_1(\cL_{\Gamma_X}) - 2\chi(X) - 3\tau(X)} \Big).
\end{equation}
\end{rem}

Let us close this subsection with the following proposition which clarifies a relationship between $BF_{X}$ and $SW^{spin}_{X}$:
\begin{prop}\label{spin-BF}
Let ${X}$ be a closed oriented 4-manifold with $b^+(X)>1$ and let $\Gamma_{X}$ be a spin${}^{c}$ structure. Suppose that the condition $(*)$ in Proposition \ref{M spin} holds and that the value of the spin cobordism Seiberg-Witten invariant for $\Gamma_{X}$ is non-trivial, i.e., $SW^{spin}(\Gamma_{X},L) \not= 0$ for some square root $L$ of $\det_{\C} (\Ind D)$. Then the value of the stable cohomotopy Seiberg-Witten invariant for $\Gamma_{X}$ is also non-trivial, i.e., $BF(\Gamma_{X}) \not= 0$. 
\end{prop}

\begin{proof}
We shall prove the claim by showing the contraposition. Assume now that $BF(\Gamma_X)$ is trivial. Let $f$ be a finite dimensional approximation of the Seiberg-Witten map associated with $\Gamma_X$ and some Riemannian metric. Fix any square root $L$ of $\Ind D$. By the hypothesis, we have a homotopy $\tilde{f}^+$ between $f^+$ and the constant map $h$. The map $h$ maps all elements in $V^+$ to $\infty \in W^+$. We may assume that $\tilde{f}$ is transverse to $0$. Since the zero locus of $h$ is empty, we conclude that $\widetilde{\cM}=\tilde{f}^{-1}(0)/S^1$ is a smooth manifold with boundary $\cM=f^{-1}(0)/S^1$. The square root $L$ also induces a spin structure on $\widetilde{\cM}$ compatible with the spin structure on $\cM$. Hence $SW^{spin}(\Gamma_X, L)$ must be trivial by the very definition.
\end{proof} 

We shall use Proposition \ref{spin-BF} to prove Theorem \ref{main-A}. On the other hand, it is a natural question to ask if there is a natural relationship among $SW_{X}$, $BF_{X}$, and $SW^{spin}_{X}$. In next subsection, we shall explore this problem and point out that there is a natural commutative diagram among $SW_{X}$, $BF_{X}$ and a refinement $\widehat{SW}^{spin}_X$ of spin cobordism Seiberg-Witten invariant. Though this result is not used in this article, this has its own interest.

\subsection{Commutative diagram among three invariants}\label{diag}

Let $X$ be a closed, oriented $4$-manifold with $b^+(X)>1$. Take a Riemannian metric $g$ and $\Gamma_X$ be a spin$^c$ structure on $X$. Let $\mu : \cA \longrightarrow \cC$ be a corresponding Seiberg-Witten map $\mu$ and $f$ be a finite approximation of $\mu$. On the other hand, we write $\cA_{irr}$ for the complement of the set of fixed point of the $S^1$-action on $\cA$, i.e.,
\[
\cA_{irr}:= \Big(A_0+\ker d \Big) \times_{\cG_0} \Big( (\Gamma(S_{\Gamma_X}^+) \backslash \{ 0 \})  \times \Omega_X^1 \Big).
\]
The action of $S^1$ on $\cA_{irr}$ is free. We denote the quotient $\cA_{irr} / S^1$ by $\cB_{irr}$. We denote the spin cobordism of the space $\cB_{irr}$ by $\Omega_d^{spin}(\cB_{irr})$. Namely, 
\[
\left.
\Omega_d^{spin}(\cB_{irr}) := 
\left\{ (Y, \sigma, \varphi) \left|
\begin{array}{lll}
\text{$Y$: a closed, oriented $d$-manifold} \\
\text{$\sigma$: a spin structure on $Y$} \\
\text{$\varphi :Y \rightarrow \cB_{irr}$ : a continuous map }
\end{array}
\right.
\right\} \right/ \sim
\]
where the equivalence relation is defined by the following way:
\[
\begin{split}
&(Y_1, \sigma_1, \varphi_1) \sim (Y_2, \sigma_2, \varphi_2) 
\stackrel{\text{def}}{\Longleftrightarrow} \\
& \hspace{20mm}
\text{ ${}^{\exists} (W, \sigma_{W},\varphi_{W})$ s.t.  } \\
& \hspace{23mm}
\text{$W$: an oriented $d+1$-manifold s.t. $\partial W= Y_1 \coprod Y_2$}, \\
&\hspace{23mm}
\text{$\sigma_{W}$: a spin structure on $W$ with $\sigma_{W}|_{Y_i} = \sigma_{i}$}, \\
&\hspace{23mm}
\text{$\varphi_W:W \rightarrow \cB_{irr}$: a continuous map with $\varphi_{W}|_{Y_i}=\varphi_i$}.
\end{split}
\]
$\Omega_{d}^{spin}(\cB_{irr})$ has a module structure in a natural way. \par
Now, we have the natural inclusion $\iota: \cM=f^{-1}(0)/S^1 \hookrightarrow \cB_{irr}$.  Assume that the condition $(*)$ in Proposition \ref{M spin} holds. Then, a choice of a square root $L$ of $\det \Ind D$ induces a spin structure $\sigma_L$ on $\cM$. The triple $(\cM, \sigma_L, \iota)$ defines a class of $\Omega_{d}^{spin}(\cB^*)$, where $d$ is the dimension of $\cM$. We write $\widehat{SW}^{spin}_X(\Gamma_X, L)$ for the 
class:
\[
\widehat{SW}^{spin}_X(\Gamma_X, L) \in \Omega_{d}^{spin}(\cB_{irr}).
\]
In a similar way to the proof of Theorem \ref{spin SW inv}, we can show the 
following:
\begin{thm}
Let $X$ be a closed, oriented $4$-manifold with $b^+(X)>1$ and $\Gamma_X$ be a spin$^c$ structure on $X$. Fix an orientation on $\cH_g^1(X) \oplus \cH_g^+(X)$. Assume that the condition $(*)$ in Proposition \ref{M spin} holds and fix a square root $L$ of $\det \Ind D$. Then $\widehat{SW}^{spin}_{X}(\Gamma_X,L) \in \Omega^{spin}_d(\cB_{irr})$ is independent of the choice of metric and finite dimensional approximation of the Seiberg-Witten map, and hence $\widehat{SW}^{spin}_X(\Gamma_X, L)$ is a differential topological invariant of $X$.
\end{thm}

We leave the detail of the proof of this theorem, as an exercise, for the interested reader. See also \cite{sasa}. \par
On the other hand, let $F$ be the forgetful map form $\Omega_d^{spin}(\cB_{irr})$ to $\Omega_d^{spin}$. Then the following follows from the definition of ${SW}^{spin}_X$ and $\widehat{SW}^{spin}_X$:
\[
F(\widehat{SW}^{spin}_X(\Gamma_X,L))=SW_{X}^{spin}(\Gamma_X,L) \in 
\Omega_d^{spin}.
\]
We are now in a position to state the main result of this subsection:
\begin{thm} \label{prop diagram}
Let $X$ be a closed, oriented $4$-manifold with $b^+(X)>b_1(X)+1$. Take a 
spin$^c$ structure $\Gamma_X$ and a finite dimensional approximation $f$ of 
the Seiberg-Witten map $\mu$ associated with $\Gamma_X$. Assume that the dimension $d$ of $\cM=f^{-1}(0)/S^1$ is even and the condition $(*)$ in Proposition \ref{M spin} holds. Fix a square root $L$ of $\det \Ind D$. Then there are the following natural maps:
\[
\begin{split}
t^{BF} &: \pi^{b^+}_{S^1, \cB}(Pic^0(X),\ind D) \longrightarrow \Z,\\
t_{L}^{spin} &: \pi^{b^+}_{S^1, \cB}(Pic^0(X),\ind D) \longrightarrow 
\Omega_d^{spin}(\cB_{irr}), \\
t^{spin} &: \Omega_d^{spin}(\cB_{irr}) \longrightarrow \Z,
\end{split}
\]
such that $t^{BF}(BF_X(\Gamma_X))=SW_X(\Gamma_X)$, 
$t_{L}^{spin}(BF_X(\Gamma_X))=\widehat{SW}_{X}^{spin}(\Gamma_X,L)$ and 
$t^{spin}(\widehat{SW}_{X}(\Gamma_X))=SW_X(\Gamma_X)$. Moreover the 
following diagram is commutative:
\[
\xymatrix{
\pi^{b^+}_{S^1, \cB}(Pic^0(X),\ind D) \ar[dd]_{t_L^{spin}} \ar[rd]^{t^{BF}} &  \\
                                                                  &\Z \\
\Omega_d^{spin}(\cB^*) \ar[ru]_{t^{spin}} &
}
\]
\end{thm}
In  what follows, we shall give a proof of Theorem \ref{prop diagram}. \par
First of all, we define the map 
\[
t^{BF} : \pi^{b^+}_{S^1, \cB}(Pic^0(X),\ind D) \rightarrow \Z. 
\]
This map was defined in \cite{b-f}. Assume that 
$b^+(X)>b_1(X)+1$ and $\dim \cM=2d'$. Take an element $[\alpha] \in 
\pi^{b^+}_{S^1, \cB}(Pic^0(X),\ind D)$. Here $\alpha:V^+ \rightarrow W^+$ is a 
representative. The difference of the dimensions of $\{ 0 \} \times V_{\R} (\subset V)$ and $W$ is $b^+(X) - b_1(X)$.
Since $b^+(X)>b_1(X)$, there is a small $\varepsilon > 0$ and a map $\beta:V_{\R} \rightarrow W$ such that $\alpha|_{V_{\R}} + \varepsilon \beta$ does not intersect with zero in $W$. Put $\alpha'=\alpha + \varepsilon \beta$. Then $\alpha^{' -1}(0) \subset (V_{\C} \backslash \{ 0 \} ) \times_{\cT} V_{\R}$, and $M=\alpha^{'-1}(0)/S^1$ is a compact, smooth manifold with dimension $d=2d'$. 
Then define as
\[
t^{BF}([\alpha]) :=\< c_1(H^{d'}), [M] \> \in \Z,
\]
where $H:=\cA_{irr} \times_{S^1} \C \rightarrow \cB_{irr}$. Then we have the following:
\begin{lem}
$t^{BF}([\alpha])$ is independent of the choice of representative $\alpha$ and $t^{BF}(BF_{X}(\Gamma_X))$ 
is equal to $SW_X(\Gamma_X)$.
\end{lem}
See \cite{b-f} for details. \par
Next we define the map 
\[
t_{L}^{spin}:\pi^{b^+}_{S^1, \cB}(Pic^0(X),\ind D) \rightarrow \Omega^{spin}_{d}(\cB_{irr}).
\]
Suppose that the condition $(*)$ in Proposition \ref{M spin} holds. Fix a square root $L$ of $\det \Ind D$.  Let $[\alpha] \in \pi^{b^+}_{S^1, \cB}(Pic^0(X),\ind D)$. As above, we may assume that $\alpha^{-1}(0) \subset ( V_{\C} \backslash \{ 0 \} ) \times_{\cT} V_{\R}$ and $M_{\alpha}=\alpha^{-1}(0)/S^1$ is a closed, smooth manifold with 
dimension $d$. We can show that the choice of $L$ induces a spin structure 
$\sigma_{L,\alpha}$ on $M_{\alpha}$ in the same way as 
Proposition \ref{M spin 2}. We put
\[
t_L^{spin}([\alpha]):=[M_{\alpha}, \sigma_{L,\alpha}, 
\iota_{\alpha}] \in \Omega_d^{spin}(\cB_{irr}),
\]
where $\iota_{\alpha}:M_{\alpha} \hookrightarrow \cB_{irr}$ is the inclusion. 
Then we have

\begin{lem}
$t_L^{spin}([\alpha])$ is independent of the choice of representative 
$\alpha$ and 
$t_{L}^{spin}(BF_X(\Gamma_X))=\widehat{SW}_X^{spin}(\Gamma_X,L)$.
\end{lem}

\begin{proof}
 By the definition of $t^{spin}_L$, we have $t^{spin}_L(BF_X(\Gamma_X))=\widehat{SW}_X(\Gamma_X,L)$. So, it is enough to show the well-definedness of the map $t_{L}^{spin}$. \par
Now, take $S^1$-equivariant maps $\alpha:V^+ \rightarrow W^+$, $\alpha':V^{'+} \rightarrow W^{'+}$ with 
$[\alpha]=[\alpha'] \in \pi_{S^1, \cB}^{b^+}(Pic^0(X),\ind D)$. We may 
suppose that both $M_0=\alpha^{-1}(0)/S^1$ and $M_1=\alpha^{'-1}(0)/S^1$ are 
closed, smooth manifolds with dimension $d$. The square root $L$ induces spin 
structures $\sigma_{0,L}$, $\sigma_{1,L}$ on $M_0, M_1$, respectively. We 
denote natural inclusions by $\iota_0:M_0 \hookrightarrow \cB_{irr}$, 
$\iota_1:M_1 \hookrightarrow \cB_{irr}$.
Considering another map $\alpha'':V^{''+} \rightarrow W^{''+}$ with 
$[\alpha'']=[\alpha]=[\alpha'] \in \pi_{S^1, \cB}^{b^+}(Pic^0(X),\ind 
D)$ and satisfying 
\[
V^+ \subset V^{''+}, \quad W^+ \subset W^{''+}, \quad V^{'+} \subset V^{''+}, \quad W^{'+} \subset W^{''+},
\]
we are able to reduce the proof to the case when $V^+ \subset V^{'+}$, $W^+ \subset W^{'+}$. \par
Now, let $U_0$, $U_1$ be the orthogonal complements of $V,W$ in 
$V',W'$. Then the linear part $l$ (see (\ref{linear part})) induces an isomorphism between $U_0$ and $U_1$, and there is an $S^1$-equivariant homotopy $h:V^{'+} \times [0,1] \rightarrow W^{'+}$ between $\big( \alpha \oplus (l|_{U_0}) \big)^+$ and $\alpha'$ by Theorem \ref{finte-SW-1}. 
From the assumption that $b^+(X)>b_1(X)+1$, we may suppose that $h^{-1}(0) 
\subset (V'_{\C} \backslash \{ 0 \}) \times_T V_{\R} \times [0,1]$ and 
$\widetilde{M}=h^{-1}(0)/S^1$ is a $d+1$-dimensional smooth manifold with 
boundary $M_0 \coprod M_1$. The square root $L$ induces a spin structures on 
$\widetilde{M}$ whose restrictions to $M_0$, $M_1$  are the spin structures $\sigma_{0,L}$, $\sigma_{1,L}$ respectively. \par
Finally, Let $\tilde{\iota}:\widetilde{M} \hookrightarrow \cB_{irr} \times [0,1]$ be the natural inclusion. Then $(\widetilde{M},\tilde{\sigma}_L,\tilde{\iota})$ is a cobordism between $(M_0,\sigma_{0,L},\iota_0)$ and $(M_1,\sigma_{1,L},\iota_{1})$. Hence we have $[M_0,\sigma_{0,L},\iota_0]=[M_1,\sigma_{1,L},\iota_1] \in \Omega_d^{spin}(\cB_{irr})$ as desired.
\end{proof}

Finally, we shall define 
\[
t^{spin}:\Omega_{d}^{spin}(\cB_{irr}) \rightarrow \Z
\]
as follows. Take an element $[M,\sigma,f] \in \Omega_d^{spin}(\cB_{irr})$. Here $M$ is a closed, 
oriented $(d=)2d'$-dimensional manifold, $\sigma$ is a spin structure on $M$ 
and $f:M \rightarrow \cB_{irr}$ is a continuous map. We define as
\[
t^{spin}([M,\sigma, f]) := \< f^{*}c_1(H)^{d'}, [M] \> \in \Z, 
\]
where, again, $H:=\cA_{irr} \times_{S^1} \C \rightarrow \cB_{irr}$. Then we have
\begin{lem}
$t^{spin}([M,\sigma,f])$ is independent of the choice of representative 
$(M,\sigma,f)$ and $t^{spin}(\widehat{SW}_X(\Gamma_X,L))=SW_X(\Gamma_X)$.
\end{lem}
\begin{proof}
$t^{spin}(\widehat{SW}^{spin}(\Gamma_X))=SW(\Gamma_X)$ follows from the definition of the invariants. So it is enough to show the well-definedness of the map $t^{spin}$. Now, take two triples $(M_0,\sigma_0, f_0)$, $(M_1,\sigma_1,f_1)$ with $[ M_0,\sigma_0, f_0 ] = [ M_1,\sigma_1,f_1 ]$ in $\Omega_d^{spin}(\cB_{irr})$. Then there is a cobordism $(\widetilde{M}, \tilde{\sigma}, \tilde{f})$ between the two triples. By Stokes theorem, we have
\[
0=\int_{\widetilde{M}} \delta f^* c_1(H^{d'}) = \< f_1^* c_1(H^{d'}), [M_1] \> - \< f_0^* c_1(H^{d'}), [M_0] \>, 
\]
where we considered $c_1(H^{d'})$ as a closed differential form and $\delta$ is the exterior derivative in the above equation. Therefore the map $t^{spin}$ is well-defined. 
\end{proof}

From the above definition of three maps, it is clear that the diagram in 
Theorem \ref{prop diagram} is commutative. Hence we have done the proof of 
Theorem \ref{prop diagram}. In particular, as a corollary of Theorem \ref{prop diagram}, we obtain

\begin{cor} 
Let $X, \Gamma_X, L$ be as in Theorem \ref{prop diagram}. 
Then we have
\[
\begin{split}
&BF_{X}(\Gamma_X)=0 \in \pi^{b^+}_{S^1, \cB}(Pic^0(X),\ind D)
\Longrightarrow
\left\{
\begin{array}{l}
SW_X(\Gamma_X)=0 \in \Z, \\
\widehat{SW}_X^{spin}(\Gamma_X,L)=0 \in \Omega_d^{spin}(\cB_{irr}), \\
SW_{X}^{spin} (\Gamma_X,L)=0 \in \Omega_d^{spin}.
\end{array}
\right. \\
&\widehat{SW}_X^{spin}(\Gamma_X,L)=0 \in \Omega_d^{spin}(\cB_{irr})
\Longrightarrow
\left\{
\begin{array}{l}
SW_{X}(\Gamma_X)=0 \in \Z, \\
{SW}_X^{spin}(\Gamma_X,L)=0 \in \Omega_d^{spin}.
\end{array}
\right.
\end{split}
\]

\end{cor}

\section{Proof of Theorem \ref{main-A}}\label{Sec-3}

In this section, we shall prove Theorems \ref{main-A} and \ref{main-B} which 
were stated in Introduction.

\subsection{Non-triviality of spin cobordism Seiberg-Witten invariant}\label{sub-31}

We shall start with the following simple observation (see also Remark \ref{dimension-f} above):
\begin{lem}\label{simple-com}
Let $X$ be any closed oriented smooth 4-manifold with a Riemannian metric 
$g$ and $\Gamma_{X}$ be a spin$^c$ structure on $X$. Let ${\cal 
I}_{\Gamma_{X}}$ be the numerical index of the Dirac operator associated to 
$g$ and $\Gamma_X$. Then the following two conditions are equivalent:
\begin{itemize}
\item ${\cal I}_{\Gamma_{X}} \equiv 0 \ (\bmod \ 2)$, 
\item ${d}_{\Gamma_{X}}+{b}^{+}(X)-{b}_{1}(X) \equiv 3 \ (\bmod \ 4)$,
\end{itemize}
where ${d}_{\Gamma_{X}}$ is the dimension of ${\cal M}$, where see (\ref{pre-moduli}) and (\ref{SW-Mdim}).
\end{lem}
\begin{proof}
This follows from a direct computation. Indeed, we have the following by the 
formula (\ref{SW-Mdim}):
\begin{equation*}
{d}_{\Gamma_{X}} = \frac{1}{4}\Big( {c^2_1(\cL_{\Gamma_X}) - 2\chi(X) - 
3\tau(X)}\Big).
\end{equation*}
Equivalently,
\begin{equation}\label{c-com-2}
{c}^{2}_{1}({\cal L}_{\Gamma_{X}})=4{d}_{\Gamma_{X}}+2\chi(X)+3\tau(X).
\end{equation}
On the other hand, the Atiyah-Singer index theorem tells us that the index 
${\cal I}_{\Gamma_{X}}$ is given by
\begin{eqnarray*}
{\cal I}_{\Gamma_{X}}=\frac{1}{8}\Big({c}^{2}_{1}({\cal 
L}_{\Gamma_{X}})-\tau(X) \Big).
\end{eqnarray*}
This formula and (\ref{c-com-2}) tell us that
\begin{eqnarray*}
{\cal I}_{\Gamma_{X}} = \frac{1}{4}\Big( 2{d}_{\Gamma_{X}}+\chi(X)+\tau(X) 
\Big).
\end{eqnarray*}
Since $\chi(X)+\tau(X)=2( 1-{b}_{1}(X)+{b}^{+}(X))$, we obtain
\begin{eqnarray*}
{\cal I}_{\Gamma_{X}}= \frac{1}{2} \Big( 
{d}_{\Gamma_{X}}+{b}^{+}(X)-{b}_{1}(X)+1 \Big).
\end{eqnarray*}
It is clear that this equality implies the desired result.
\end{proof}

Proposition \ref{M spin}, Theorem \ref{spin SW inv}, Definition \ref{sasahira-inv} and Lemma \ref{simple-com} imply the following:
\begin{prop}\label{Sasahira-inva-prop}
Let $X$ be closed, oriented $4$-manifold with $b^+(X)>1$ and $\Gamma_X$ be a 
spin$^c$ structure on $X$. Fix an orientation $\cO$ on $\cH_g^1(X) \oplus 
\cH_g^+(X)$. Choose a square root $L$ of $\det_{\C} (\Ind D)$ and assume 
that the following conditions are satisfied:
\[
(*)_{1}\left\{
\begin{array}{ll}
{d}_{\Gamma_{X}}+{b}^{+}(X)-{b}_{1}(X) \equiv 3 \ (\bmod \ 4), & \\
\frak{S}^{ij}(\Gamma_X) \equiv 0 \ (\bmod \ 2)  \ \text{for all $i, j$}.
\end{array}
\right.
\]
Then the spin cobordism class of $\cM$ defines
\[
SW^{spin}(\Gamma_X, L) \in \Omega^{spin}_{{d}_{\Gamma_{X}}}, 
\]
where, again, ${d}_{\Gamma_{X}}$ is the dimension of ${\cal M}$.  
\end{prop}

Proposition \ref{Sasahira-inva-prop} and a method developed in \cite{sasa} enable us to prove the following which is the main result of this subsection. This is a real key to prove Theorem \ref{main-A}:
\begin{thm}\label{thm-spin-non-v}
For $m=1,2,3$, let $X_m$ be a closed oriented almost complex 4-manifold with 
${b}^{+}(X_m)>1$ and satisfying
\begin{eqnarray}\label{condition-11}
{b}^{+}(X_{m})-{b}_{1}(X_{m}) \equiv 3 \ (\bmod \ 4).
\end{eqnarray}
Let $\Gamma_{X_m}$ be a spin${}^{c}$ structure on $X_m$ which is induced
by the almost complex structure and assume that $SW_{X_{m}}(\Gamma_{X_{m}}) 
\equiv 1 \ (\bmod \ 2)$. Under Definition \ref{def-1}, moreover assume that 
the following condition holds for each $m$:
\begin{eqnarray}\label{condition-22}
\frak{S}^{ij}(\Gamma_{X_{m}}) \equiv 0 \ (\bmod \ 2) & \text{for all $i, j$}.
\end{eqnarray}
Put $X= \#_{m=1}^n X_m$ for $n=2, 3$, and let $\Gamma_X$ be a spin$^c$ structure on $X$ defined by $\Gamma_X = \#_{m=1}^n (\pm \Gamma_{X_m})$. Here $-\Gamma_{X_m}$ is the complex conjugation of $\Gamma_{X_m}$ and the sign $\pm$ are arbitrary.
Fix an orientation $\cO$ on $\cH_g^1(X) \oplus \cH_g^+(X)$ and choose a square root $L$ of $\det_{\C} (\Ind D)$. 
Then ${\cM}$ associated with the spin${}^{c}$ structure $\Gamma_{X}$ (see (\ref{pre-moduli})) defines a non-trivial spin cobordism class, i.e.,
\begin{eqnarray*}
{SW}^{spin}(\Gamma_{X}, L) \not\equiv 0 \in \Omega^{spin}_{n-1}.
\end{eqnarray*}
\end{thm}

Here, let us make the following observation:

\begin{rem} \label{remark almost}
The following holds: 
\begin{enumerate}
\item Let $X$ be a closed oriented smooth 4-manifold. We shall call a class $A \in H^2(X, {\mathbb Z})$ is an almost canonical class if 
\begin{equation} \label{almost complex}
A \equiv w_2(X) \pmod 2, \quad
A^2=2\chi(X)+3\tau(X).
\end{equation}
Such classes exist on $X$ if and only if $X$ admits an almost complex structure. More precisely, $A \in H^2(X, {\mathbb Z})$ is an almost canonical class if and only if there is an almost complex structure $J$ on $X$ which is compatible with the orientation, such that $A$ is just the first Chern class of the canonical bundle associated with the almost complex structure. 
\item
Let $X$ be a closed oriented smooth 4-manifold with $b^+(X) \geq 2$ and a spin$^c$ structure ${\Gamma}_{X}$.  Then the Seiberg-Witten invariant $SW_{X}(-\Gamma_{X})$ is equal to $SW_{X}(\Gamma_{X})$ up to sign. In particular, $SW_{X}(-\Gamma_{X})$ is also odd if $SW_{X}(\Gamma_{X})$ is odd. 
\item
Under the situation in Theorem \ref{thm-spin-non-v}, for each $m$, the condition $(*)_{1}$ in Proposition \ref{Sasahira-inva-prop} for the spin$^c$ structure $\Gamma_{X_m}$ induced by the almost complex structure is equivalent to 
\[
(*)_{2}\left\{
\begin{array}{ll}
{b}^{+}(X)-{b}_{1}(X) \equiv 3 \ (\bmod \ 4), & \\
\frak{S}^{ij}(\Gamma_{X_m}) \equiv 0 \ (\bmod \ 2)  \ \text{for all $i, j$}, 
\end{array}
\right.
\]
namely, $d_{\Gamma_{X_m}} = 0$ holds. In particular,  the condition $(*)_{1}$ for $-\Gamma_{X_m}$ also holds if  the condition $(*)_{1}$ for $\Gamma_{X_m}$ holds. 
\item
Under the situation in Theorem \ref{thm-spin-non-v}, the spin$^c$ structure $\Gamma_X$ on $X=\#_{m=1}^n X_m$ satisfies the condition $(*)_{1}$ in Proposition \ref{Sasahira-inva-prop}. In particular, $SW^{spin}(\Gamma_{X},L)$ is defined for $\Gamma_X$.
\item One can easily check that $d_{\Gamma_{X}}=n-1$ holds for the spin$^c$ structure $\Gamma_X = \#_{m=1}^n (\pm \Gamma_{X_m})$ by using the formula (\ref{SW-Mdim})
\item Under the same situation with Theorem \ref{thm-spin-non-v}, one can still define the spin cobordism Seiberg-Witten invariant for a connected sum $X= \#_{m=1}^n X_m$, where $n \geq 4$. However, it is known that the spin cobordism Seiberg-Witten invariant in this case must vanish. See Remark 3.16 in \cite{sasa}. 
\end{enumerate}

In fact, the first claim is known as a result of Wu. See \cite{hh}. 
The second claim is well known to the expert. See \cite{mor, nico} for the detail. 
The third claim follows from the fact that $c_{1}(\cL_{\Gamma_{X_m}})$ is the first Chern class of $X$ defined by the almost complex structure. This fact and the formulas (\ref{SW-Mdim}), (\ref{almost complex}) tell us that $d_{\Gamma_{X_m}} = 0$. This implies the desired result. It is not also hard to prove the fourth claim. Notice that both  $(*)_{1}$ in Proposition \ref{Sasahira-inva-prop} and $(*)_{2}$ for $\Gamma_{X_m}$ are equivalent to the $(*)$ in Proposition \ref{M spin} for $\Gamma_{X_m}$. Similarly, $(*)_{1}$ in Proposition \ref{Sasahira-inva-prop} for $\Gamma_X$ is also equivalent to $(*)$ in Proposition \ref{M spin} for $\Gamma_{X}$. Since each component $X_{m}$ of the connected sum $X=\#_{m=1}^n X_m$ has the spin$^c$ structure $\Gamma_{X_m}$ satisfies $(*)$ in Proposition \ref{M spin}, the spin$^c$ structure $\Gamma_X$ on the connected sum $X$ also satisfies $(*)$, where we used the definition of $\frak{S}^{ij}(\Gamma_{X})$ and the sum formula of the index of the Dirac operator. Hence the fourth claim now follows. 
\end{rem}

Let us review the definition of the Lie group spin structures before giving a proof of Theorem \ref{thm-spin-non-v}. See also \cite{kirb}. In the proof of Theorem \ref{thm-spin-non-v}, the moduli space can be identified with $S^1$ or $T^2$ and we shall prove that the spin structure on the moduli space is the Lie group spin structure. \par
Let $G$ be a $k$-dimensional compact, oriented, Lie group and suppose that we have an invariant Riemannian metric on $G$. When $k \geq 3$, let $P_{SO}$ be the orthonormal frame bundle of $TG$. When $k=1$ or $k=2$, let $P_{SO}$ be the orthonormal frame bundle of $TG \oplus \underline{\R}^{k'}$ for some $k' \geq 2$.
Fix an orthonormal basis $\{ e_1,\dots, e_k \}$ of the Lie algebra $\fg = T_e G$ compatible with the orientation.   Then we are able to define spin structures on $G$ in two ways. For $g \in G$, we denote the multiplications of $g$ from right and left by $R_g$, $L_g$. The derivative of $R_g$ and the basis $\{ e_1,\dots,e_k \}$ give a global trivialization of the tangent bundle $TG$. Also $L_g$ and $\{ e_1,\dots, e_k \}$ give us another trivialization of $TG$. 
These give us two trivializations
\[
\varphi_{R}:P_{SO} \stackrel{\cong}{\longrightarrow} G \times SO(k''),
\quad
\varphi_{L}:P_{SO} \stackrel{\cong}{\longrightarrow} G \times SO(k'').
\]
Here $k''$ is $k$ if $k \geq 3$ and $k+k''$ if $k=1$ or $k=2$.
The double covering $Spin (k'') \rightarrow SO(k'')$ and the trivializations $\varphi_{R}$, $\varphi_{L}$ give two double coverings of $P_{SO}$. Hence we obtain two spin structures on $G$. The isomorphism classes of these spin structures are independent of the choice of $\{ e_1,\dots, e_k \}$ since $SO(k)$ is path-connected. Of course, if the Lie group is commutative, these two spin structures are the same. Spin structures defined in this way are called Lie group spin structures. \par 
In the course of the  proof of Theorem \ref{thm-spin-non-v}, we shall use the following well-known result (for example, see \cite{kirb, furuta-k-m-2}):
\begin{thm} \label{1, 2 dim spin}
The spin cobordism groups $\Omega^{spin}_1$ and $\Omega_2^{spin}$ are isomorphic to $\Z_2$ and the generators are represented by the Lie group spin structures on $S^1$ and $T^2$ respectively.
\end{thm}

We are now in a position to give a proof of Theorem \ref{thm-spin-non-v}: \par
\vspace{3mm}
\noindent
{\it Proof of Theorem \ref{thm-spin-non-v}} First of all, we prove the theorem in the case where $n=2$. Let $f_m:V_m \rightarrow W_m$ be a finite dimensional approximation of the Seiberg-Witten map associated with $\Gamma_{X_m}$ or $-\Gamma_{X_m}$, where $m=1,2$. \par 
Suppose now that the rank of $(W_{m})_{\C}$ are even. Firstly, for simplicity, let us assume that $\cM_m=f_m^{-1}(0)/S^1$ consists of one point for all $m$. By Bauer's connected sum formula \cite{b-1} for stable cohomotopy Seiberg-Witten invariants, we may suppose that 
\begin{eqnarray*}
f=f_1 \times f_2 : V=V_1 \times V_2 \longrightarrow W=W_1 \times W_2
\end{eqnarray*}
is a finite dimensional approximation of the Seiberg-Witten map associated to $\Gamma_{X}:= ( \pm \Gamma_{X_{1}} ) \# (\pm  \Gamma_{X_{2}} )$ on $X=X_1 \# X_2$. Then $\cM_{X}:=f^{-1}(0)/S_d^1=f_1^{-1}(0) \times f_2^{-1}(0)/S^1_d$ is naturally identified with $S^1$. Here $S^1_d$ is the diagonal of $S^1 \times S^1$. Since $f=f_1 \times f_2$ is $S^1 \times S^1$-equivariant, we have a natural action of $S^1_q=(S^1 \times S^1) / S^1_d \cong S^1$ on $\cM_{X} \cong S^1$. Hence, $\cM_{X}$ admits a Lie group spin structure $\sigma^{Lie}_{\cM_{X}}$ coming from the action of $S^1_q$. \par
On the other hand, we have a  natural action of $S^1 \times S^1$ on $V_{irr}$ and this action also induces $S^1_q$-actions on both $\bar{V}:=V_{irr}/S^1_d$ and $E:=V_{irr} \times_{S^1_d} W$. Via natural actions of $S^1_q$ on both $T\bar{V}|_{\cM_{X}}$ and $E|_{\cM_{X}} = V_{irr} \times_{S^1_d} (W_1 \oplus W_2)|_{\cM_{X}}$, we are able to get Lie group spin structures on both $T\bar{V}|_{\cM_{X}}$ and $E|_{\cM_{X}}$. These Lie group spin structures together induce the Lie group spin structure $\sigma^{Lie}_{\cM_{X}}$ on $\cM_{X}$. \par 
On the other hand, we fix a square root $L$ of $\det_{\C} (\Ind D)$ associated with $\Gamma_{X}$. Then, as explained in subsection \ref{sasa-spin-1}, we have a spin structure $\sigma_{\bar{V}}$ on $\bar{V}:=V_{irr}/S^1_d$ induced by $L$, and we also have a natural spin structure $\sigma_{E}$ on $E:=V_{irr} \times_{S^1_d} W$.  As before, the restrictions of both $\sigma_{\bar{V}}$ and $\sigma_{E}$ to $\cM_{X}$ give rise to a spin structure $\sigma_{\cM_{X}}$ on $\cM_{X}$. \par Hence, we have two spin structures on $\cM_{X}$, i.e.,  $\sigma^{Lie}_{\cM_{X}}$, $\sigma_{\cM_{X}}$. Therefore, we are able to define two spin cobordism classes:
\begin{eqnarray*}
SW^{spin}(\Gamma_{X}, L)= [\cM_{X}, \sigma_{\cM_{X}}], \  [\cM_{X}, \sigma^{Lie}_{\cM_{X}}] \in \Omega_1^{spin} \cong \Z_2.   
\end{eqnarray*}
Here notice that the generator of $\Omega_1^{spin} \cong \Z_2$ is represented by the Lie group spin structure on $S^1 \cong \cM_{X}$ as mentioned in Theorem \ref{1, 2 dim spin}. Hence, if we are able to show that $ \sigma_{\cM_{X}}$ is isomorpshic to $\sigma^{Lie}_{\cM_{X}}$, the non-triviality of the spin cobordism Seiberg-Witten invariant follows:
\begin{eqnarray*}
SW^{spin}(\Gamma_{X}, L)= [\cM_{X}, \sigma_{\cM_{X}}]= [\cM_{X}, \sigma^{Lie}_{\cM_{X}}] \not=0 \in \Omega_1^{spin} \cong \Z_2.   
\end{eqnarray*}
Hence, our task is to show that $ \sigma_{\cM_{X}}$ is isomorpshic to $\sigma^{Lie}_{\cM_{X}}$. By the constructions of these spin structures, it is enough to prove the following lemma:

\begin{lem} \label{lem spin V E}
The restrictions of $\sigma_{\bar{V}}$, $\sigma_{E}$ to $\cM_{X}$ are the Lie group spin structures induced by the $S^1_q$-actions on $T\bar{V}|_{\cM_{X}}$, $E|_{\cM_{X}}$.
\end{lem}

\begin{proof}
Let $t_m \in \cT_m :=H^1(X_m,\R) / H^1(X_m,\Z)$ be the images of $\cM_m$ by the projections $V_m \rightarrow \cT_m$, where $m=1, 2$. Note that we assumed $\cM_m$ is one point. 
Let $V_{m, \C}$ be the complex part of $V_{m}$ and $V_{\R}$ be the real part of $V$,  and put
\[
\bar{V}' = 
\left\{ 
\big[ ( V_{1,\C} \backslash \{ 0 \}) \times (V_{2, \C} \backslash \{ 0 \} ) \big] \times_{\cT} V_{\R} \right\} / S^1_{d}
\subset
\bar{V}.
\]
We have the natural projection
\[
p:\bar{V}'_{t} \rightarrow \bar{V}'_t/S^1_q = \bar{V}_{1  t_1} \times \bar{V}_{2  t_2}.
\]
Here $t=(t_1, t_2) \in \cT$ and $\bar{V}'_{t}$, $\bar{V}_{m t_m}$ are the fibers over $t$, $t_{m}$.
Note that $\cM_X$ is included in $\bar{V}_t'$.

Since $\Gamma_{X_m}$ satisfies the condition $(*)$ in Proposition \ref{M spin}, $T \bar{V}_{m}$ are spin. 
We denote by $F$ the restriction of $(T\bar{V}_1 \times T\bar{V}_2) \oplus \underline{\R}$ to $\bar{V}_{1 t_1} \times \bar{V}_{2 t_2}$ and we fix a spin structure $\sigma_F$ on $F$. The restriction of $T\bar{V}$ to $\bar{V}_{t}'$ is naturally isomorphic to $p^* F$.
It is easy to see that $\bar{V}'_{t}$ is simply connected. Hence $H^1( \bar{V}'_t, \Z_2)$ is trivial. This means that spin structures on the restriction of $T \bar{V}$ to $\bar{V}'_t$ are unique (up to isomorphism). Therefore the restriction of the spin structure $\sigma_{\bar{V}}$ to $\bar{V}_{t}$ is isomorphic to $p^* \sigma_F$. Restricting the isomorphism to $\cM_X$, we get an isomorphism 
\[
\sigma_{\bar{V}}|_{\cM_X} \cong p^* \sigma_F|_{\cM_X}.
\]
The map $p$ is the projection from $\bar{V}_t'$ to $\bar{V}_t'/S^1_q$. Hence there is a natural lift of the $S^1_q$-action to $p^* \sigma_F$.
Therefore $\sigma_{\bar{V}}|_{\cM_X}$ is the spin structure induced by the $S^1_q$-action.

The proof that $\sigma_{E}|_{\cM_X}$ is induced by the $S^1_q$-action is similar. Put
\[
E_m = ( V_{m, \C} \backslash \{ 0 \}) \times_{ \cT_m } V_{m,\R}) \times_{S^1} W_m.
\]
Then $E|_{ \bar{V}_{t}' }$ is isomorphic to $p^* (E_1|_{ \bar{V}_{1, t_1} } \times E_2|_{ \bar{V}_{2, t_m} } )$. As before, we can show that the restriction $\sigma_{E}|_{\bar{V}_{t}'}$ is isomorphic to the pull-back of a spin structure on $E_1|_{ \bar{V}_{1, t_1} } \times E_2|_{ \bar{V}_{2, t_m} }$ since $H^1( \bar{V}_t',\Z_2)$ is trivial. Therefore $\sigma_{E}|_{\cM_X}$ is the spin structure induced by the $S^1_q$-action.
\end{proof}

Let $s$ be the section of $E$ induced by $f=f_1 \times f_2$. Since $f$ is $S^1 \times S^1$-equivariant, $s$ is $S^1_q$-equivariant. Recall that the spin structure on $\cM_X$ is defined by $\sigma_{\bar{V}}$, $\sigma_{E}$ and $s$. We have seen that these are compatible with the $S^1_q$-actions. Hence the spin structure on $\cM_X$ is the spin structure induced by the $S^1_q$-action. So the spin cobordism class of $\cM_X$ is non-trivial. \par
We assumed that $\cM_m$ consists of one point for all $m$. In general, the numbers $\# \cM_m$ of points in $\cM_m$ are odd, since the Seiberg-Witten invariants are odd by the assumption and Remark \ref{remark almost}. The quotient $\cM_X=f^{-1}(0)/S^1$ is a union of copies of $S^1$ and the number of components is just equal to the product
\[
(\# \cM_1) \cdot (\# \cM_2)
\]
of the numbers of elements of $\cM_m$.
From the above discussion, the spin structure on each $S^1$ is the Lie group spin structure. Therefore the spin cobordism class of $\cM_X$ is non-trivial in $\Omega_1^{spin} \cong \Z_2$. We have done the proof in the case where $n=2$. \par
The proof of the case $n=3$ is similar. Let $f_m$ be finite dimensional approximations associated with $\Gamma_{X_m}$ or $- \Gamma_{X_m}$  for $m=1, 2, 3$ and $f$ be a finite dimensional approximation associated with $\Gamma_X = \#_{m=1}^3 ( \pm \Gamma_{X_m})$. By Bauer's connected sum formula (\cite{b-1}), we may assume that 
\[
f=f_1 \times f_2 \times f_3:
V = V_1 \times V_2 \times V_3 
\longrightarrow 
W = W_1 \times W_2 \times W_3
\]
Hence $f$ is $T^3=S^1 \times S^1 \times S^1$-equivariant.
We write $\cM_X$ for the quotient $f^{-1}(0)/S^1_d$, where $S^1_d$ is the diagonal of $T^3$. Since $f_m^{-1}(0)$ are unions of copies of $S^1$, 
\[
\cM_X = \coprod^N T^2
\]
Here $N$ is the product 
\[
( \# \cM_1 ) \cdot ( \# \cM_2 ) \cdot ( \# \cM_3)
\]
of the numbers of elements of $\cM_m$.
The assumption that the Seiberg-Witten invariants are odd and Remark \ref{remark almost} mean that $N$ is odd. As in the case where $n=2$, what we must show is that the spin structure on each torus is the Lie group spin structure. To prove this, we need to show the following lemma which is proved in the similar way to Lemma \ref{lem spin V E}. Note that there are natural actions of $T^2_q:=T^3/S^1_d$ on $T \bar{V}$, $E$.

\begin{lem}
The restrictions of the spin structures on $T \bar{V}$, $E$ induced by $L$ to $\cM_X$ are the spin structures induced by the natural $T^2_q:=T^3/S^1_d$-actions on $TV|_{\cM_X}$, $E|_{\cM_X}$.
\end{lem}

We leave the detail of the proof of this lemma for the interested reader. \par
Since $f$ is $T^3$-equivariant, the sections $s$ of $E$ defined by $f$ is $T^2_q$-equivariant. The spin structures on $\cM_X$ is defined by the spin structures on $T \bar{V}$, $E$ and the section $s$ and these are compatible with the $T^2_q$-actions. Hence the spin structures on each component of $\cM_X$ is the Lie group spin structure. Therefore the spin cobordism class of $\cM_X$ is non-trivial by Fact \ref{1, 2 dim spin}. We have done the proof of Theorem \ref{thm-spin-non-v}.

\vspace{5mm}

In particular, Theorem \ref{thm-spin-non-v} implies the following result:

\begin{thm}\label{cor-B}
For $m=1,2,3$, let $X_m$ be 
\begin{itemize}
\item a closed oriented almost complex 4-manifold with ${b}_{1}({X}_{m})=0$, ${b}^{+}({X}_{m}) \equiv 3 \ (\bmod \ 4)$ and $SW_{X_{m}}(\Gamma_{X_{m}}) \equiv 1 \ (\bmod \ 2)$, where $\Gamma_{X_{m}}$ is a spin${}^c$ structure compatible with the almost complex structure, or 
\item a closed oriented almost complex 4-manifold with ${b}^{+}(X_m)>1$, $c_{1}(X_{m}) \equiv 0 \ (\bmod \ 4)$ and $SW_{X_{m}}(\Gamma_{X_{m}}) \equiv 1 \ (\bmod \ 2)$, where $\Gamma_{X_{m}}$ is a spin${}^c$ structure compatible with the almost complex structure. 
\end{itemize}

Let $X:=\displaystyle\#_{m=1}^{n}{X}_{m}$, here $n=2, 3$ and $\Gamma_X = \#_{m=1}^n ( \pm \Gamma_{X_m})$. Here the signs $\pm$ are arbitrary. Fix an orientation $\cO$ on $\cH_g^1(X) \oplus \cH_g^+(X)$ and choose a square root $L$ of $\det_{\C} (\Ind D)$. Then ${\cM}$ associated with the spin${}^{c}$ structure $\Gamma_{X}$ defines a non-trivial spin cobordism class:
\begin{eqnarray*}
{SW}^{spin}(\Gamma_{X}, L) \not\equiv 0 \in \Omega^{spin}_{n-1}.
\end{eqnarray*}
\end{thm}

\begin{proof}
By Theorem \ref{thm-spin-non-v}, it is sufficient to check that the conditions (\ref{condition-11}) and (\ref{condition-22}) hold for each $X_m$. \par
When $X_m$ is an almost complex $4$-manifold with $b_1(X_m)=0$, $b^+(X_m) \equiv 3 \bmod 4$,  (\ref{condition-11}), (\ref{condition-22}) hold clearly.

Let $X_m$ be an almost complex 4-manifold with $c_1(X_m) \equiv 0 \bmod 4$.
Since $X_m$ is spin, it follows from Rochlin's theorem that the signature of $X_m$ can be divided by $16$. The numerical index of spin-c Dirac operators is given by
\[
\frac{c_1^2(X_m) - \tau(X_m)}{8}.
\] 
Since $c_1^2(X_m)$ and $\tau(X_m)$ are divided by $16$, the index is even.
Hence Lemma \ref{simple-com} implies that (\ref{condition-11}) holds.
Moreover it follows from the definition of $\frak{S}^{ij}$ that (\ref{condition-22}) is satisfied.

\end{proof}

We give examples of 4-manifolds which satisfy the conditions in Theorem \ref{cor-B}.

\begin{cor}\label{sym-cor-1}
For $m=1,2,3$, let $X_{m}$ be 
\begin{itemize}
\item a product $\Sigma_{g} \times \Sigma_{h}$ of oriented closed surfaces of odd genus $g, h \geq 1$, or 
\item a closed symplectic 4-manifold with ${b}_{1}({X}_{m})=0$ and ${b}^{+}({X}_{m}) \equiv 3 \ (\bmod \ 4)$, or 
\item a primary Kodaira surface. 
\end{itemize}
Let $\Gamma_{X_m}$ be the spin$^c$ structures on $X_m$ induced by an almost complex structure compatible with the symplectic structure. And let $X$ be the connected sum $X:=\displaystyle\#_{m=1}^{n}{X}_{m}$ for $n=2, 3$ and we denote by $\Gamma_{X}$ a spin${}^{c}$ structure on $X$ defined by $\Gamma_X = \#_{m=1}^n  (\pm \Gamma_{X_m})$. Here the signs $\pm$ are arbitrary.
Fix an orientation $\cO$ on $\cH_g^1(X) 
\oplus \cH_g^+(X)$ and choose a square root $L$ of $\det_{\C} (\Ind D)$. 
Then ${\cM}$ associated with the spin${}^{c}$ structure 
$\Gamma_{X}$ defines a non-trivial spin cobordism class:
\begin{eqnarray*}
{SW}^{spin}(\Gamma_{X}, L) \not\equiv 0 \in \Omega^{spin}_{n-1}.
\end{eqnarray*}

\end{cor}

\begin{proof}
By Taubes's theorem in \cite{t-1}, $SW(\Gamma_{X_m})$ is equal to $\pm 1$, in particular it is odd.

Let $X_m$ be $\Sigma_g \times \Sigma_h$ with $g, h$ odd. Then $c_1( \cL_{\Gamma_{X_m} })$ is
\[
2(g-1) \alpha + 2(h-1) \beta \in H^2(X_m;\Z).
\]
Here $\alpha$, $\beta$ are the generators of $H^2(\Sigma_g;\Z)$, $H^2(\Sigma_h;\Z)$ respectively. Since $g$ and $h$ are odd, we have
\[
c_1( \cL_{ \Gamma_{X_m} }) \equiv 0 \bmod 4.
\]
Hence $(X_m, \Gamma_m)$ satisfies the second condition in Theorem \ref{cor-B}.

Let $X_m$ be a symplectic 4-manifold with $b^+ > 1$ with $b_1=0$, $b^+ \equiv 3 \bmod 4$. Then $(X_m, \Gamma_{X_m})$ clearly satisfies the first condition in Theorem \ref{cor-B}.

Let $X_m$ be a primary Kodaira surface. This is a symplectic 4-manifold with $c_1(X_m) = 0$. Hence $(X_m, \Gamma_m)$ satisfies the second condition in Theorem \ref{cor-B}.

\end{proof}

Here we give remarks on almost complex 4-manifolds with $c_1 = 0$, with $b^ + > 1$ and with $SW(\Gamma_X) \equiv 1 \bmod 2$.

\begin{rem} \label{remark-thm25}
Let $X$ be a closed oriented almost complex 4-manifold with vanishing first Chern class, with ${b}^{+}(X)>1$ and with $SW(\Gamma_{X}) \equiv 1 \bmod 2$.
Then this satisfies the second condition in Theorem  \ref{cor-B}. 
We are able to deduce that there are constraints to the Betti numbers of $X$ as follows. The 4-manifold $X$ must be spin and satisfies $c^2_{1}({\cal L}_{\Gamma_{X}})=2\chi(X) + 3\tau(X)=0$. Rochlin's theorem tells us that the signature $\tau(X)={b}^{+}(X) - {b}^{-}(X)$ of $X$ is divided by $16$. Hence, $\tau(X)=16k$ holds for some integer $k \in {\mathbb Z}$. Namely, we have ${b}^{-}(X)={b}^{+}(X)-16k$. By the direct computation, we have ${b}^{-}(X)={b}^{+}(X)-16k$. By the direct computation, we have $0 = 2\chi(X) + 3\tau(X)=4-4{b}_{1}(X)+5{b}^+(X)-{b}^{-}(X) = 4-4{b}_{1}(X)+5{b}^+(X)-({b}^{+}(X)-16k) = 4(1-{b}_{1}(X)+{b}^{+} + 4k)$. Hence we get 
\begin{eqnarray}\label{tau-1}
{b}_{1}(X)=1+{b}^{+} + 4k. 
\end{eqnarray}
The assumption that $SW_X(\Gamma_X) \equiv 1 \bmod 2$ implies that $b^+(X) \leq 3$. In fact, Bauer \cite{b-06} proved that  $SW_{M}(\Gamma_{M}) \equiv 0 \ (\bmod \ 2)$  for all almost complex 4-manifolds $M$ with vanishing first Chern class and with $b^+(M) \geq 4$ (cf. \cite{Li-1, Li-2}). Since we assume that ${b}^{+}(X)>1$ and $SW_{X_{m}}(\Gamma_{X}) \equiv 1 \ (\bmod \ 2)$, we have ${b}^{+}(X)=2$ or  ${b}^{+}(X)=3$.  Notice also that $\tau(X)=16k \leq 0$ since  ${b}^{+}(X) \leq 3$. \par
Suppose that ${b}^{+}(X)=2$. Then, (\ref{tau-1}) tells us that ${b}_{1}(X)=3 + 4k$. Since $k \leq 0$, we have ${b}_{1}(X)=3$ and $k=0$. Equivalently, we have ${b}_{1}(X)=3$ and $\tau(X)=0$. In particular, ${b}^{+}(X)-{b}_{1}(X)=2-3=-1 \equiv 3 \ (\bmod \ 4)$.  \par
On the other hand, suppose that ${b}^{+}(X)=3$. By (\ref{tau-1}), we have ${b}_{1}(X)=4 + 4k$. Since $\tau(X)=16k \leq 0$ and ${b}_{1}(X) \geq 0$, we have $k=0$ or $k=-1$. In the case of $k=0$, we have ${b}_{1}(X)=4$ and $\tau(X)=0$. Hence we have ${b}^{+}(X)-{b}_{1}(X)=3-4=-1 \equiv 3 \ (\bmod \ 4)$. On the other hand, in the case of $k=-1$, we have ${b}_{1}(X)=0$ and $\tau(X)=-16$. \par
We have proved that there are three cases. Indeed, we have $({b}^{+}(X), {b}_{1}(X), \tau(X))=(2,3,0)$, $(3,4,0)$, or $(3,0,-16)$. A primary Kodaira surface is an example of the first case $(2,3,0)$. A 4-torus is an example of the second case $(3,4,0)$. A $K3$ surface is an example of the third case $(3,0,-16)$. 
\end{rem}

\subsection{Non-vanishing theorem}\label{sub-32}

Proposition \ref{spin-BF} and Theorem \ref{thm-spin-non-v} immediately imply 
the following result which is nothing but Theorem \ref{main-A}:
\begin{thm}\label{thm-A}
For $m=1,2,3$, let $X_m$ be a closed oriented almost complex 4-manifold with 
${b}^{+}(X_m)>1$ and satisfying
\begin{eqnarray*}
{b}^{+}(X_{m})-{b}_{1}(X_{m}) \equiv 3 \ (\bmod \ 4).
\end{eqnarray*}
Let $\Gamma_{X_{m}}$ be a spin${}^{c}$ structure on $X_m$ which is induced 
by the almost complex structure and assume that $SW_{X_{m}}(\Gamma_{X_{m}}) 
\equiv 1 \ (\bmod \ 2)$. Under Definition \ref{def-1}, moreover assume that 
the following condition holds for each $m$:
\begin{eqnarray*}
\frak{S}^{ij}(\Gamma_{X_{m}}) \equiv 0 \bmod 2 & \text{for all $i, j$}.
\end{eqnarray*}
Let $X=\#_{m=1}^n X_m$ and $\Gamma_X = \#_{m=1}^n (\pm \Gamma_{X_m})$ for $n=2, 3$. Here the signs $\pm$ are arbitrary.
Then the connected sum $X$ has a non-trivial stable Seiberg-Witten invariant for the spin$^c$ structure $\Gamma_X$.

\end{thm}

Similarly, it is also clear that Proposition \ref{spin-BF} and Theorem 
\ref{cor-B} tell us that the following result holds, i.e., Theorem 
\ref{main-B}:

\begin{thm}\label{cor-1}
For $m=1,2,3$, let $X_m$ be 
\begin{itemize}
\item a product $\Sigma_{g} \times \Sigma_{h}$ of oriented closed surfaces of odd genus $h, g \geq 1$, or 
\item a closed oriented almost complex 4-manifold with ${b}_{1}({X}_{m})=0$, ${b}^{+}({X}_{m}) \equiv 3 \ (\bmod \ 4)$ and $SW_{X_{m}}(\Gamma_{X_{m}}) \equiv 1 \ (\bmod \ 2)$, where $\Gamma_{X_{m}}$ is a spin${}^c$ structure compatible with the almost complex structure,  or 
\item a closed oriented almost complex 4-manifold with vanishing first Chern class, ${b}^{+}(X_m)>1$ and $SW_{X_{m}}(\Gamma_{X_{m}}) \equiv 1 \ (\bmod \ 2)$, where $\Gamma_{X_{m}}$ is a spin${}^c$ structure compatible with the almost complex structure. 
\end{itemize}
And let $X = \#_{m=1}^n X_m$ and $\Gamma_{X}=\#_{m=1}^n (\pm \Gamma_{X_m})$ for $n=2, 3$. Here the signs $\pm$ are arbitrary. Then a connected sum $X$ has a non-trivial stable cohomotopy Seiberg-Witten invariant for the spin${}^{c}$ structure $\Gamma_{X}$. In particular, Conjecture \ref{conj-1} in the case where $\ell=2$ is true.
\end{thm}

Finally, Proposition \ref{spin-BF} and Corollary \ref{sym-cor-1} also imply

\begin{cor}\label{key-cor-1}
For $m=1,2,3$, let $X_{m}$ be 
\begin{itemize}
\item a product $\Sigma_{g} \times \Sigma_{h}$ of oriented closed surfaces of odd genus $h, g \geq 1$, or 
\item a closed symplectic 4-manifold with ${b}_{1}({X}_{m})=0$ and ${b}^{+}({X}_{m}) \equiv 3 \ (\bmod \ 4)$, or 
\item a primary Kodaira surface. 
\end{itemize}
Let $\Gamma_{X_m}$ be the spin$^c$ structures on $X_m$ induced by an almost complex structure compatible with the symplectic structure. And let $X= \#_{m=1}^n X_m$ and $\Gamma_{X}= \#_{m=1}^n ( \pm \Gamma_{X_m})$ for $n=2, 3$. Here the signs $\pm$ are arbitrary.  Then the connected sum $X$ has a non-trivial stable 
cohomotopy Seiberg-Witten invariant for the spin${}^{c}$ structure 
$\Gamma_{X}$. 
\end{cor}

\section{Various applications of Theorem \ref{main-A}}\label{sec-4}

In this section, we shall give various application of Theorem \ref{main-A} and Theorem \ref{main-B}.

\subsection{Decompositions and exotic smooth structures of connected sums of $4$-manifolds}
\label{subsec-exotic}

We will give proofs of Theorem \ref{thm decomp} and Theorem \ref{thm exotic-B}.
The key of the proofs is the following lemma: 
\begin{lem} \label{lem moduli dim}
Let $Z_{l}$ be closed, oriented, 4-manifolds with $b^+ > 0$ for $l=1, 2, \dots$ and $\Gamma_{l}$ be spin$^c$ structures on $Z_l$.
Put $Z=\#_{l=1}^N Z_l$, $\Gamma_{Z}:=\#_{l=1}^N \Gamma_{l}$ for some $N \geq 0$.
Assume that the moduli space $\cM^{SW}_{\Gamma_{Z}}(g, \eta)$ is not empty for all Riemannian metrics $g$ and self-dual 2-forms $\eta$ on $Z$. Then the virtual dimension of $\cM^{SW}_{\Gamma_Z}(g, \eta)$ is larger than or equal to $N-1$.
\end{lem}

\begin{proof}
To simplify notations we consider the case $N=3$.
Take points $z_1 \in Z_1$, $z_2, z_2' \in Z_2$, $z_3 \in Z_3$ and small open disks $D_1, D_2, D_2', D_3$ centered at these points. We put 
\[
\begin{split}
\hat{Z}_1 &= (Z_1 \backslash D_1) \cup S^3 \times \R_{ \geq 0}, \\
\hat{Z}_2 &= S^3 \times \R_{\leq 0} \cup ( Z_2 \backslash D_2 \cup D_2') \cup S^3 \times \R_{\geq 0}, \\
\hat{Z}_3 &= S^3 \times \R_{ \leq 0} \cup (Z_3 \backslash D_4),
\end{split}
\]
and for each $T>0$ we define 
\[
\begin{split}
\hat{Z}_1(T) &= \hat{Z}_1 \backslash S^3 \times [2T, \infty) \\
\hat{Z}_2(T) &= \hat{Z_2} \backslash \big( S^3 \times ( -\infty, -2T) \cup S^3 \times [2T, \infty) \big) \\
\hat{Z}_3(T) &= \hat{Z}_3 \backslash S^3 \times (-\infty, -2T].
\end{split}
\]

There is an identification
\[
\begin{array}{rccc}
\varphi_T: & S^3 \times (T, 2T) & \cong & S^3 \times (-2T, -T) \\
&(y, t) & \longmapsto & (y, t-3T).
\end{array}
\]
Gluing $\hat{Z}_1(T), \hat{Z}_2(T), \hat{Z}_3(T)$ by using $\varphi_{T}$, we have a manifold $Z(T)$ which is diffeomorphic to the connected sum $Z=\#_{l=1}^3 Z_l$. We take  Riemannian metrics $\hat{g}_l$ on $\hat{Z}_l$ which coincide with $g_{S^3} + dt^2$ on the ends. Here $g_{S^3}$ is the standard metric on $S^3$. These metrics naturally induce a Riemannian metric $g(T)$ on $Z(T)$.

Let $\cM^{SW}_{ \hat{ \Gamma }_l }(\hat{g}_l, \hat{\eta}_l)$ be the moduli spaces of monopoles on $\hat{Z}_l$ which converge to the trivial monopole on $S^3$ for all $l$. Here $\hat{\Gamma}_{ l }$ are  spin$^c$ structures on $\hat{Z}_l$ induced by $\Gamma_{l}$.
Since $b^+(\hat{Z}_l)>0$, we can choose self-dual 2-forms $\hat{\eta}_l$ such that $\cM^{SW}_{ \hat{\Gamma}_{ l }}(\hat{g}_l, \hat{\eta}_l)$ contain no reducible monopoles and are smooth of the expected dimension or empty. 
Moreover we may suppose that the supports of $\hat{\eta}_l$ do not intersect the ends of $\hat{Z}_l$.
(See Proposition 4.4.1 in \cite{nico}.)
Extending $\hat{\eta}_l$ trivially, we consider $\hat{\eta}_l$ as self-dual 2-forms on $Z(T)$, and we get a self-dual $2$-form $\eta(T) := \hat{\eta}_1 + \hat{\eta}_2 + \hat{\eta}_3$ on $Z(T)$.
Then we have
\begin{equation} \label{eq dim sum}
\dim \cM^{SW}_{\Gamma_Z}(g(T), \eta(T))=
\sum_{l=1}^3 \dim \cM^{SW}_{ \hat{\Gamma}_l }(\hat{g}_l, \hat{\eta}_l) + 2.
\end{equation}
Here $\dim \cM^{SW}_{\Gamma_Z}(g(T), \eta(T))$, $\dim \cM^{SW}_{ \hat{\Gamma}_l }(\hat{g}_l, \hat{\eta}_l)$ are the virtual dimensions of the moduli spaces.
This is derived from the excision principle of index of elliptic differential operators. We can also see this from the theory of gluing of monopoles. For large $T$, coordinates of $\cM^{SW}_{ \Gamma_Z }(g(T), \eta(T))$ are given by coordinates of $\cM^{SW}_{ \hat{\Gamma}_{l}}( \hat{g}_l, \hat{\eta}_l)$ and gluing parameters. Since $Z(T)$ has two necks, the space of gluing parameters is $U(1) \times U(1)$ and it is 2-dimensional. Hence we have the formula (\ref{eq dim sum}).

Let $\{ T^{\alpha} \}_{\alpha=1}^{\infty}$ be a sequence of positive numbers which diverges to infinity. 
By the assumption, $\cM^{SW}_{\Gamma_Z}(g(T^{\alpha}), \eta(T^{\alpha}))$ are non-empty.
Hence we can take elements $[\phi^{\alpha}, A^{\alpha}] \in \cM^{SW}_{ \Gamma_Z }(g(T^{\alpha}), \eta(T^{\alpha}))$ for all $\alpha$. There is a subsequence $\{ [\phi^{\alpha'}, A^{\alpha'}] \}_{\alpha'}$ which converges to some $([\phi^{\infty}_1, A^{\infty}_1], [\phi^{\infty}_2, A^{\infty}_2], [\phi^{\infty}_3, A^{\infty}_3]) \in \cM^{SW}_{\hat{\Gamma}_1}( \hat{g}_1, \hat{\eta}_1 ) \times \cM^{SW}_{\hat{\Gamma}_2}( \hat{g}_2, \hat{\eta}_2 )   \times  \cM^{SW}_{\hat{\Gamma}_3}( \hat{g}_3, \hat{\eta}_3 )$. In particular, $\cM^{SW}_{\hat{\Gamma}_l}( \hat{g}_l, \hat{\eta}_l )$ are non-empty. Since $\cM^{SW}_{\hat{\Gamma}_l}( \hat{g}_l, \hat{\eta}_l )$ are non-empty and smooth of the expected dimension, their virtual dimensions are at least zero. From (\ref{eq dim sum}), we have
\[
\dim \cM^{SW}_{\Gamma_Z}(g(T), \eta(T)) \geq 2.
\]
The virtual dimension of $\cM^{SW}_{\Gamma_Z}(g, \eta)$ is independent of $g, \eta$, so we have obtained the required result.
The proof for the general case is similar.
\end{proof}

We restate Theorem \ref{thm decomp} here.

\begin{thm}
Let $X_m$ be closed symplectic 4-manifolds with $c_1(X_m) \equiv 0 \ (\bmod \ 4)$ for $m=1, 2, 3$, and $X$ be a connected sum $\#_{m=1}^n X_m$, where $n=2, 3$. Then $X$ can not be written as a connected sum $\#_{m=1}^{N} Y_{m}$ with $b^+(Y_m) >0 $ and with $N>n$.
\end{thm}

\begin{proof}
It follows from Theorem \ref{main-B} that $\cM^{SW}_{\Gamma_X}(g,\eta)$ are non-empty for all $g, \eta$. Here $\Gamma_X$ is the connected sum of the spin$^c$ structures $\Gamma_{X_m}$ of $X_m$ induced by the almost complex structures. 
Suppose that $X$ has a decomposition $X = \#_{m=1}^N Y_{m}$ with $b^+(Y_m) > 0$ and $N > n$. Then Lemma \ref{lem moduli dim} implies
\[
\dim \cM^{SW}_{\Gamma_X}(g,\eta) \geq N - 1.
\]
On the other hand, the dimension of $\cM^{SW}_{\Gamma_X}(g,\eta)$ is $n-1$. Therefore we have $N \leq n$. Since we assumed that $N > n$, this is a contradiction.
\end{proof}

Next we show Theorem \ref{thm exotic-B}. More precisely we prove the following.

\begin{thm} \label{thm exotic}
Let $X$ be a closed, simply connected, non-spin, symplectic 4-manifold with $b^+ \equiv 3 \bmod 4$.
We denote $b^+(X)$, $b^-(X)$ by $p$, $q$, and put $Y= p \CP^2 \# q \overline{\CP}^2 $.
Let $X_m$ be almost complex 4-manifolds which have the properties in Theorem \ref{main-A} for $m=1, 2$, and let $X'$ be $X_1$ or the connected sum $\#_{m=1}^2 X_m$.
Then the connected sum $X \# X'$ is homeomorphic to $Y \# X'$, but not diffeomorphic to $Y \# X'$.
\end{thm}

\begin{proof}
First we show that $X \# X'$ is homeomorphic to $Y \# X'$.
Let $Q_X$ be the intersection form on $H_2(X,\Z)$.
If $q$ is zero, by Donaldson's theorem \cite{d}, $Q_X$ is isomorphic to the intersection form of $Y = p \CP^{2}$.
If $q$ is not zero, then $Q_{X}$ is an odd, indefinite,  unimodular form.
In general, an odd, indefinite, unimodular form is diagonalizable.
(See, for example, \cite{serre}.)
Hence $Q_X$ is isomorphic to the intersection form of $Y= p \CP^2 \#  q \overline{\CP}^2 $.
Therefore $X$ is homeomorphic to $Y$ by Freedman's theorem \cite{freedman}, and so $X \# X'$ is homeomorphic to $Y \# X'$.

Next we prove that $X \# X'$ is not diffeomorphic to $Y \# X'$.
Let $\Gamma$ be the spin$^c$ structure on $X \# X'$ induced by almost complex structures on $X$, $X_m$.
Then the dimension of moduli space associated with $\Gamma$ is $1$ or $2$ and it follows from Theorem \ref{main-A} that $\cM^{SW}_{\Gamma}(g, \eta)$ is non-empty for any $g, \eta$.

On the other hand, by the assumption that $b^+(X) \equiv 3 \bmod 4$, $Y$ is the connected sum $p \CP^2 \# q \overline{\CP}^2 $ with $p \geq 3$. Hence we can write $Y \# X'$ as $\#_{l=1}^4 Z_l$ with $b^+(Z_l) > 0$. 
Suppose that $X \# X'$ is diffeomorphic to $Y \# X'$. Then it follows from  Lemma \ref{lem moduli dim} that the dimension of the moduli space $\cM^{SW}_{\Gamma}(g, \eta)$ is at least $3$. Hence we have a contradiction since the dimension of the moduli space is 1 or 2.

\end{proof}

\subsection{Adjunction inequality}\label{subsec-adj}

Let $X_i$, $X$ be as in Theorem \ref{bau-1}. Then the non-vanishing result of the stable cohomotopy Seiberg-Witten invariants for $X$ implies that the Seiberg-Witten equations on $X$ have solutions for any Riemannian metrics and perturbations (see also subsection \ref{4.22} below). Hence we are able to apply the arguments of Kronheimer-Mrowka \cite{K-M} to $X$ and we obtain an estimate for the genus of embedded surfaces in $X$ as a corollary of Bauer's non-vanishing theorem:  
\begin{cor} \label{thm adjunction-bau}
Let $X_i$ and $X$ be as in Theorem \ref{bau-1} and let $\Gamma_{X_i}$ be a spin$^c$ structure on $X_i$ induced by the complex structure. 
Let $\Gamma_X$ be a spin$^c$ structure on $X$ defined by $\Gamma_{X}=\#_{i=1}^n (\pm \Gamma_{X_i})$.Here $n= 2, 3$ and the signs $\pm$ are arbitrary. Assume that $\Sigma$ is an embedded surface in $X$ with $[\Sigma] \cdot [\Sigma] \geq 0$ and $g(\Sigma)>0$, where $g(\Sigma)$ is the genus of $\Sigma$. Then, 
\[
[\Sigma] \cdot [\Sigma] - \< c_1(\cL_{\Gamma_{X}}), [\Sigma] \> \leq
2g(\Sigma) - 2.
\]
\end{cor}

The above inequality is called adjunction inequality. Notice that we assume that $b_{1}(X_i)=0$. In the case where $b_{1} \not=0$, Theorem \ref{main-A} implies the following result: 

\begin{thm} \label{thm adjunction}
Let $X_m$ and $X$ be as in Theorem \ref{main-A} and let $\Gamma_{X_m}$ be a spin$^c$ structure on $X_m$ induced by the complex structure. 
Let $\Gamma_X$ be a spin$^c$ structure on $X$ defined by $\Gamma_{X}=\#_{m=1}^n (\pm \Gamma_{X_m})$.
Here $n= 2, 3$ and the signs $\pm$ are arbitrary.
Assume that $\Sigma$ is an embedded surface in $X$ with $[\Sigma] \cdot [\Sigma] \geq 0$ and $g(\Sigma)>0$. Then, 
\[
[\Sigma] \cdot [\Sigma] - \< c_1(\cL_{\Gamma_{X}}), [\Sigma] \> \leq
2g(\Sigma) - 2.
\]

\end{thm}

See also \cite{furuta-k-m-2, furuta-k-m-3, furuta-k-m-1, sasa} for related results. In particular, we notice that Theorem \ref{thm adjunction} never follow form adjunction inequalities proved in \cite{furuta-k-m-2, furuta-k-m-3, furuta-k-m-1, sasa}. \par
By Theorem \ref{thm adjunction} and Corollary \ref{key-cor-1}, we obtain
\begin{cor}
For $m=1,2,3$, let $X_{m}$ be 
\begin{itemize}
\item a product $\Sigma_{h} \times \Sigma_{g}$ of oriented closed surfaces of odd genus $h, g \geq 1$, or 
\item a closed symplectic 4-manifold with ${b}_{1}({X}_{m})=0$ and ${b}^{+}({X}_{i}) \equiv 3 \ (\bmod \ 4)$, or 
\item a primary Kodaira surface. 
\end{itemize}
let $\Gamma_{X_m}$ be a spin$^c$ structure on $X_m$ induced by the complex structure. 
Let $\Gamma_X$ be a spin$^c$ structure on $X$ defined by $\Gamma_{X}=\#_{m=1}^n (\pm \Gamma_{X_m})$.
Here $n= 2, 3$ and the signs $\pm$ are arbitrary.
Assume that $\Sigma$ is an embedded surface in $X$ with $[\Sigma] \cdot [\Sigma] \geq 0$ and $g(\Sigma)>0$. Then, 
\[
[\Sigma] \cdot [\Sigma] - \< c_1(\cL_{\Gamma_{X}}), [\Sigma] \> \leq
2g(\Sigma) - 2.
\]
\end{cor}

\subsection{Monopole classes and curvature bounds}\label{4.22}

In this subsection, for the convenience of the reader, following a recent beautiful article \cite{leb-17} of LeBrun, we shall recall firstly curvature estimates arising Seiberg-Witten monopole equations in terms of the convex hull of the set of all monopole classes on 4-manifolds. We shall use these estimates in the rest of this article. The main results in this subsection are Theorems \ref{mono-key-bounds} and \ref{bf-ricci} below. \par
First of all, let us recall the definition of monopole class \cite{kro, leb-11, ishi-leb-2, leb-17}.
\begin{defn}\label{ishi-leb-2-key}
Let $X$ be a closed oriented smooth 4-manifold with $b^+(X) \geq 2$. An element $\frak{a} \in H^2(X, {\mathbb Z})$/torsion $\subset H^2(X, {\mathbb R})$ is called monopole class of $X$ if there exists a spin${}^c$ structure $\Gamma_{X}$ with 
\begin{eqnarray*}
{c}^{\mathbb R}_{1}({\cal L}_{\Gamma_{X}}) = \frak{a} 
\end{eqnarray*} 
which has the property that the corresponding Seiberg-Witten monopole equations have a solution for every Riemannian metric on $X$. Here ${c}^{\mathbb R}_{1}({\cal L}_{\Gamma_{X}})$ is the image of the first Chern class ${c}_{1}({\cal L}_{\Gamma_{X}})$ of the complex line bundle ${\cal L}_{\Gamma_{X}}$ in $H^2(X, {\mathbb R})$. We shall denote the set of all monopole classes on $X$ by ${\frak C}(X)$.
\end{defn} 

Crucial properties of the set ${\frak C}(X)$ are summarized as follow \cite{leb-17, ishi-leb-2}:

\begin{prop}[\cite{leb-17}]\label{mono}
Let $X$ be a closed oriented smooth 4-manifold with $b^+(X) \geq 2$. Then ${\frak C}(X)$ is a finite set. Moreover ${\frak C}(X) = -{\frak C}(X)$ holds, i.e., $\frak{a} \in H^2(X, {\mathbb R})$ is a monopole class if and only if $-\frak{a} \in H^2(X, {\mathbb R})$ is a monopole class, too. 
\end{prop}

Recall that, for any subset $W$ of a real vector space $V$, one can consider the convex hull ${\bf{Hull}}(W) \subset V$, meaning the smallest convex subset of $V$ containing $W$. Then, Proposition \ref{mono} implies 

\begin{prop} [\cite{leb-17}] \label{mono-leb}
Let $X$ be a closed oriented smooth 4-manifold with $b^+(X) \geq 2$. Then the convex hull ${\bf{Hull}}({\frak C}(X)) \subset H^2(X, {\mathbb R})$ of ${\frak C}(X)$ is  compact, and symmetric, i.e., ${\bf{Hull}}({\frak C}(X)) = -{\bf{Hull}}({\frak C}(X))$. 
\end{prop}

Since ${\frak C}(X)$ is a finite set, we are able to write as ${\frak C}(X)=\{{\frak{a}}_{1},{\frak{a}}_{2}, \cdots, {\frak{a}}_{n} \}$. The convex hull ${\bf{Hull}}({\frak C}(X))$ is then expressed as follows:
\begin{eqnarray}\label{hull}
{\bf{Hull}}({\frak C}(X))= \{ \sum^{n}_{i=1}t_{i} {\frak{a}}_{i} \ | \ t_{i} \in [0,1], \ \sum^{n}_{i=1}t_{i}=1 \}. 
\end{eqnarray}
Notice that the symmetric property tells us that ${\bf{Hull}}({\frak C}(X))$ contains the zero element. On the other hand, consider the following self-intersection function:
\begin{eqnarray*}
{\cal Q} : H^2(X, {\mathbb R}) \rightarrow {\mathbb R}
\end{eqnarray*}
which is defined by $x \mapsto x^2:=\< x \cup x, [X] \>$, where $[X]$ is the fundamental class of $X$. Since this function ${\cal Q}$ is a polynomial function and hence is a continuous function on $H^2(X, {\mathbb R})$. We can therefore conclude that the restriction ${\cal Q} |_{{\bf{Hull}}({\frak C}(X))}$ to the compact subset ${\bf{Hull}}({\frak C}(X))$ of $H^2(X, {\mathbb R})$ achieves its maximum. Then we introduce 
\begin{defn}[\cite{leb-17}]\label{beta}
Suppose that $X$ is a closed oriented smooth 4-manifold with $b^+(X) \geq 2$. Let ${\bf{Hull}}({\frak C}(X)) \subset H^2(X, {\mathbb R})$ be the convex hull of the set ${\frak C}(X)$ of all monopole classes on $X$. If ${\frak C}(X) \not= \emptyset$, define
\begin{eqnarray*}
{\beta}^2(X):= \max \{ {\cal Q}(x):=x^2 \ | \ x \in {\bf{Hull}}({\frak C}(X)) \}.
\end{eqnarray*}
On the other hand, if ${\frak C}(X) = \emptyset$ holds, define simply as ${\beta}^2(X):=0$.
\end{defn}
Notice again that ${\bf{Hull}}({\frak C}(X))$ contains the zero element if $\frak{C}(X)$ is not empty. Hence the above definition with this fact implies that ${\beta}^2(X) \geq 0$ holds. \par
The existence of the monopole classes gives a constraint on the existence of Riemannian metrics of some type:
\begin{prop}[\cite{leb-17}]\label{beta-ine-key-0}
Let $X$ be a closed oriented smooth 4-manifold with ${b}^+(X) \geq 2$. If there is a non-zero monopole class $\frak{a} \in H^2(X, {\mathbb R})-\{0\}$, then $X$ cannot admit Riemannian metric $g$ of scalar curvature $s_{g} \geq 0$. 
\end{prop}

On the other hand, it is also known that the existence of monopole classes implies the following family of integral inequalities: 
\begin{thm}[\cite{leb-17}]\label{beta-ine-key}
Suppose that $X$ is a closed oriented smooth 4-manifold with $b^+(X) \geq 2$. Then any Riemannian metric $g$ on $X$ satisfies the following curvature estimates:
\begin{eqnarray*}
{\int}_{X}{{s}^2_{g}}d{\mu}_{g} \geq {32}{\pi}^{2}\beta^2(X), 
\end{eqnarray*}
\begin{eqnarray*}
{\int}_{X}\Big({s}_{g}-\sqrt{6}|W^{+}_{g}|\Big)^2 d{\mu}_{g} \geq 72{\pi}^{2}\beta^2(X), 
\end{eqnarray*}
where $s_{g}$ and $W^{+}_{g}$ denote respectively the scalar curvature and the self-dual Weyl curvature of $g$. If $X$ has a non-zero monopole class and, moreover, equality occurs in either the first or the second estimate if and only if $g$ is a K{\"{a}}hler-Einstein metric with negative scalar curvature. 
\end{thm}
Notice that if $X$ has no monopole class, we define as $\beta^2(X):=0$ (see Definition \ref{beta} above). On the other hand, notice also that the left-hand side of these two curvature estimates in Theorem \ref{beta-ine-key} is always non-negative. Therefore the result of Theorem \ref{beta-ine-key} holds trivially when $X$ has no monopole class. Thus, the main problem is to detect the existence of monopole classes. It is known that the non-triviality of the stable cohomotopy Seiberg-Witten invariant implies the existence of monopole classes:
\begin{prop}[\cite{ishi-leb-2}]\label{BF-mono}
Let $X$ be a closed oriented smooth 4-manifold with ${b}^+(X) \geq 2$ and a spin$^c$ structure $\Gamma_{X}$. Suppose that $BF_{X}(\Gamma_{X})$ is non-trivial. Then ${c}^{\mathbb R}_{1}({\cal L}_{\Gamma_{X}})$ is a monopole class. 
\end{prop}
On the other hand, it is also known that the Bauer's connected sum formula \cite{b-1} implies
\begin{prop}[\cite{ishi-leb-2, b}]\label{BF-mono-ne}
Let $X$ be a closed oriented smooth 4-manifold with $b^+(X) \geq 2$ and a spin$^c$ structure $\Gamma_{X}$. Suppose that $BF_{X}(\Gamma_{X})$ is non-trivial. Let $N$ be a closed oriented smooth 4-manifold with ${b}^+(N)=0$ and let $\Gamma_{N}$ be any spin${}^c$ structure on $N$ with ${c}^{2}_{1}({\cal L}_{\Gamma_{N}})=-b_{2}(N)$. Then the stable cohomotopy Seiberg-Witten invariant is also non-trivial for the spin${}^c$ structure $\Gamma_{X} \# \Gamma_{N}$ on the connected sum $X \# N$.
\end{prop}

Theorem \ref{thm-A} with Propositions \ref{BF-mono} and \ref{BF-mono-ne} tells us that the following holds (cf. Proposition 10 in \cite{ishi-leb-2}):
\begin{thm}\label{con-mono}
Let ${X}_{m}$ be as in Theorem \ref{thm-A} and suppose that $N$ is a closed oriented smooth 4-manifold with $b^{+}(N)=0$ and let $E_{1}, E_{2}, \cdots, E_{k}$ be a set of generators for $H^2(N, {\mathbb Z})$/torsion relative to which the intersection form is diagonal. 
(We can take such generators by Donaldson's theorem \cite{d}.)
Then, for any $n=2,3$, 
\begin{eqnarray}\label{mono-cone}
\sum^{n}_{m=1} \pm {c}_{1}(X_{m}) + \sum^{k}_{r=1} \pm{E}_{r} 
\end{eqnarray}
is a monopole class of $M:=\Big(\#^{n}_{m=1}{X}_{m} \Big) \# N$, where ${c}_{1}(X_{m})$ is the first Chern class of the canonical bundle of the almost-complex 4-manifold $X_{m}$ and the $\pm$ signs are arbitrary, and are independent of one another. 
\end{thm}
As a corollary of Theorem \ref{con-mono}, we are able to get 
\begin{cor}\label{mono-cor}
Let ${X}_{m}$, $N$ and $M$ be as in Theorem \ref{con-mono} above. Then, for any $n=2,3$, 
\begin{eqnarray}\label{monopole-123446}
\beta^2(M) \geq  \sum^{n}_{m=1}{c}^2_{1}(X_{m}). 
\end{eqnarray}
\end{cor} 

\begin{proof}
First of all, by the very definition, we have 
\begin{eqnarray*}
{\beta}^2(M):= \max \{ {\cal Q}(x):=x^2 \ | \ x \in {\bf{Hull}}({\frak C}(M)) \}.
\end{eqnarray*}
On the other hand, by (\ref{mono-cone}), we particularly have the following two monopole classes of $M$:
\begin{eqnarray*}
{\frak{a}}_{1}:=\sum^{n}_{m=1} {c}_{1}(X_{m}) + \sum^{k}_{r=1} {E}_{r}, \ {\frak{a}}_{2}:=\sum^{n}_{m=1} {c}_{1}(X_{m}) - \sum^{k}_{r=1} {E}_{r}. 
\end{eqnarray*}
By (\ref{hull}), we are able to conclude that
\begin{eqnarray*}
\sum^{n}_{m=1} {c}_{1}(X_{m})= \frac{1}{2}{\frak{a}}_{1}+\frac{1}{2}{\frak{a}}_{2} \in {\bf{Hull}}({\frak C}(M)). 
\end{eqnarray*}
We therefore obtain the desired bound:
\begin{eqnarray*}
{\beta}^2(M) \geq \Big( \sum^{n}_{m=1} {c}_{1}(X_{m})\Big)^2=\sum^{n}_{m=1}{c}^2_{1}(X_{m}). 
\end{eqnarray*}
\end{proof}

Theorem \ref{beta-ine-key} and Corollary \ref{mono-cor} imply an important result for our purpose:
\begin{thm}\label{mono-key-bounds}
Let ${X}_{m}$ be as in Theorem \ref{thm-A} and suppose that $N$ is a closed oriented smooth 4-manifold with $b^{+}(N)=0$. Consider a connected sum $M:=\Big(\#^{n}_{m=1}{X}_{m} \Big) \# N$, where $n=2,3$. Then any Riemannian metric $g$ on $M$ satisfies the following curvature estimates:
\begin{eqnarray}\label{weyl-leb-sca-1}
{\int}_{M}{{s}^2_{g}}d{\mu}_{g} \geq {32}{\pi}^{2}\sum^{n}_{m=1}{c}^2_{1}(X_{m}), 
\end{eqnarray}
\begin{eqnarray}\label{weyl-leb-sca-2}
{\int}_{M}\Big({s}_{g}-\sqrt{6}|W^{+}_{g}|\Big)^2 d{\mu}_{g} \geq 72{\pi}^{2}\sum^{m}_{n=1}{c}^2_{1}(X_{m}). 
\end{eqnarray}
\end{thm}
On the other hand, recall that, for any closed oriented Riemannian 4-manifold $(X, g)$, we have the following Gauss-Bonnet-type formula \cite{hit, be, thor}:
\begin{eqnarray}\label{ein-gauss}
2\chi(X)+3\tau(X)=\frac{1}{4\pi^{2}}\int_{X}\Big(2|W^{+}_{g}|^{2} + \frac{s^{2}_{g}}{24} - \frac{| \stackrel {\circ}{r}_{g} |^2}{2}\Big)d{\mu}_{g}, 
\end{eqnarray}
where $W^{+}_{g}$ is the self-dual part of the Weyl curvature of $g$ and $\stackrel {\circ}{r}_{g}$ is the trace-free part of the Ricci curvature $r_{g}$ of $g$. We also have 
\begin{eqnarray*}
{\int}_{X}|r_{g}|^2 d{\mu}_{g} = {\int}_{X} \Big(\frac{s^{2}_{g}}{4} + {|\stackrel {\circ}{r}_{g}|^2} \Big)d{\mu}_{g}.  
\end{eqnarray*} 
We therefore have the following equality for any Riemannian metric $g$ on $X$: 
\begin{eqnarray}\label{r-w-1}
{\int}_{X}|r_{g}|^2 d{\mu}_{g} = {\int}_{X} \Big(\frac{s^{2}_{g}}{3} + 4|W^{+}_{g}|^{2} \Big)d{\mu}_{g} - 8{\pi}^2 \Big( 2\chi(X) + 3\tau(X) \Big).  
\end{eqnarray} 
On the other hand, the Cauchy-Schwarz and triangle inequalities \cite{leb-11} tell us that the following inequality holds:
\begin{eqnarray}\label{r-w-2}
{\int}_{X} \Big(\frac{s^{2}_{g}}{3} + 4|W^{+}_{g}|^{2} \Big)d{\mu}_{g} \geq \frac{2}{9}{\int}_{M}\Big({s}_{g}-\sqrt{6}|W^{+}_{g}|\Big)^2 d{\mu}_{g}. 
\end{eqnarray} 
By using (\ref{weyl-leb-sca-2}), (\ref{r-w-1}) and (\ref{r-w-2}), we are able to prove

\begin{thm}\label{bf-ricci}
Let ${X}_{m}$ be as in Theorem \ref{thm-A} and suppose that $N$ is a closed oriented smooth 4-manifold with $b^{+}(N)=0$. Consider a connected sum $M:=\Big(\#^{n}_{m=1}{X}_{m} \Big) \# N$, where $n=2,3$. Then any Riemannian metric $g$ on $M$ satisfies 
\begin{eqnarray}\label{r-w-3}
{\int}_{M}|r_{g}|^2 d{\mu}_{g} \geq 8{\pi}^{2} \Big[4n-\Big(2\chi(N)+3\tau(N)\Big)+\sum^{n}_{m=1}{c}^2_{1}(X_{m}) \Big]. 
\end{eqnarray} 
\end{thm}
\begin{proof}
 A direct computation tells us that
\begin{eqnarray*}
2\chi(M)+3\tau(M) &=& -4{n}+\Big( 2\chi(N)+3\tau(N) \Big)+\sum_{m=1}^{n}\Big( 2\chi(X_{m})+3\tau(X_{m}) \Big) \\
      &=& -4{n}+\Big( 2\chi(N)+3\tau(N) \Big)+\sum_{m=1}^{n}{c}^2_{1}(X_{m}).               
\end{eqnarray*}
This formula, (\ref{r-w-1}) and (\ref{r-w-2}) imply
\begin{eqnarray*}
{\int}_{M}|r_{g}|^2 d{\mu}_{g} &\geq& \frac{2}{9}{\int}_{M}\Big({s}_{g}-\sqrt{6}|W^{+}_{g}|\Big)^2 d{\mu}_{g} \\
&-& 8{\pi}^2\Big[-4{n}+\Big( 2\chi(N)+3\tau(N) \Big)+\sum_{m=1}^{n}{c}^2_{1}(X_{m})  \Big]. 
\end{eqnarray*}
By this bound and (\ref{weyl-leb-sca-2}), we have 
\begin{eqnarray*}
{\int}_{M}|r_{g}|^2 d{\mu}_{g} &\geq& \frac{2}{9} \Big(72{\pi}^{2}\sum^{m}_{n=1}{c}^2_{1}(X_{m}) \Big) \\
&+& 8{\pi}^2\Big[4{n}-\Big( 2\chi(N)+3\tau(N) \Big)-\sum_{m=1}^{n}{c}^2_{1}(X_{m})  \Big]. 
\end{eqnarray*}
This immediately implies the desired bound. 
\end{proof}

\subsection{Computation of several differential geometric invariants}\label{4.2}

In this section, we shall compute the values of  several differential geometric invariants. The main results in this subsection are Theorems \ref{compu-scalar-invariant} and \ref{compu-Ricci-invariant} stated below. \par
As one of interesting differential geometric invariants, there exists a natural diffeomorphism invariant arising from a variational problem for the total scalar curvature of Riemannian metrics on a closed oriented Riemannian manifold $X$ of dimension $n\geq 3$. As was conjectured by Yamabe \cite{yam}, and later proved by Trudinger, Aubin, and Schoen \cite{aubyam,lp,rick,trud}, every conformal class on a smooth compact manifold contains a Riemannian metric of constant scalar curvature. Hence, for each conformal class $[g]=\{ vg ~|~v: X\to {\Bbb R}^+\}$, we are able to consider an associated number $Y_{[g]}$, which is so called Yamabe constant of the conformal class $[g]$ and defined by 
\begin{eqnarray*}
Y_{[g]} = \inf_{h \in [g]}  \frac{\int_X 
s_{{h}}~d\mu_{{h}}}{\left(\int_X 
d\mu_{{h}}\right)^{\frac{n-2}{n}}}, 
\end{eqnarray*}
where $s_{h}$ is the scalar curvature of the metric $h$ and $d\mu_{{h}}$ is the volume form with respect to the metric $h$. The Trudinger-Aubin-Schoen theorem tells us that this number is actually realized as the constant scalar curvature of some unit volume metric in the conformal class $[g]$. Then, Kobayashi \cite{kob} and Schoen \cite{sch} independently introduced the following interesting invariant of $X$:
\begin{eqnarray*}
{\mathcal Y}(X) = \sup_{\mathcal{C}}Y_{[g]}, 
\end{eqnarray*}
where $\mathcal{C}$ is the set of all conformal classes on $X$. This is now commonly known as the Yamabe invariant of $X$. It is known that ${\mathcal Y}(X) \leq 0$ if and only if $X$ does not admit a metric of positive scalar curvature. \par
The Yamabe invariant ${\mathcal Y}(X)$ is closely related with the diffeomorphism invariant defined \cite{BCG, leb-11} by 
\begin{eqnarray}\label{scalar-yama-def}
{\mathcal I}_{s}(X):=\inf_{g \in {\mathcal R}_{X}} {\int}_{X}|s_{g}|^{n/2}d{\mu}_{g}, 
\end{eqnarray}
where  the space of all Riemannian metrics on $X$ is denoted  by ${\cal R}_{X}$. It is known that the invariant ${\mathcal I}_{s}$ vanishes for every simply connected $n$-manifold with $n \geq 5$. Moreover, for every closed  $n$-manifold with $n \geq 3$ admitting non-negative scalar curvature i.e., ${\mathcal Y}(X) \geq 0$, we have the following \cite{leb-1}: 
\begin{eqnarray}\label{vani-scal}
{\mathcal I}_{s}(X)=0. 
\end{eqnarray}
 On the other hand,  Proposition 12 in \cite{ishi-leb-2} tells us that the following equality holds whenever ${\mathcal Y}(X) \leq 0$:
\begin{eqnarray}\label{s-y-1}
{\mathcal I}_{s}(X)=|{\mathcal Y}(X)|^{n/2}.
\end{eqnarray}
Hence, the invariant ${\mathcal I}_{s}(X)$ of a closed 4-manifold $X$ with ${\mathcal Y}(X) \leq 0$ is just
\begin{eqnarray}\label{scalar-yama}
{\mathcal I}_{s}(X)=|{\mathcal Y}(X)|^2=\inf_{g \in {\mathcal R}_{X}} {\int}_{X}s^{2}_{g}d{\mu}_{g}. 
\end{eqnarray}
On the other hand, consider the following quantity:
\begin{eqnarray}\label{scalar-yama-def-2}
{\mathcal K}(X):=\sup_{g \in {\mathcal R}_{X}}\Big( (\min_{x \in X}{s}_{g})(vol_{g})^{n/2} \Big), 
\end{eqnarray}
where $vol_{g}={\int}_{X}d{\mu}_{g}$ is the total volume with respect to $g$. Kobayashi \cite{kob} pointed out that the following equality holds whenever ${\mathcal Y}(X) \leq 0$:
\begin{eqnarray}\label{scalar-yama-3}
{\mathcal K}(X)={\mathcal Y}(X).
\end{eqnarray}
It is now clear that the scalar curvature bound (\ref{weyl-leb-sca-1}) in Theorem \ref{mono-key-bounds}, (\ref{scalar-yama}) and (\ref{scalar-yama-3}) imply
\begin{prop}\label{prop-scal}
Let ${X}_{m}$, $N$ and $M$ be as in Theorem \ref{mono-key-bounds}. Then, for $n=2,3$, 
\begin{eqnarray*}
{\mathcal I}_{s}(M)=|{\mathcal Y}(M)|^2=|{\mathcal K}(M)|^2 \geq {32}{\pi}^{2}\sum^{n}_{m=1}{c}^2_{1}(X_{m}). 
\end{eqnarray*}
\end{prop}
\begin{proof}
Notice that there is nothing to prove when $\sum^{n}_{m=1}{c}^2_{1}(X_{m}) \leq 0$. Hence we may assume that $\sum^{n}_{m=1}{c}^2_{1}(X_{m}) >0$. Then the connected sum $M$ has non-zero monopole classes by Theorem \ref{con-mono}. This fact and Proposition \ref{beta-ine-key-0} force that the connected sum $M$ cannot admit any Riemannian metric $g$ of scalar curvature $s_{g} \geq 0$. Thanks to a result of Kobayashi \cite{kob}, it is known that, for any closed $n$-manifold $X$ with $n \geq 3$ has ${\mathcal Y}(X)>0$ if and only if $X$ admits a metric of positive scalar curvature. Hence, we are able to conclude that the connected sum $M$ in question must satisfy ${\mathcal Y}(M) \leq 0$.  This fact, (\ref{weyl-leb-sca-1}) in Theorem \ref{mono-key-bounds}, (\ref{scalar-yama}) and (\ref{scalar-yama-3}) imply the desired result:
\end{proof}

On the other hand,  Proposition 13 in \cite{ishi-leb-2} tells us that 
\begin{eqnarray}\label{bound-conn}
{\mathcal I}_{s}(X \# Y) \leq {\mathcal I}_{s}(X) + {\mathcal I}_{s}(Y), 
\end{eqnarray}
where $X$ and $Y$ are any closed smooth manifolds with $n \geq 3$. Proposition \ref{prop-scal}, (\ref{vani-scal}) and (\ref{bound-conn}) imply the following result which can be seen as a generalization of both Theorems A and B in \cite{ishi-leb-2} to the case where $b_{1} \not=0$:  
\begin{main}\label{compu-scalar-invariant}
Let $N$ be a closed oriented smooth 4-manifold with $b^{+}(N)=0$ and with a Riemannian metric of non-negative scalar curvature. For $m=1,2,3$, let $X_m$ be a minimal K{\"{a}}hler surface with ${b}^{+}(X_m)>1$ and satisfying
\begin{eqnarray*}
{b}^{+}(X_{m})-{b}_{1}(X_{m}) \equiv 3 \ (\bmod \ 4).
\end{eqnarray*}
Let $\Gamma_{X_{m}}$ be a spin${}^{c}$ structure on $X_m$ which is induced by the K{\"{a}}hler structure. Under Definition \ref{def-1}, moreover assume that the following condition holds for each $m$:
\begin{eqnarray*}
\frak{S}^{ij}(\Gamma_{X_{m}}) \equiv 0 \bmod 2 & \text{for all $i, j$}.
\end{eqnarray*}
 Then, for $n=2,3$, a connected sum $M:=(\#^{n}_{m=1}{X}_{m} ) \# N$ satisfies
\begin{eqnarray}\label{scalar-com}
{\mathcal I}_{s}(M) = |{\mathcal Y}(M)|^2 = |{\mathcal K}(M)|^2 = {32}{\pi}^{2}\sum^{n}_{m=1}{c}^2_{1}(X_{m}). 
\end{eqnarray}
In particular, the Yamabe invariant of $M$ is given by 
\begin{eqnarray*}
{\mathcal Y}(M)={-4{\pi}}\sqrt{2\sum^n_{m=1}c^2_{1}(X_{m})}.
\end{eqnarray*}
\end{main}
\begin{proof}
First of all, notice that we have 
\begin{eqnarray}\label{n-vani}
{\mathcal I}_{s}(N)=0 
\end{eqnarray}
by the assumption that $N$ admits a metric of non-negative scalar curvature and (\ref{vani-scal}). Moreover, LeBrun \cite{leb-44, leb-1} showed that, for any minimal compact K{\"{a}}hler surface $X$ with ${b}^{+}(X)>1$, the following holds:
\begin{eqnarray*}
{\mathcal I}_{s}(X)= |{\mathcal Y}(X)|^2 = |{\mathcal K}(X)|^2 = {32}{\pi}^{2}{c}^2_{1}(X)
\end{eqnarray*}
This fact with the bounds (\ref{n-vani}) and (\ref{bound-conn}) implies that
\begin{eqnarray*}
{\mathcal I}_{s}(M)=|{\mathcal Y}(M)|^2 = |{\mathcal K}(M)|^2 \leq {32}{\pi}^{2}\sum^{n}_{m=1}{c}^2_{1}(X_{m}). 
\end{eqnarray*}
Proposition \ref{prop-scal} with this bound tells us that the desired equality holds as promised. 
\end{proof}

On the other hand, instead of scalar curvature, it is so natural to consider the following Ricci curvature version of (\ref{scalar-yama-def}): 
\begin{eqnarray}\label{ricci-inv}
{\mathcal I}_{r}(X):=\inf_{g \in {\mathcal R}_{X}} {\int}_{X}|r_{g}|^{n/2}d{\mu}_{g}. 
\end{eqnarray}
Here $r_{g}$ is again the Ricci curvature of $g$. It is known \cite{leb-11} that there is the following relation between (\ref{scalar-yama-def}) and (\ref{ricci-inv}):
\begin{eqnarray}\label{ricci-scalar}
{\mathcal I}_{r}(X) \geq {n}^{-n/4}{\mathcal I}_{s}(X), 
\end{eqnarray}
and that equality holds if the Yamabe invariant is both non-positive and realized by an Einstein metric. The failure of the equality gives a quantitative obstruction to Yamabe's program for finding Einstein metrics. Therefore, it is quite interesting to investigate when the above inequality (\ref{ricci-scalar}) becomes strict. \par
By using Theorem \ref{bf-ricci} and the same method with the proof of Theorem C in \cite{ishi-leb-2}, we are able to obtain the following interesting result: 
\begin{main}\label{compu-Ricci-invariant}
Let $N$ be a closed oriented smooth 4-manifold with anti-self-dual metric of positive scalar curvature. For $m=1,2,3$, let $X_m$ be a minimal K{\"{a}}hler surface as in Theorem \ref{compu-scalar-invariant}.  Then, for any $n=2,3$, a connected sum $M:=(\#^{n}_{m=1}{X}_{m} ) \# N$ satisfies
\begin{eqnarray}\label{ricci-com}
{\mathcal I}_{r}(M) = 8{\pi}^{2} \Big[4n-\Big( 2\chi(N)+3\tau(N) \Big)+\sum^{n}_{m=1}{c}^2_{1}(X_{m}) \Big]. 
\end{eqnarray}
\end{main}

We leave, as an exercise, the detail of the proof of Theorem \ref{compu-Ricci-invariant} for the interested reader. Use Theorem \ref{bf-ricci} and the strategy of the proof of Theorem C in \cite{ishi-leb-2}. We also notice that Theorem \ref{compu-Ricci-invariant} in the case where $b_{1}({X}_{m}) \not=0$ never follows from Theorem C in \cite{ishi-leb-2}. \par
The above hypotheses regarding $N$ and Proposition 1 in \cite{leb-topology} force that $b^{+}(N)=0$. Hence we have
\[
2\chi(N) + 3 \tau(N) = 4 - 4 b_1(N) + 5b^+(N) - b^- (N)=4 - 4b_1(N) - b^-(N) \leq 4.
\]
By this fact, (\ref{scalar-com}) and (\ref{ricci-com}), we are able to conclude that the strict inequality holds whenever $n=2,3$: 
\begin{eqnarray*}
{\mathcal I}_{r}(M)> \frac{1}{4}{\mathcal I}_{s}(M). 
\end{eqnarray*}
Hence, the Yamabe sup-inf on this connected sums never realized by an Einstein metric. \par 
On the other hand, since the connected sum $k \overline{{\mathbb C}{P}}^2 \# {\ell} (S^{1} \times S^{3})$ admits anti-self-dual metrics of positive scalar curvature \cite{kim, leb-self}, Theorem \ref{compu-Ricci-invariant} particularly implies

\begin{cor}
Let $X_m$ be a minimal K{\"{a}}hler surface as in Theorem \ref{compu-scalar-invariant}.  Then, for any $n=2,3$, and any integers $k, \ell \geq 0$, 
\begin{eqnarray*}
{\mathcal I}_{r}\Big((\#^{n}_{m=1}{X}_{m}) \# k \overline{{\mathbb C}{P}}^2 \# {\ell} (S^{1} \times S^{3})  \Big) = 8{\pi}^{2} \Big[k + 4(n + \ell -1)+\sum^{n}_{m=1}{c}^2_{1}(X_{m}) \Big]. 
\end{eqnarray*}

\end{cor}

\subsection{Invariant arising from a variant of Perelman's ${\mathcal F}$-functional}\label{sub-43}

The main results of this subsection are Theorem \ref{cor-perel-yama} and Theorem \ref{compu-scalar-pre-invariant} below. Theorem \ref{compu-scalar-pre-invariant} is nothing but Theorem \ref{main-CCC} stated in Introduction. \par
Let us start with recalling the definition of Pelerman's ${\mathcal F}$-functional \cite{p-1, p-2, lott}. Let $X$ be a closed oriented Riemannian manifold of dimension $n \geq 3$ and $g$ be any Riemannian metric on $X$. We shall denote the space of all Riemannian metrics on $X$ by ${\cal R}_{X}$ and the space of all $C^{\infty}$ functions on $X$ by $C^{\infty}(X)$. Then, the ${\mathcal F}$-functional which was introduced by Perelman \cite{p-1} is the following functional ${\mathcal F} : {\cal R}_{X} \times C^{\infty}(X) \rightarrow {\mathbb R}$ defined by
\begin{eqnarray}\label{f-functional}
{\cal F}(g, f):={\int}_{X}({s}_{g} + |{\nabla}f|^{2}){e}^{-f} d\mu_{g}, 
\end{eqnarray}
where $f \in C^{\infty}(X)$, ${s}_{g}$ is the scalar curvature and $d\mu_{g}$ is the volume measure with respect to $g$. One of the fundamental discovery of Perelman is that the Ricci flow can be viewed as the gradient flow of ${\cal F}$-functional. Moreover the ${\cal F}$-functional is nondecreasing under the following coupled version of the Ricci flow:
\begin{eqnarray}\label{c-Ricci}
\frac{\partial}{\partial t}{g}=-2Ric_{g}, \ \frac{\partial}{\partial t}f=-\Delta f - s + |\nabla f|^2, 
\end{eqnarray}
where $Ric_{g}$ is the Ricci curvature and $s$ is the scalar curvature of the evaluating metric. It is then known that, for a given metric $g$, there exists a unique minimizer of the ${\cal F}$-functional under the constraint ${\int}_{X}{e}^{-f} d\mu_{g} =1$. Hence it is so natural to consider the following functional ${{\lambda}} : {\cal R}_{X} \rightarrow {\mathbb R}$ which is so called Perelman $\lambda$-functional: 
\begin{eqnarray*}
{\lambda}(g):=\inf_{f} \ \{ {\cal F}(g, f) \ | \ {\int}_{X}{e}^{-f} d\mu_{g} =1 \}. 
\end{eqnarray*}
It turns out that ${\lambda}(g)$ is nothing but the least eigenvalue of the elliptic operator $4 \Delta_g+s_g$, where $\Delta = d^*d= - \nabla\cdot\nabla $ is the positive-spectrum  Laplace-Beltrami operator associated with $g$. The nondecreasing of the ${\cal F}$-functional implies the nondecreasing of $\lambda$-functional. This also has a fundamental importance. In fact, Perelman used this fact to prove the non-existence of non-trivial steady and expanding Ricci breathers. Now, following Perelman, let us consider the scale-invariant quantity $\lambda(g) (vol_g)^{2/n}$, where $vol_g:={\int}_{X}d\mu_{g}$. Then let us recall
\begin{defn}[\cite{p-1, p-2, lott}]\label{pere-inv}
Perelman's $\bar{\lambda}$ invariant of $X$ is defined to be 
\begin{eqnarray*}\label{p-inv}
\bar{\lambda}(X)= \sup_{g \in {\cal R}_{X}} \lambda(g) (vol_g)^{2/n}.  
\end{eqnarray*}
\end{defn}
It turns our that, for any $X$ which dose not admit positive scalar curvature metric, $\bar{\lambda}(X)={\mathcal Y}(X)$ always holds \cite{A-ishi-leb-3}, where ${\mathcal Y}(X)$ is the Yamabe invariant of $X$. \par
In this subsection, inspired by recent interesting works of Cao \cite{cao-X} and Li \cite{li}, we would like to introduce one parameter family $\bar{\lambda}_{k}$ of smooth invariants, where  $k \in {\mathbb R}$. We shall call this invariant $\bar{\lambda}_{k}$ invariant. In particular, $\bar{\lambda}_{k}$ invariant includes Perelman's $\bar{\lambda}$ invariant as a special case. Indeed, $\bar{\lambda}_{1}=\bar{\lambda}$ holds as we shall see below. \par
We shall start with introducing the following definition which is essentially due to Li \cite{li}. The definition in the case where $k \geq 1$ is nothing but Definition 41 in \cite{li}. We notice that the following definition was also appeared as the equality (8) in \cite {o-s-w}: 
\begin{defn}[\cite{li, o-s-w}]
Let $X$ be a closed oriented Riemannian manifold with dimension $\geq 3$. Then, we define the following variant ${\mathcal F}_{k} : {\cal R}_{X} \times C^{\infty}(X) \rightarrow {\mathbb R}$ of the Perelman's $\mathcal F$-functional: 
\begin{eqnarray}\label{li-pere}
{\mathcal F}_{k}(g, f):={\int}_{X}\Big(k{s}_{g}+|\nabla f|^2 \Big){e}^{-f} d{\mu}_{g},
\end{eqnarray}
where $k$ is a real number $k \in {\mathbb R}$. We shall call this ${\mathcal F}_{k}$-functional.
\end{defn}
Notice that ${\mathcal F}_{1}$-functional is nothing but Perelman's ${\mathcal F}$-functional (\ref{f-functional}). Li \cite{li} showed that all functionals ${\mathcal F}_{k}$ with $k \geq 1$ have the monotonicity properties under the coupled system (\ref{c-Ricci}). 

\begin{rem}
It is not clear, at least for the present authors, that if these ${\mathcal F}_{k}$-functional have the monotonicity properties under the coupled system (\ref{c-Ricci}) in the case where $k<1$. In fact, the proof of Li \cite{li} breaks down in the case where $k<1$. See the proof of Theorem 42 in \cite{li}. 
\end{rem}

As was already mentioned in \cite{li, lott} essentially, for a given metric $g$ and $k \in {\mathbb R}$, there exists a unique minimizer of the ${\cal F}_{k}$-functional under the constraint ${\int}_{X}{e}^{-f} d\mu_{g} =1$. In fact, by using a direct method of the elliptic regularity theory \cite{g-t}, one can see that the following infimum is always attained: 
\begin{eqnarray*}
{{\lambda}}(g)_{k}:=\inf_{f} \ \{ {\cal F}_{k}(g, f) \ | \ {\int}_{X}{e}^{-f} d\mu_{g} =1 \}. 
\end{eqnarray*}
Notice that $\lambda(g)_k$ is nothing but the least eigenvalue of the elliptic operator $4 \Delta_g+ks_g$. It is then natural to introduce the following quantity: 
\begin{defn}
For any real number $k \in {\mathbb R}$, the $\bar{\lambda}_{k}$ invariant of $X$ is defined to be 
\begin{eqnarray*}\label{p-inv}
\bar{\lambda}_{k}(X)= \sup_{g \in {\cal R}_{X}}\lambda(g)_{k} (vol_g)^{2/n}.  
\end{eqnarray*}
\end{defn}
It is clear that $\bar{\lambda}_{1}=\bar{\lambda}$ holds. The $\bar{\lambda}_{k}$ invariant is also closely related to the Yamabe invariant. Indeed, we shall prove the following result which can be seen as a generalization of Theorem A proved in \cite{A-ishi-leb-3}: 
\begin{prop}\label{lambda-k-inv}
Suppose that $X$ is a smooth closed $n$-manifold, $n \geq 3$. Then the following holds: 
$$\bar{\lambda}_{k}(X) = \begin{cases}
     k{\mathcal Y}(X) & \text{ if  } {\mathcal Y}(X) \leq 0 \text{ and } k \geq \frac{n-2}{n-1}, \\
     +\infty  & \text{ if  } {\mathcal Y}(X) >  0 \text{ and } k > 0.
\end{cases}
$$
\end{prop}
Let us include the proof of Proposition \ref{lambda-k-inv} for completeness and for the reader's convenience. \par
Suppose now that $X$ is a closed oriented Riemannian manifold of dimension $n \geq 3$, and moreover that $\gamma:=[g]=\{ ug ~|~u: X \to {\Bbb R}^+\}$ is the conformal class of an arbitrary metric $g$. As was already mentioned, Trudinger, Aubin, and Schoen \cite{aubyam,lp,rick,trud} proved every conformal class on $X$ contains a Riemannian metric of constant scalar curvature. Such a metric $\hat{g}$ can be constructed by minimizing the Einstein-Hilbert functional:
$$
\hat{g}\mapsto  \frac{\int_X 
s_{\hat{g}}~d\mu_{\hat{g}}}{\left(\int_X 
d\mu_{\hat{g}}\right)^{\frac{n-2}{n}}},
$$
among all metrics conformal to $g$. Notice that, by setting $\hat{g} = u^{4/(n-2)}g$, the following identity holds:
\begin{eqnarray*}
\frac{\int_X 
s_{\hat{g}}~d\mu_{\hat{g}}}{\left(\int_X 
d\mu_{\hat{g}}\right)^{\frac{n-2}{n}}}= 
\frac{\int_X\left[ s_gu^2 +
4 \frac{n-1}{n-2}|\nabla u|^2\right] d\mu_g}{\left(\int_X  u^{2n/(n-2)}d\mu_g\right)^{(n-2)/n}}. 
\end{eqnarray*}
Associated to each conformal class $\gamma:=[g]$, we are also able to define the  Yamabe constant of the conformal class $\gamma$ in the following way:
\begin{eqnarray}\label{yama-def-0}
Y_{\gamma} = \inf_{u \in {C}^{\infty}_{+}(X)}\frac{\int_X\left[ s_gu^2 +
4 \frac{n-1}{n-2}|\nabla u|^2\right] d\mu_g}{\left(\int_X  u^{2n/(n-2)}d\mu_g\right)^{(n-2)/n}}, 
\end{eqnarray}
where ${C}^{\infty}_{+}(X)$ is the set of all positive functions $u: X \to {\Bbb R}^+$. Trudinger-Aubin-Schoen theorem teaches us that this number is actually realized as the constant scalar curvature of some unit-volume metric in each conformal class $\gamma$. A constant-scalar-curvature metric of this type is called a Yamabe minimizer. Again, the Yamabe invariant \cite{kob, sch} of $X$ is then given by 
\begin{eqnarray}\label{yama-def-1}
{\mathcal Y}(X) = \sup_{\gamma \in \mathcal{C}} Y_{\gamma}, 
\end{eqnarray}
where $\mathcal{C}$ is the set of all conformal classes on $X$. \par 
We are now in a position to prove the following lemma. We shall use the following to prove Proposition \ref{lambda-k-inv}: 
\begin{lem} \label{yupyup} 
Suppose that $\gamma$ is a conformal class on a closed oriented Riemannian manifold $X$ of dimension $n \geq 3$, which does not contain a metric of positive scalar curvature, i.e., ${Y}_\gamma \leq 0$. Then 
\begin{eqnarray}
Y_\gamma = \frac{1}{k} \Big(\sup_{g\in \gamma} \lambda(g)_k (vol_{g})^{2/n} \Big), 
\end{eqnarray}
where $k$ is a real number satisfying $k \geq \frac{n-2}{n-1}$. 
\end{lem}
\begin{proof}
Let $g\in \gamma$, and let $\hat{g}= u^{4/(n-2)}g$  be the Yamabe minimizer
in $\gamma$. By (\ref{yama-def-0}) and the hypothesis that $Y_\gamma \leq 0$, we have 
\begin{eqnarray*}
0 \geq Y_\gamma = \frac{\int_X\left[ s_gu^2 +
4 \frac{n-1}{n-2}|\nabla u|^2\right] d\mu_g}{\left(\int_X  u^{2n/(n-2)}d\mu_g\right)^{(n-2)/n}}.
\end{eqnarray*}
Namely, 
\begin{eqnarray}\label{43-1}
0 \geq {\int_X\left[ s_gu^2 + 4 \frac{n-1}{n-2}|\nabla u|^2\right] d\mu_g} = Y_\gamma {\left(\int_X  u^{2n/(n-2)}d\mu_g\right)^{(n-2)/n}}. 
\end{eqnarray}
On the other hand, the eigenvalue $\lambda(g)_k$ can be expressed in terms of Raleigh quotient as
\begin{eqnarray*}
\lambda(g)_k = \inf_{{u \in {C}^{\infty}_{+}(X)}} \frac{\int_{X} \left[k s_{g}u^2 + 4|\nabla u|^2 \right]d\mu_{g}}{\int_{X} u^{2}d\mu_{g}}. 
\end{eqnarray*}
Thus 
\begin{eqnarray*}
\lambda(g)_{k} \int_{X} u^2 d\mu_g &\leq& \int_{X} \left[k s_{g}u^2 + 4|\nabla u|^2 \right]d\mu_g = k \Big(\int_{X} \left[s_{g}u^2 + 4\frac{1}{k}|\nabla u|^2 \right]d\mu_g \Big) \\
&\leq& k \Big( \int_{X} \left[ s_{g}u^2 + 4\frac{n-1}{n-2} |\nabla u|^2 \right]d\mu_{g} \Big), 
\end{eqnarray*}
where we used the hypothesis that $k \geq \frac{n-2}{n-1}$, i.e., $\frac{1}{k} \leq \frac{n-1}{n-2}$. This bound and (\ref{43-1}) tells us that
\begin{eqnarray*}
\lambda(g)_{k} \int_{X} u^2 d\mu_g &\leq& k Y_\gamma \left(\int  u^{2n/(n-2)}d\mu_{g} \right)^{(n-2)/n} \\
&\leq& k Y_\gamma (vol_{g})^{-2/n} \int u^2 d\mu 
\end{eqnarray*}
where notice that, since $Y_\gamma\leq 0$, the last step is an the application of the H\"older inequality
$$\int f_1f_2 ~d\mu \leq \left(\int |f_1|^pd\mu \right)^{1/p} \left(\int |f_2|^qd\mu \right)^{1/q}, 
~~~\frac{1}{p}+ \frac{1}{q}=1,$$
with $f_1=1$, $f_2=u^2$, $p= n/2$, and $q=n/(n-2)$. Moreover, equality holds precisely when $u$ is constant, namely,  precisely when $g$ has constant scalar curvature. Since we shows that 
\begin{eqnarray*} 
\frac{1}{k} \lambda(g)_{k} (vol_{g})^{2/n} \leq Y_\gamma
\end{eqnarray*}
for every $g\in \gamma$, and  since equality occurs if $g$ is the
Yamabe minimizer, it follows that 
\begin{eqnarray*}
Y_\gamma = \frac{1}{k} \Big(\sup_{g\in \gamma} \lambda(g)_k (vol_{g})^{2/n} \Big).
\end{eqnarray*} 
\end{proof} 

The proof of Lemma \ref{yupyup} tells us that, under ${\mathcal Y}(X) \leq 0$ and any real number $k \geq \frac{n-2}{n-1}$, each constant scalar curvature metric maximizes $\frac{1}{k}\lambda_k (vol)^{2/n}$ in its conformal class. Given any maximizing sequence $\hat{g_{i}}$ for $\frac{1}{k}\lambda_k (vol)^{2/n}$, we may construct a new maximizing sequence ${g_{i}}$ consisting of unit volume constant scalar curvature metrics by conformal rescaling. However, for any such sequence, the constant number $s_{g_{i}}$ is viewed either as $\{ {Y}_{[{g_{i}]}} \}$ or as $\{ \frac{1}{k}\lambda({g_{i}})_k (vol)^{2/n}_{{g_{i}}} \}$. Therefore, we are able to conclude that the suprema over the space of all Riemannian metrics of ${Y}_{[g]}$ and $\frac{1}{k}\lambda({g}_k (vol)^{2/n}_{g}$ must coincide, namely, ${\mathcal Y}(X) = \frac{1}{k}\bar{\lambda}_{k}(X)$ must holds in this case, i.e., $k{\mathcal Y}(X) = \bar{\lambda}_{k}(X)$. Therefore, it is enough to prove the following lemma in order to prove Proposition \ref{lambda-k-inv}: 
\begin{lem} \label{aha} 
If ${\mathcal Y}(X)>0$, then $\bar{\lambda}_{k}(X) = +\infty$ for any positive real number $k > 0$. 
\end{lem}
\begin{proof}
Given such a manifold $X$ with ${\mathcal Y}(X)>0$ and any smooth non-constant function $f: X\to \RR$,  Kobayashi \cite{kob} has shown that there exists a unit-volume  metric $g$ on $M$ with $s_{g}=f$. The claim of this lemma follows from this result of Kobayashi. First of all, for any sufficiently large positive constant $L$, take a smooth non-constant function $f: X\to \RR$ such that $\min_{x}f \geq L$. Then the above result of Kobayashi tells us that there is a metric $g$ on $M$ with $s_{g}=f$ and $vol_{g}=1$. Notice that $\min_{x}s_{g}=\min_{x}f \geq L$ holds. For this metric $g$, the eigenvalue $\lambda(g)_k$ can be expressed in terms of Raleigh quotient as
\begin{eqnarray}\label{Raleigh}
\lambda(g)_k = \inf_{{u \in {C}^{\infty}_{+}(X)}} \frac{\int_{X} \left[k s_{g}u^2 + 4|\nabla u|^2 \right]d\mu_{g}}{\int_{X} u^{2}d\mu_{g}}. 
\end{eqnarray}
On the other hand, we have 
\begin{eqnarray*}
\frac{{\int}_{X}\left[k {s}_{g}u^2+4|\nabla u|^2\right] d{\mu}_{g}}{{\int}_{X}u^2 d{\mu}_{g}} &\geq& \frac{{\int}_{X}k {s}_{g}u^2 d{\mu}_{g}}{{\int}_{X}u^2 d{\mu}_{g}} \geq \frac{{\int}_{X}k ({\min_{x}{s}_{g}})u^2 d{\mu}_{g}}{{\int}_{X}u^2 d{\mu}_{g}} \\
&=& \frac{k({\min_{x}{s}_{g})}{\int}_{X}u^2 d{\mu}_{g}}{{\int}_{X}u^2 d{\mu}_{g}}= k({\min_{x}{s}_{g}})=k({\min_{x}f}) \\
&\geq& kL. 
\end{eqnarray*}
This bound and (\ref{Raleigh}) imply that 
\begin{eqnarray*}
\lambda(g)_k \geq kL.
\end{eqnarray*}
Since $vol_{g}=1$, this bound tells us the following holds: 
\begin{eqnarray*}
\bar{\lambda}_{k}(X):= \sup_{g}{\lambda}(g)_{k}(vol_{g})^{n/2}  \geq  \sup_{g, vol_{g}=1}{\lambda}(g)_{k}(vol_{g})^{n/2}  \geq kL.
\end{eqnarray*}
Therefore, we obtain 
\begin{eqnarray*}
\bar{\lambda}_{k}(X) \geq kL.
\end{eqnarray*}
Thus, taking $L \rightarrow +\infty$, $\bar{\lambda}_{k}(X) = +\infty$ holds for any $k > 0$. 
\end{proof}

We therefore proved Proposition \ref{lambda-k-inv}. 

\begin{rem}
Notice that the Yamabe invariant of any closed smooth manifold is always finite. For example, it is known that ${\mathcal Y}({\mathbb C}{P}^2)=12{\pi}\sqrt{2}$ holds. On the other hand, Lemma \ref{aha} tells us that $\bar{\lambda}_{k}({\mathbb C}{P}^2) = +\infty$ holds for any $k > 0$. 
\end{rem}

In particular, (\ref{s-y-1}), (\ref{scalar-yama-3}) and Proposition \ref{lambda-k-inv} tell us that the following  result holds, where notice that any manifold $M$ which dose not admit any Riemannian metric of positive scalar curvature must satisfy ${\mathcal Y}(M) \leq 0$:
\begin{thm}\label{cor-perel-yama}
Let $X$ be a smooth compact $n$-manifold with $n \geq 3$ and assume that $X$ dose not admit any Riemannian metric of positive scalar curvature. Then, for any real number $k$ with $k \geq \frac{n-2}{n-1}$, the following holds:

\begin{eqnarray}\label{equi-scalar}
{\mathcal I}_{s}(M) = |{\mathcal Y}(M)|^{\frac{n}{2}} = |{\mathcal K}(M)|^{\frac{n}{2}} = \Big|\frac{\bar{\lambda}_{k}(M)}{k} \Big{|}^{\frac{n}{2}}. 
\end{eqnarray}
\end{thm}

Theorem \ref{compu-scalar-invariant} and Theorem \ref{cor-perel-yama} immediately imply Theorem \ref{main-CCC} which was already mentioned in Introduction. More precisely, we have 
\begin{thm}\label{compu-scalar-pre-invariant}
Let $N$ be a closed oriented smooth 4-manifold with $b^{+}(N)=0$ with a Riemannian metric of non-negative scalar curvature. For $m=1,2,3$, let $X_m$ be a minimal K{\"{a}}hler surface with ${b}^{+}(X_m)>1$ and satisfying
\begin{eqnarray*}
{b}^{+}(X_{m})-{b}_{1}(X_{m}) \equiv 3 \ (\bmod \ 4).
\end{eqnarray*}
Let $\Gamma_{X_{m}}$ be a spin${}^{c}$ structure on $X_m$ which is induced by the K{\"{a}}hler structure. Under Definition \ref{def-1}, moreover assume that the following condition holds for each $m$:
\begin{eqnarray*}
\frak{S}^{ij}(\Gamma_{X_{m}}) \equiv 0 \bmod 2 & \text{for all $i, j$}.
\end{eqnarray*}
 Then, for any $n=2,3$ and any real number $k \geq \frac{2}{3}$, a connected sum $M:=(\#^{n}_{m=1}{X}_{i} ) \# N$ satisfies
\begin{eqnarray*}
{\mathcal I}_{s}(M) = |{\mathcal Y}(M)|^2 = |{\mathcal K}(M)|^2 =\Big|\frac{\bar{\lambda}_{k}(M)}{k} \Big{|}^2 ={32}{\pi}^{2}\sum^{n}_{m=1}{c}^2_{1}(X_{m}). 
\end{eqnarray*}
In particular, $\overline{{\lambda}}_{k}$-invariant of $M$ is given by 
\begin{eqnarray*}
\bar{\lambda}_{k}(M)={-4k{\pi}}\sqrt{2\sum^n_{m=1}c^2_{1}(X_{m})}.
\end{eqnarray*}
\end{thm}

On the other hand, it is now well known that the value of the Yamabe invariant is sensitive to the choice of smooth structures of a four-manifold. By using Theorem \ref{cor-perel-yama} and a result in \cite{leb-44}, we are able to prove the following result. This result can be seen as a generalization of Theorem 5 in \cite{kot-2}: 
\begin{cor}\label{main-pere}
The number of distinct values that the following four invariants can take on the smooth structures in a fixed homeomorphism type of simply connected 4-manifolds $X$ is unbounded: 
\begin{itemize}
\item The Yamabe invariant ${\mathcal Y}(X)$, 
\item the invariant ${\mathcal I}_{s}(X)$ arising from the scalar curvature, 
\item the invariant ${\mathcal K}(X)$, 
\item the $\bar{{\lambda}}_{k}$-invariant $\bar{{\lambda}}_{k}(X)$ for a real number $k$ satisfying $k \geq \frac{2}{3}$. 
\end{itemize}
\end{cor}

\begin{proof}
By Theorem \ref{cor-perel-yama}, we have 
\begin{eqnarray}\label{equi-pere-yama}
{\mathcal I}_{s}(M) = |{\mathcal Y}(M)|^{2} = |{\mathcal K}(M)|^{2} = \Big|\frac{\bar{\lambda}_{k}(M)}{k} \Big{|}^{2}. 
\end{eqnarray}
Hence, in order to prove the statement of this corollary, it is enough to prove the claim for the Yamabe invariant. First of all, let us recall that LeBrun \cite{leb-44} proved that the Yamabe invariant of any minimal complex surface $M$ of general type satisfies
\begin{eqnarray}\label{comp-yama}
{\mathcal Y}(M)={\mathcal Y}(M  \# \ell \overline{ {\mathbb C}{P}^{2}})=-4{\pi}\sqrt{2 c^2_{1}(M)} < 0, 
\end{eqnarray}
where $\ell \geq 0$. In what follows, we shall use the method of the proof of Theorem 5 in \cite{kot-2}. For the reader's convenience, we shall reproduce the argument. Using the standard result \cite{persson} on the geography on minimal surface of general type, for every integer $n$, one can always find positive integers $\alpha$ and $\beta$ satisfying the property that all pairs of integers $(\alpha-j, \beta+j)$, where $1 \leq j \leq n$, are realized as pairs $({c}_{2}(X_{j}), c^2_{1}(X_{j}))$ of Chern numbers of some simply connected minimal complex surface $X_{j}$ of general type. Consider the $j$-times blowup $M_{j}$ of $X_{j}$, i.e., $M_{j} = X_{j} \# j \overline{ {\mathbb C}{P}^{2}}$. Then all these $M_{j}$, where $1 \leq j \leq n$, are simply connected, non-spin and have the same Chern numbers $(\alpha, \beta)$. Hence, Freedman \cite{freedman} tells us that they must be homeomorphic to each other. On the other hand, equality (\ref{comp-yama}) implies that $M_{j}$ have pairwise different Yamabe invariants, i.e.,
\begin{eqnarray*}
{\mathcal Y}(M_{j}) = {\mathcal Y}(X_{j}  \# \ell \overline{ {\mathbb C}{P}^{2}}) = {\mathcal Y}(X_{j})=-4{\pi}\sqrt{2 c^2_{1}(X_{j})} = -4{\pi}\sqrt{2(\beta + j)}.  
\end{eqnarray*}
By using this and (\ref{equi-pere-yama}), the desired result now follows.
\end{proof}

\begin{rem}
As Corollary \ref{main-pere} above, let us here remark that all the results in \cite{kot-2} for the Perelman's $\bar{\lambda}$ invariant also holds for the above four invariants, i.e., ${\mathcal I}_{s}(X), \ {\mathcal Y}(X), {\mathcal K}(X)$ and ${\overline{\lambda}_{k}(X)}$, where $k \geq \frac{2}{3}$, without serious change of the proof. We leave this for the interested reader. 
\end{rem}

\subsection{Einstein metrics, simplicial volumes, and smooth structures}\label{sub-44}

In this section, we shall prove Theorem \ref{main-CC} which was already mentioned in Introduction. \par
First of all, we shall show that Theorem \ref{mono-key-bounds} and the method of the proof of Theorem D in \cite{ishi-leb-2} give rise to a new obstruction to the existence of Einstein metrics on 4-manifolds. The following theorem includes interesting cases which cannot be derived from Theorem D in \cite{ishi-leb-2}: 
\begin{main}\label{einstein}
Let $N$ be a closed oriented smooth 4-manifold with $b^{+}(N)=0$. For $m =1,2,3$, let $X_m$ be a closed oriented almost complex 4-manifold with ${b}^{+}(X_m)>1$ and satisfying
\begin{eqnarray*}
{b}^{+}(X_{m})-{b}_{1}(X_{m}) \equiv 3 \ (\bmod \ 4).
\end{eqnarray*}
Let $\Gamma_{X_{m}}$ be a spin${}^{c}$ structure on $X_m$ which is induced by the almost complex structure and assume that $SW_{X_{m}}(\Gamma_{X_{m}}) \equiv 1 \ (\bmod \ 2)$. Under Definition \ref{def-1}, moreover assume that the following condition holds for each $m$:
\begin{eqnarray*}
\frak{S}^{ij}(\Gamma_{X_{m}}) \equiv 0 \bmod 2 & \text{for all $i, j$}.
\end{eqnarray*}
Then a connected sum $M:=(\#_{m=1}^{n}{X}_{m}) \# N$, where $n=2,3$, cannot admit any Einstein metric if the following holds:
\begin{eqnarray}\label{asm}
4{n}-\Big(2\chi(N)+3\tau(N) \Big) \geq \frac{1}{3}\sum_{m=1}^{n}\Big( 2\chi(X_{m})+3\tau(X_{m}) \Big). 
\end{eqnarray}
\end{main}

\begin{proof}
First of all, a direct computation tells us that
\begin{eqnarray}\label{u-222}
2\chi(M)+3\tau(M)=-4{n}+\Big( 2\chi(N)+3\tau(N) \Big)+\sum_{m=1}^{n} \Big( 2\chi(X_{m})+3\tau(X_{m}) \Big). 
\end{eqnarray}
On the other hand, notice that the condition that $b^{+}(N)=0$ forces that $2\chi(N)+3\tau(N)= 4 - 4 b_1(N) + 5b^+(N) - b^- (N)=4 - 4b_1(N) - b^-(N) \leq 4$. Hence we always have 
\begin{eqnarray}\label{nega-N}
-4{n}+(2\chi(N)+3\tau(N)) < 0
\end{eqnarray}
when $n=2,3$.  \par
Now assume that $\sum_{m=1}^{n}(2\chi(X_{m})+3\tau(X_{m})) \leq 0$. Then, by (\ref{u-222}) and (\ref{nega-N}), we have 
\begin{eqnarray}\label{hit-vio}
2\chi(M)+3\tau(M) < 0.  
\end{eqnarray}
Notice that any Einstein 4-manifold must satisfy the Hitchin-Thorpe inequality (\ref{ht-int}). Hence, in the case where $\sum_{m=1}^{n}(2\chi(X_{m})+3\tau(X_{m})) \leq 0$, we are able to conclude that $M$ cannot admit any Einstein metric by (\ref{hit-vio}). Let us remark that, (\ref{asm}) holds trivially in the case where $\sum_{m=1}^{n}(2\chi(X_{m})+3\tau(X_{m})) \leq 0$ because we have $4{n}-(2\chi(N)+3\tau(N)) > 0$. \par
By the above observation, we may assume that $\sum_{m=1}^{n}(2\chi(X_{m})+3\tau(X_{m})) > 0$ holds. In particular, Theorem \ref{main-A} or Theorem \ref{con-mono} tells us that the connected sum $M$ has non-zero monopole classes. \par
As was already noticed in \cite{leb-11, ishi-leb-2}, we have the following inequality for any Riemannian metric $g$ on $M$ (cf. Proposition 3.1 in \cite{leb-11}): 
\begin{eqnarray*}
\int_{M}\Big(2|W^{+}_{g}|^{2} + \frac{s^{2}_{g}}{24} \Big)d{\mu}_{g} \geq \frac{1}{27}{\int}_{M}\Big({s}_{g}-\sqrt{6}|W^{+}_{g}| \Big)^2 d{\mu}_{g}. 
\end{eqnarray*}
This fact, the existence of non-zero monopole classes on $M$ and the curvature bound (\ref{weyl-leb-sca-2}) in Theorem \ref{mono-key-bounds} imply the following bound for any Riemannian metric $g$ on $M$:
\begin{eqnarray}\label{u-1}
\frac{1}{4\pi^{2}}\int_{M}\Big(2|W^{+}_{g}|^{2} + \frac{s^{2}_{g}}{24} \Big)d{\mu}_{g} > \frac{2}{3}\sum_{m=1}^{n}\Big( 2\chi(X_{m})+3\tau(X_{m}) \Big), 
\end{eqnarray}
where notice that $M$ has non-zero monopole class and $M$ cannot admit any symplectic structure. This and Theorem \ref{beta-ine-key} force that the above inequality must be strict. \par
On the other hand, by the definition, any Einstein 4-manifold $(X, g)$ must satisfy $\stackrel {\circ}{r}_{g} \equiv 0$, where $\stackrel {\circ}{r}_{g}$ is again the trace-free part of the Ricci curvature $r_{g}$ of $g$. Therefore, the equality (\ref{ein-gauss}) implies
\begin{eqnarray*}
2\chi(X)+3\tau(X)=\frac{1}{4\pi^{2}}\int_{X}\Big(2|W^{+}_{g}|^{2} + \frac{s^{2}_{g}}{24} \Big)d{\mu}_{g}. 
\end{eqnarray*}
Suppose now that the connected sum $M$ admit an Einstein metric $g$. Then the left-hand side of the above inequality (\ref{u-1}) is nothing but $2\chi(M)+3\tau(M)$. By combining (\ref{u-1}) with (\ref{u-222}), we are able to obtain
\begin{eqnarray*}
4{n}-\Big( 2\chi(N)+3\tau(N) \Big) < \frac{1}{3}\sum_{m=1}^{n}\Big( 2\chi(X_{m})+3\tau(X_{m}) \Big). 
\end{eqnarray*}
By contraposition, we get the desired result.
\end{proof}

On the other hand, let us recall the definition of simplicial volume due to Gromov \cite{gromov}. Let $M$ be a closed manifold. We denote by ${C}_{*}(M):=\sum^{\infty}_{k=0}{C}_{k}(M)$ the real coefficient singular chain complex of $M$. A chain $c \in {C}_{k}(M)$ is a finite combination $\sum{r}_{i}{\sigma}_{i}$ of singular simplexes ${\sigma}_{i} : {\Delta}^k \rightarrow M$ with real coefficients ${r}_{i}$. We define the norm $|c|$ of $c$ by $|c| : = \sum|r_{i}| \geq 0$. If $[\alpha] \in H_{*}(M, {\mathbb R})$ is any homology class, then the norm $||\alpha||$ of $[\alpha]$ is define as 
\begin{eqnarray*}
||\alpha||:=\inf \{|\frak a| : [{\frak a}] \in H_{*}(M, {\mathbb R}), [\frak a]=[\alpha]\}, 
\end{eqnarray*}
where the infimum is taken over all cycles representing $\alpha$. 
Suppose that $M$ is moreover oriented. Then we have the fundamental class $[M] \in H_{n}(M, {\mathbb R})$ of $M$. We then define the {simplicial volume} of $M$ by $||M||$. It is known that any simply connected manifold $M$ satisfies $||M||=0$. For the product of compact oriented manifolds, Gromov pointed out (see p.10 of \cite{gromov}) that the simplicial volume is essentially multiplicative. Indeed, there are universal constants $c_{n}$ depending only on the dimension $n$ of the product $M_{1} \times M_{2}$ such that
\begin{eqnarray}\label{sim-upper}
c^{-1}_{n}||M_1|| \cdot ||M_2|| \leq ||M_{1} \times M_{2}|| \leq c_{n}||M_1|| \cdot ||M_2||
\end{eqnarray}
On the other hand, for the connected sum, we have the following formula (cf \cite{be}):
\begin{eqnarray}\label{sim-connec}
||M_{1} \# M_{2}|| = ||M_{1}|| + ||M_{2}||
\end{eqnarray}
We shall use the following to prove Theorem \ref{main-CC}:
\begin{lem}\label{simplicial-lem}
Let $X_{m}$ be a closed oriented simply connected 4-manifold and consider a connected sum:
\begin{eqnarray*}
M:=(\#_{m} X_{m}) \# k (\Sigma_{h} \times \Sigma_{g}) \# \ell_{1}({S}^{1} \times {S}^{3}) \# \ell_{2} \overline{{\mathbb C}{P}^{2}}, 
\end{eqnarray*}
where $g, h \geq 1$, $m, k \geq 1$ and $\ell_{1}, \ell_{2} \geq 0$. Then the simplicial volume of $M$ satisfies the following bound:
\begin{eqnarray}\label{simplicial-M}
16c^{-1}_{4}k(g-1)(h-1) \leq ||M|| \leq 16c_{4}k(g-1)(h-1),    
\end{eqnarray}
where $c_{4}$ is the positive universal constant depending only on the dimension of the product $\Sigma_{h} \times \Sigma_{g}$. On the other hand, we have
\begin{eqnarray*}
2\chi(M)+3\tau(M) &=& \Big(\sum_{m} 2\chi(X_{m})+3\tau(X_{m}) \Big)+4k(g-1)(h-1)\\
&-&4(m+k+\ell_{1})-{\ell}_{2}, \\
2\chi(M)-3\tau(M) &=& \Big( \sum_{m} 2\chi(X_{m})-3\tau(X_{m}) \Big)+4k(g-1)(h-1) \\
                  &-& 4(m+k+\ell_{1})+5{\ell}_{2}.
\end{eqnarray*}

\end{lem}
\begin{proof}
First of all, as was already noticed in \cite{gromov}, any closed surface $\Sigma_{g}$ of genus $g \geq 2$ satisfies 
\begin{eqnarray*}
||\Sigma_{g}|| = -2\chi(\Sigma_{g})=4(g-1).  
\end{eqnarray*}
The bounds (\ref{sim-upper}) and (\ref{sim-connec}) together imply the following bound on the simplicial volume of a connected sum $k (\Sigma_{g} \times \Sigma_{h} )$ of $k$-copies of the product $\Sigma_{g} \times \Sigma_{h}$: 
\begin{eqnarray*}
16c^{-1}_{4}k(g-1)(h-1) &=& kc^{-1}_{4} ||\Sigma_{g}|| \cdot ||\Sigma_{h}|| \leq ||k (\Sigma_{g} \times \Sigma_{h} ) || = k ||\Sigma_{g} \times \Sigma_{h} || \\
& \leq & k c_{4}||\Sigma_{g} || \cdot ||\Sigma_{h} || = 16c_{4}k(g-1)(h-1).  
\end{eqnarray*}
Now, consider the connected sum $M:=(\#_{m} X_{m}) \# k (\Sigma_{h} \times \Sigma_{g})\# \ell({S}^{1} \times {S}^{3})$. By the formula (\ref{sim-connec}), we are able to conclude that $||M||=k||\Sigma_{g} \times \Sigma_{h}||$ holds, where notice that $||\#_{m} X_{m}||=0$, $||{S}^{1} \times {S}^{3} ||=0$ and $||\overline{{\mathbb C}{P}^{2}}||=0$. Therefore we are able to obtain the following bound on the simplicial volume of $M$:
\begin{eqnarray*}
16c^{-1}_{4}k(g-1)(h-1) \leq ||M|| \leq 16c_{4}k(g-1)(h-1).  
\end{eqnarray*}
On the other hand, one can easily derive the formulas on $2\chi(M)+3\tau(M)$ and $2\chi(M)-3\tau(M)$ by simple direct computations, where note that $\tau(\Sigma_{g} \times \Sigma_{h})=0$. 
\end{proof}

On the other hand, as an interesting special case of Theorem \ref{einstein}, we obtain 
\begin{cor}\label{speical-ein}
For $m=1,2,3$, let $X_{m}$ be a simply connected symplectic 4-manifold with ${b}^{+}(X_{m}) \equiv 3 \ (\bmod \ 4)$. Consider a connected sum $M:=(\#_{m=1}^{n} X_{m}) \# k (\Sigma_{h} \times \Sigma_{g}) \# \ell_{1}({S}^{1} \times {S}^{3}) \# \ell_{2} \overline{{\mathbb C}{P}^{2}}$, where $n, k \geq 1$ satisfying $n+k \leq 3$, $\ell_{1}, \ell_{2} \geq 0$ and $g, h$ are odd integers $\geq 1$. Then $M$ cannot admit any Einstein metric if 
\begin{eqnarray*}
4(m+\ell_{1} + k) + \ell_{2} \geq \frac{1}{3}\Big( \sum_{m=1}^{n} 2\chi(X_{m})+3\tau(X_{m})+4k(1-h)(1-g) \Big).  
\end{eqnarray*}
\end{cor}
\begin{proof}
Use Corollary \ref{key-cor-1} and Theorem \ref{einstein}. 
\end{proof}

On the other hand, we need to recall a construction of a certain sequence of homotopy $K3$ surfaces. Let $Y_{0}$ be a Kummer surface with an elliptic fibration $Y_{0} \rightarrow {\mathbb C}{P}^{1}$. Let $Y_{\ell}$ be obtained from $Y_{0}$ by performing a logarithmic transformation of order $2 \ell + 1$ on a non-singular fiber of $Y_{0}$. It turns out that the $Y_{\ell}$ are simply connected spin manifolds with $b^{+}(Y_{\ell}) = 3$ and $b^{-}(Y_{\ell}) = 19$. By the Freedman classification \cite{freedman}, $Y_{\ell}$ is homeomorphic to $K3$ surface. However, $Y_{\ell}$ is a K{\"{a}}hler surface with $b^{+}(Y_{\ell}) > 1$ and hence a result of Witten \cite{w} tells us that $\pm {c}_{1}(Y_{\ell})$ are monopole classes for each $\ell$. We have $ {c}_{1}(Y_{\ell}) = 2{\ell}\frak{f}$, where $\frak{f}$ is Poincar{\'{e}} dual to the multiple fiber which is introduced by the logarithmic transformation. See also \cite{BPV}. \par
We are now in a position to prove

\begin{thm}\label{simplicial-Ein-inf}
There exist infinitely many closed topological spin 4-manifolds satisfying the following three properties: 
\begin{itemize}
\item Each 4-manifold $M$ has non-trivial simplicial volume, i.e., $||M|| \not=0$.
\item Each 4-manifold $M$ satisfies the strict Gromov-Hitchin-Thorpe inequality, i.e., 
\begin{eqnarray*}
2\chi(M) - 3|\tau(M)| > \frac{1}{81{\pi}^2}||M||. 
\end{eqnarray*}
 \item Each 4-manifold $M$ is admits infinitely many distinct smooth structures for which no compatible Einstein metric exists.
\end{itemize}
\end{thm}
\begin{proof}
First of all, take any pair $(m, n)$ of positive integers satisfying $4m+2n-1 \equiv 3 \ (\bmod \ 4)$, $m \geq 2$ and $n \geq 1$. For any pair $(g, h)$ of odd integers which are greater than and equal to $3$, if necessarily, by taking another pair $(m, n)$ of large positive integers satisfying $4m+2n-1 \equiv 3 \ (\bmod \ 4)$, we are always able to find at east one positive integer $\ell_{1}$ satisfying the following three inequalities
\begin{eqnarray}\label{ein-in-1}
2n+\Big(1-\frac{4c_{4}}{81{\pi}^2}\Big) (g-1)(h-1)-3 > \ell_{1}.
\end{eqnarray}
\begin{eqnarray}\label{ein-in-2}
2(n+12m)+\Big(1-\frac{4c_{4}}{81{\pi}^2}\Big) (g-1)(h-1)+21 > \ell_{1}.
\end{eqnarray}
\begin{eqnarray}\label{ein-in-3}
\ell_{1} \geq \frac{1}{3}\Big( 2n+(g-1)(h-1) \Big)-3, 
\end{eqnarray}
where $c_{4}$ is the universal constant appeared in Lemma \ref{simplicial-lem}. Notice that the inequality (\ref{ein-in-1}) implies the inequality (\ref{ein-in-2}), and note also that we have infinitely many choices of such pair $(m, n)$ and, hence, of $\ell_{1}$. \par
On the other hand, let us recall that Gompf \cite{gom} showed that, for arbitrary integers $\alpha \geq 2$ and $\beta \geq 0$, one can construct a simply connected symplectic spin 4-manifold $X_{\alpha, \beta}$ satisfying
\begin{eqnarray}\label{go-1}
\Big( \chi(X_{\alpha, \beta}), \tau(X_{\alpha, \beta}) \Big)= \Big( 24\alpha+4\beta, -16\alpha \Big)
\end{eqnarray}
 Notice also that this implies
\begin{eqnarray}\label{go-2}
b^+(X_{\alpha, \beta}) &=& 4\alpha+2\beta-1,
\end{eqnarray}
\begin{eqnarray}\label{go-3}
2\chi(X_{\alpha, \beta}) + 3\tau(X_{\alpha, \beta}) &=& 8\beta, 
\end{eqnarray}
\begin{eqnarray}\label{go-4}
2\chi(X_{\alpha, \beta}) - 3\tau(X_{\alpha, \beta}) &=& 8(12\alpha+\beta).
\end{eqnarray}
Now, as was already observed in the above, for any pair $(g, h)$ of odd integers which are greater than and equal to $3$, we can find infinitely many pairs $(m, n)$ satisfying $4m+2n-1 \equiv 3 \ (\bmod \ 4)$ and also can find at least one  positive integer $\ell_{1}$ satisfying inequalities $(\ref{ein-in-1})$, $(\ref{ein-in-2})$ and (\ref{ein-in-3}).  For each such five integers $(m, n, g, h, \ell_{1})$ and for each new integer $\ell \geq 0$, consider the following connected sum:
\begin{eqnarray*}
M(m, n, \ell, g, h, \ell_{1}):=X_{m, n} \# {Y}_{\ell} \# (\Sigma_{g} \times \Sigma_{h}) \# \ell_{1}({S}^{1} \times {S}^{3}), 
\end{eqnarray*}
where ${Y}_{\ell}$ is obtained from $Y_{0}$ by performing a logarithmic transformation of order $2 \ell + 1$ on a non-singular fiber of $Y_{0}$. Note that we have $b_{1}(X_{m, n})=0$, $b^+(X_{m, n})=4m+2n-1 \equiv 3 \ (\bmod \ 4)$, $b_{1}(Y_{\ell})=0$ and $b^+(Y_{\ell})=3$. Each of these smooth oriented smooth 4-manifolds is homeomorphic to the following spin 4-manifold: 
\begin{eqnarray}\label{homeo-spin}
X_{m, n} \# K3 \# (\Sigma_{g} \times \Sigma_{h}) \# \ell_{1}({S}^{1} \times {S}^{3}). 
\end{eqnarray}
For any fixed $(m, n, g, h, \ell_{1})$, the sequence $\{M(m, n, \ell, g, h, \ell_{1}) \ | \ \ell \in {\mathbb N} \}$ contains infinitely many distinct diffeotype. There are two essentially same way to see this. One can use the bandwidth argument developed in \cite{ishi-leb-1, ishi-leb-2} to see this. Alternatively, one can also see this more directly by using only the finiteness property (see Proposition \ref{mono}) of the set of monopole classes (cf. \cite{kot-g, kot-2}). Moreover, each of these smooth oriented smooth 4-manifolds cannot admit any Einstein metric as follows. First of all, notice that Corollary \ref{speical-ein} tells us that, for any fixed $(m, n, \ell, g, h, \ell_{1})$, each 4-manifold $M(m, n, \ell, g, h, \ell_{1})$ cannot admit any Einstein metric if 
\begin{eqnarray*}
4(2+\ell_{1} + 1)  \geq \frac{1}{3}\Big( 2\chi(X_{m,n})+3\tau(X_{m,n})+ 2\chi(Y_{\ell})+3\tau(Y_{\ell})+4(1-h)(1-g) \Big), 
\end{eqnarray*}
equivalently, 
\begin{eqnarray*}
\ell_{1}+ 3 \geq \frac{1}{12}\Big( 8n+4(1-h)(1-g) \Big), 
\end{eqnarray*}
where we used $ 2\chi(X_{m,n})+3\tau(X_{m,n})=8n$ (see (\ref{go-3})) and $2\chi(Y_{\ell})+3\tau(Y_{\ell})=0$. The last inequality is nothing but the inequality (\ref{ein-in-3}) above. Hence, for any fixed $(m, n, \ell, g, h, \ell_{1})$, each 4-manifold $M(m, n, \ell, g, h, \ell_{1})$ cannot admit any Einstein metric as desired. Hence each of topological spin manifolds (\ref{homeo-spin}) admits infinitely many distinct smooth structures for which no compatible Einstein metric exists. \par
In what follows, we shall prove that each $M$ of topological spin manifolds (\ref{homeo-spin}) has non-zero simplicial volume and satisfies the strict Gromov-Hitchin-Thorpe inequality. In fact, by the bound (\ref{simplicial-M}), we have the following bound on the simplicial volume of $M$:
\begin{eqnarray}\label{tel-1}
0 < \frac{16c^{-1}_{4}}{81{\pi}^2}(g-1)(h-1) \leq \frac{1}{81{\pi}^2}||M|| \leq \frac{16c_{4}}{81{\pi}^2}(g-1)(h-1).  
\end{eqnarray}
In particular, $||M|| \not=0$ holds. On the other hand, we have 
\begin{eqnarray}\label{euler-M-1}
2\chi(M)+3\tau(M) = 8n+4(g-1)(h-1)-4(3+\ell_{1}), 
\end{eqnarray}
\begin{eqnarray}\label{euler-M-2}
2\chi(M)-3\tau(M) = 8(12m+n)+96+4(g-1)(h-1)-4(3+\ell_{1}), 
\end{eqnarray}
where see the final formulas in Lemma \ref{simplicial-lem} and notice that, for a K3 surface, we have $2\chi+3\tau=0$ and $2\chi-3\tau=96$. \par
Now, by multiplying both sides of (\ref{ein-in-1}) by $4$, we have
\begin{eqnarray*}
8n+\Big(4-\frac{16c_{4}}{81{\pi}^2}\Big) (g-1)(h-1) > 4(\ell_{1}+3).
\end{eqnarray*}
Equivalently, 
\begin{eqnarray*}
8n+4(g-1)(h-1) - 4(\ell_{1}+3) >\frac{16c_{4}}{81{\pi}^2}(g-1)(h-1).
\end{eqnarray*}
This inequality, (\ref{tel-1}) and (\ref{euler-M-1}) imply
\begin{eqnarray*}
2\chi(M) + 3\tau(M) > \frac{1}{81{\pi}^2}||M||. 
\end{eqnarray*}
Similarly, by multiplying both sides of (\ref{ein-in-2}) by $4$, we get 
\begin{eqnarray*}
8(12m+n)+\Big(4-\frac{16c_{4}}{81{\pi}^2} \Big) (g-1)(h-1)+84 > 4\ell_{1}.
\end{eqnarray*}
Namely we have 
\begin{eqnarray*}
8(12m+n)+96+4(g-1)(h-1)-4(3+\ell_{1}) > \frac{16c_{4}}{81{\pi}^2}(g-1)(h-1).
\end{eqnarray*}
This inequality, (\ref{tel-1}) and (\ref{euler-M-2}) tells us that the following holds: 
\begin{eqnarray*}
2\chi(M) - 3\tau(M) > \frac{1}{81{\pi}^2}||M||. 
\end{eqnarray*}
Therefore, the spin 4-manifold $M$ satisfies the strict Gromov-Hitchin-Thorpe inequality as desired:
\begin{eqnarray*}
2\chi(M) - 3|\tau(M)| > \frac{1}{81{\pi}^2}||M||. 
\end{eqnarray*}
Hence, the spin 4-manifold $M$ has the desired properties. Because we have an infinitely many choice of the above integers $(g, h, m, n, \ell_{1})$, we are able to conclude that there exist infinitely many closed topological spin 4-manifolds with desired properties as promised. 
\end{proof}

Similarly, we have 
\begin{thm}\label{simplicial-Ein-inf-2}
There exist infinitely many closed topological non-spin 4-manifolds and each of these 4-manifolds satisfies the three properties in Theorem \ref{simplicial-Ein-inf}.
\end{thm}
\begin{proof}
The proof is similar to the spin case. In fact, instead of the above connected sum, consider the following connected sum 
\begin{eqnarray*}
X_{m, n} \# {Y}_{\ell} \# (\Sigma_{h} \times \Sigma_{g}) \# \ell_{2} \overline{{\mathbb C}{P}^{2}}. 
\end{eqnarray*}
Notice that these 4-manifolds are non-spin whenever $\ell_{2} \geq 1$. For completeness, let us include the proof of this theorem. As the case of spin, take again any pair $(m, n)$ of positive integers satisfying $4m+2n-1 \equiv 3 \ (\bmod \ 4)$, $m \geq 2$ and $n \geq 1$. For any pair $(g, h)$ of odd integers which are greater than and equal to $3$, if necessarily, by taking another pair $(m, n)$ of large positive integers satisfying $4m+2n-1 \equiv 3 \ (\bmod \ 4)$, we are always able to find at east one positive integer $\ell_{2}$ satisfying the following three inequalities
\begin{eqnarray}\label{ein-in-12}
8n+4\Big( 1-\frac{4c_{4}}{81{\pi}^2} \Big) (g-1)(h-1)-12 > \ell_{2}.
\end{eqnarray}
\begin{eqnarray}\label{ein-in-22}
8(n+12m)+4 \Big( 1-\frac{4c_{4}}{81{\pi}^2} \Big) (g-1)(h-1)+84 > -5\ell_{2}.
\end{eqnarray}
\begin{eqnarray}\label{ein-in-32}
\ell_{2} \geq \frac{1}{3}\Big(8n+4(g-1)(h-1)\Big)-12, 
\end{eqnarray}
where $c_{4}$ is again the universal constant appeared in Lemma \ref{simplicial-lem}. Notice that we have infinitely many choices of such pair $(m, n)$ and of $\ell_{2}$. \par
On the other hand, as was already used in the proof of Theorem \ref{simplicial-Ein-inf}, the construction of  Gompf \cite{gom} enables us to construct, for arbitrary integers $\alpha \geq 2$ and $\beta \geq 0$, a simply connected symplectic spin 4-manifold $X_{\alpha, \beta}$ satisfying (\ref{go-1}), (\ref{go-2}), (\ref{go-3}) and (\ref{go-4}). \par
As was already observed above, for any pair $(g, h)$ of odd integers which are greater than and equal to $3$, we are able to find infinitely many pairs $(m, n)$ satisfying $4m+2n-1 \equiv 3 \ (\bmod \ 4)$ and also can find at least one positive integer $\ell_{2}$ satisfying inequalities $(\ref{ein-in-12})$, $(\ref{ein-in-22})$ and (\ref{ein-in-32}). For each such five integers $(m, n, g, h, \ell_{2})$ and for each new integer $\ell \geq 0$, consider the following:
\begin{eqnarray*}
N(m, n, \ell, g, h, \ell_{2}):=X_{m, n} \# {Y}_{\ell} \# (\Sigma_{g} \times \Sigma_{h}) \# \ell_{2} \overline{{\mathbb C}{P}^{2}}, 
\end{eqnarray*}
where ${Y}_{\ell}$ is again obtained from $Y_{0}$ by performing a logarithmic transformation of order $2 \ell + 1$ on a non-singular fiber of $Y_{0}$. We have $b_{1}(X_{m, n})=0$, $b^+(X_{m, n})=4m+2n-1 \equiv 3 \ (\bmod \ 4)$, $b_{1}(Y_{\ell})=0$ and $b^+(Y_{\ell})=3$. Each of these smooth oriented smooth 4-manifolds is homeomorphic to the following non-spin 4-manifold: 
\begin{eqnarray}\label{homeo-non-spin}
X_{m, n} \# K3 \# (\Sigma_{g} \times \Sigma_{h}) \# \ell_{2}\overline{{\mathbb C}{P}^{2}}. 
\end{eqnarray}
For any fixed $(m, n, g, h, \ell_{1})$, again, the sequence $\{N(m, n, \ell, g, h, \ell_{2}) \ | \ \ell \in {\mathbb N} \}$ contains infinitely many distinct diffeotype. Moreover, we can also see that each of these smooth oriented smooth 4-manifolds cannot admit any Einstein metric. In fact, Corollary \ref{speical-ein} tells us that, for any fixed $(m, n, \ell, g, h, \ell_{2})$, each 4-manifold $N(m, n, \ell, g, h, \ell_{2})$ cannot admit any Einstein metric if 
\begin{eqnarray*}
4(2+ 0+ 1) + {\ell}_{2}  \geq \frac{1}{3}\Big( 2\chi(X_{m,n})+3\tau(X_{m,n})+ 2\chi(Y_{\ell})+3\tau(Y_{\ell})+4(1-h)(1-g) \Big), 
\end{eqnarray*}
equivalently, 
\begin{eqnarray*}
\ell_{1}+ 12 \geq \frac{1}{12}\Big( 8n+4(1-h)(1-g) \Big), 
\end{eqnarray*}
where notice again that we have $ 2\chi(X_{m,n})+3\tau(X_{m,n})=8n$ and $2\chi(Y_{\ell})+3\tau(Y_{\ell})=0$. This inequality is nothing but the inequality (\ref{ein-in-32}) above. Hence, for any fixed $(m, n, \ell, g, h, \ell_{2})$, each 4-manifold $N(m, n, \ell, g, h, \ell_{2})$ cannot admit any Einstein metric as desired. Hence each of topological non-spin manifolds (\ref{homeo-non-spin}) admits infinitely many distinct smooth structures for which no compatible Einstein metric exists. \par
Finally, we shall prove that each $N$ of topological spin manifolds (\ref{homeo-non-spin}) has non-zero simplicial volume and satisfies the strict Gromov-Hitchin-Thorpe inequality. As the case of spin, we have 
\begin{eqnarray}\label{tel-12}
0 < \frac{16c^{-1}_{4}}{81{\pi}^2}(g-1)(h-1) \leq \frac{1}{81{\pi}^2}||N|| \leq \frac{16c_{4}}{81{\pi}^2}(g-1)(h-1).  
\end{eqnarray}
In particular, $||N|| \not=0$ holds. On the other hand, we get
\begin{eqnarray}\label{euler-M-12}
2\chi(N)+3\tau(N) = 8n+4(g-1)(h-1)-12 -{\ell}_{2}, 
\end{eqnarray}
\begin{eqnarray}\label{euler-M-22}
2\chi(M)-3\tau(M) = 8(12m+n)+84+4(g-1)(h-1)+5\ell_{2}, 
\end{eqnarray}
where we used the final formulas in Lemma \ref{simplicial-lem}. \par
Now, inequality (\ref{ein-in-12}) is equivalent to 
\begin{eqnarray*}
8n+4(g-1)(h-1)-12-{\ell}_{2} > \frac{16c_{4}}{81{\pi}^2}(g-1)(h-1).
\end{eqnarray*}
This inequality, (\ref{tel-12}) and (\ref{euler-M-12}) imply
\begin{eqnarray*}
2\chi(N) + 3\tau(N) > \frac{1}{81{\pi}^2}||N||. 
\end{eqnarray*}
Similarly, inequality (\ref{ein-in-22}) can be rewritten as 
\begin{eqnarray*}
8(12m+n)+4(g-1)(h-1)+84 + 5{\ell}_{2} > \frac{16c_{4}}{81{\pi}^2} (g-1)(h-1).
\end{eqnarray*}
This inequality, (\ref{tel-12}) and (\ref{euler-M-22}) implies
\begin{eqnarray*}
2\chi(N) - 3\tau(N) > \frac{1}{81{\pi}^2}||N||. 
\end{eqnarray*}
Thus, the non-spin 4-manifold $N$ satisfies the strict Gromov-Hitchin-Thorpe inequality as desired:
\begin{eqnarray*}
2\chi(N) - 3|\tau(N)| > \frac{1}{81{\pi}^2}||N||. 
\end{eqnarray*}
Therefore, the non-spin 4-manifold $N$ has the desired properties. Since we have an infinitely many choice of the above integers $(g, h, m, n, \ell_{2})$, we are able to conclude that there exist infinitely many closed topological non-spin 4-manifolds with desired properties.  
\end{proof}

It is now clear that Theorem \ref{main-CC} mentioned in Introduction follows from Theorems \ref{simplicial-Ein-inf} and \ref{simplicial-Ein-inf-2}. 

\begin{rem}\label{final-rem}
It is known that the right-hand side of Gromov-Hitchin-Thorpe inequality (\ref{k-GHT}) can be replaced by the volume entropy (or asymptotic volume) \cite{kot-g}. For the reader's convenience, let us here recall briefly the definition of the volume entropy (or asymptotic volume) of a Riemannian manifold. Let $X$ be a closed oriented Riemannian manifold with smooth metric $g$, and let $\Tilde{M}$ be its universal cover with the induced metric $\Tilde{g}$. For each $\Tilde{x} \in \Tilde{M}$, let $V(\Tilde{x}, R)$ be the volume of the ball with the center $\Tilde{x}$ and radius $R$. We set 
\begin{eqnarray*}
{\lambda}(X, g):=\lim_{R \rightarrow +\infty}\frac{1}{R}\log V(\Tilde{x}, R). 
\end{eqnarray*}
Thanks to work of Manning \cite{Manning}, it turns out that this limit exists and is independent of the choice of $\Tilde{x}$. We call ${\lambda}(X, g)$ the volume entropy of the metric $g$ and define the volume entropy of $X$ to be
\begin{eqnarray*}
{\lambda}(X):=\inf_{g \in {\cal R}^{1}_{X}}{\lambda}(X, g), 
\end{eqnarray*}
where ${\cal R}^{1}_{X}$ means the set of all Riemannian metrics $g$ with $V(X, g)=1$. It is known that the volume entropy can only be positive for manifolds with fundamental groups of exponential growth \cite{mil}. Now, it is known that any closed Einstein 4-manifold $X$ must satisfy the following bound \cite{kot-g}:
\begin{eqnarray}\label{gro-vol}
2\chi(X) - 3|\tau(X)| \geq \frac{1}{54{\pi}^2}{\lambda}(X)^4, 
\end{eqnarray}
where equality can occur if and only if every Einstein metric on $X$ is flat, is a non-flat Calabi-Yau metric, or is a metric of constant negative sectional curvature. The inequality (\ref{gro-vol}) is stronger than the inequality (\ref{k-GHT}). Notice that the connected sums appeared in Theorems \ref{simplicial-Ein-inf} and \ref{simplicial-Ein-inf-2} have non-zero volume entropy because we have the following inequality \cite{gromov, p-p} which holds for any closed manifold $X$ of dimension $n$:
\begin{eqnarray*}
\frac{1}{c_{n} n!} || X || \leq [{\lambda}(X)]^n, 
\end{eqnarray*}
where $c_{n}$ is the universal constant depends only on $n$. It is natural to ask if Theorem \ref{main-CC} still holds for the inequality (\ref{gro-vol}). To prove such a result, we need to compute the value of the volume entropy of connected sums or to give an upper bound like Lemma \ref{simplicial-lem} above. To the best of our knowledge, there is no literature which discusses such a subject. If we could prove such a result, Theorem \ref{main-CC} can be easily generalized to the case of volume entropy. In the present article, we could not pursue this point.  We hope to return this issue in near future. 
\end{rem}

\section{Concluding remarks}\label{final}

In the present article, we have proved a new non-vanishing theorem of stable cohomotopy Seiberg-Witten invariants and have given several applications of the non-vanishing theorem to geometry of 4-manifolds. As was already mentioned in Introduction, our non-vanishing theorem, i.e.,Theorem \ref{main-A}, includes Bauer's non-vanishing theorem as a special case except the case of four connected sums. Moreover, we showed that Conjecture \ref{conj-1} in the case where $\ell = 2$ is true. Based on Theorem \ref{main-A}, Theorem \ref{main-B} and Corollary \ref{conje-cor}, it is so natural to propose the following including Conjecture \ref{conj-1} in the case where $\ell =3$ as a special case:
\begin{conj}\label{conj-2}
For $i=1,2,3,4$, let $X_i$ be are almost complex 4-manifolds with $b^+(X_i) > 1$, ${b}_{1}(X_{i}) \not=0$, $SW_{X}(\Gamma_{X}) \equiv 1 \ (\bmod \ 2)$, and satisfying both conditions (\ref{spin-0}) and (\ref{spin-014}). Then a connected sum $X:=\#^{4}_{i=1}{X}_{i}$ has a non-trivial stable cohomotopy Seiberg-Witten invariant. 
\end{conj}
Notice that, if one drop the condition that ${b}_{1}(X_{i}) \not=0$ in the above, the claim is not true because the conclusion contradicts Bauer's non-vanishing theorem. Unfortunately, the method developed in the present article cannot be used to explore the above conjecture because the spin cobordism Seiberg-Witten invariant must vanish for the connected sum $X:=\#^{4}_{i=1}{X}_{i}$. See Remark \ref{remark almost} above. Hence, in order to attack the above conjecture, we need to develop a completely new method.  We hope to return this issue in some day. It is also interesting to ask, in the case where ${b}_{1}(X_{i}) \not=0$, if there is an integer $n \geq 5$ such that the connected sum $\#^{n}_{i=1}{X}_{i}$ has a non-trivial stable cohomotopy Seiberg-Witten invariant. This is also completely open. We remark that, for any closed 4-manifold $X$ with $b^+(X) \geq 1$ and ${b}_{1}(X)=0$, there is some large integer $N$ such that, for any $n \geq N$, the $n$-fold connected sum of $X$ with itself, has a trivial stable cohomotopy Seiberg-Witten invariant. See \cite{furuta-k-m-0} for more detail. \par 
On the other hand, it is  not so easy,  at least for the present authors, to find examples of almost complex 4-manifold with $b^+ \geq 2$, $b_{1} \not=0$, $SW_{X_{i}}(\Gamma_{X_{i}}) \equiv 1 \ (\bmod \ 2)$ and satisfying both (\ref{spin-0}) and (\ref{spin-014}), where $\Gamma_{X_{i}}$ is a spin${}^c$ structure compatible with the almost complex structure. In Theorem \ref{cor-1} above, we saw that the product $\Sigma_{g} \times \Sigma_{h}$ with $g$, $h$ odd and primary Kodaira surface are such examples. We have the following problem:
\begin{problem}
Find another example of almost complex 4-manifold $X$ with $b^+(X) \geq 2$, $b_{1}(X) \not=0$, $SW_{X}(\Gamma_{X}) \equiv 1 \ (\bmod \ 2)$ and satisfying both (\ref{spin-0}) and (\ref{spin-014}). 
\end{problem}
On the other hand, Okonek and Teleman \cite{o-t} introduces a new class of the stable cohomotopy Seiberg-Witten invariant, which has clear factorial properties with respect to diffeomorphism of 4-manifolds. In particular, for any closed 4-manifold with $b^{+}(X) \geq 2$ and $b_{1}(X)=0$, the invariant of Okonek and Teleman is equivalent to the stable cohomotopy Seiberg-Witten invariant $BF_{X}$ due to Bauer and Furuta \cite{b-f}. However, it seems that the invariant of Okonek and Teleman is finer than the stable cohomotopy Seiberg-Witten invariant in general. Among other things, in \cite{o-t}, Okonek and Teleman clarifies a relationship between the new invariant and a variant of the original Seiberg-Witten invariant $SW_{X}$, which is called full Seiberg-Witten invariant \cite{o-t-1}. On the other hand, in Theorem \ref{prop diagram} proved in subsection \ref{diag} above, we established a natural commutative diagram among  $BF_{X}$, $\widehat{SW}^{spin}_X$ and $SW_{X}$. In light of these results, there should exist an analogue of Theorem \ref{prop diagram} in the context of Okonek and Teleman \cite{o-t}. We also hope to return this issue in a further research. \par

\noindent
{} \par
\noindent
{\bf Acknowledgement:} We would like to express deep gratitude to Mikio Furuta for his warm encouragement. Furthermore, the first author would like to express deep gratitude to Claude LeBrun for his warm encouragement. Some of the main parts of this article were written during the first author's stay at State University of New York at Stony Brook. He would like to express many thanks to Claude LeBrun and the department of Mathematics at SUNY for their hospitality during his stay. The second author is partially supported by the 21th century COE program at Graduate School of Mathematical Sciences, the University of Tokyo.


\vfill

{\footnotesize 
\noindent
{Masashi Ishida}\\
{Department of Mathematics,  
Sophia University,\\ 7-1 Kioi-Cho, Chiyoda-Ku, 
 Tokyo 102-8554, Japan }\\
{\sc e-mail}: masashi@math.sunysb.edu \\ 
{\footnotesize

{\footnotesize 
\noindent
{Hirofumi Sasahira}\\
{Graduate School of Mathematical Science,  
the University of Tokyo,\\ 3-8-1 Komaba Meguro-Ku,  
 Tokyo 153-8941, Japan }\\
{\sc e-mail}: sasahira@ms.u-tokyo.ac.jp\\

\end{document}